\documentclass[12pt,a4paper]{article}
\usepackage[margin=1in]{geometry}
\usepackage{graphicx,xcolor,cite,mathtools}
\usepackage{amssymb,amsmath,amsthm,mathrsfs,paralist,bm,esint,setspace}
\usepackage{tikz,pgfplots}
\RequirePackage[colorlinks,citecolor=blue,urlcolor=blue]{hyperref}
\usepackage{comment}

\newcommand{\cB}{{\mathcal{B}}}

\newcommand{\bD}{{\mathbb{D}}}
\newcommand{\N}{{\mathbb{N}}}
\newcommand{\R}{{\mathbb{R}}}

\newcommand{\E}{\mathrm{E}}

\newcommand{\HH}{\mathcal{H}}
\newcommand{\LL}{\mathcal{L}}
\newcommand{\MM}{\mathcal{M}}

\newcommand{\XX}{\mathcal{X}}

\newcommand{\rT}{\text{T}}

\renewcommand{\P}{\mathrm{P}}
\renewcommand{\d}{\mathrm{d}}

\newcommand {\ep}{{\varepsilon}}
\newcommand {\deltah}{{\zeta}}

\newcommand {\half}{\frac 12}

\newcommand{\e}{\mathrm{e}}

\title{Sharp upper bounds on hitting probabilities for the solution to the stochastic heat equation on the line}
\author{Robert C.~Dalang\footnote{Institut de math\'ematiques,
Ecole Polytechnique F\'ed\'erale de Lausanne, Station 8,
CH-1015 Lausanne, Switzerland.
Email: robert.dalang@epfl.ch Partly supported by the Swiss National Foundation for Scientific Research}, David Nualart\footnote{Department of Mathematics, Kansas University, KS-66045, Lawrence, U.S.A. 
Email: nualart@ku.edu} 
and Fei Pu\footnote{Laboratory of Mathematics and Complex Systems,
School of Mathematical Sciences, Beijing Normal University, 100875, Beijing, China. Email: fei.pu@bnu.edu.cn Partly supported by National Natural Science Foundation of China (No. 12201047)}
}
\date{}                                           

\begin{document}
\newtheorem{stat}{Statement}[section]
\newtheorem{proposition}[stat]{Proposition}
\newtheorem*{prop}{Proposition}
\newtheorem{corollary}[stat]{Corollary}
\newtheorem{theorem}[stat]{Theorem}
\newtheorem{lemma}[stat]{Lemma}
\theoremstyle{definition}
\newtheorem{definition}[stat]{Definition}
\newtheorem*{cremark}{Remark}
\newtheorem{remark}[stat]{Remark}
\newtheorem*{OP}{Open Problem}
\newtheorem{example}[stat]{Example}
\newtheorem{nota}[stat]{Notation}
\numberwithin{equation}{section}
\maketitle

\begin{abstract}
    For Gaussian random fields with values in $\R^d$, sharp upper and lower bounds on the probability of hitting a fixed set have been available for many years. These apply in particular to the solutions of systems of linear SPDEs. For non-Gaussian random fields, the available bounds are less sharp. For nonlinear systems of stochastic heat equations, a sharp lower bound was obtained in a previous paper by two of the authors. Here, we obtain the corresponding sharp upper bound. The proof requires a bound on the joint probability density function of a two-dimensional random vector whose components are the solution to the {\em nonlinear} stochastic heat equation and the supremum over a small rectangle of the solution to the {\em linear} stochastic heat equation, in terms of  the size of the rectangle. This bound makes use of a formula that expresses the density of a {\em locally nondegenerate} random vector as an iterated Skorohod integral. The main effort is to estimate, using Malliavin calculus, each of the terms that arise from this formula.
 \end{abstract}

\section{Introduction}
 Let $d$ and $k$ be positive integers. Consider an $\R^{d}$-valued continuous random field $U=(U(x),\, x \in \R^{k})$. Fix a compact subset $I \subset \R^k$ with positive Lebesgue measure.  The {\em range of $U$} over $I$ (also termed the sample path $x \mapsto U(x)$, $x \in I$) is the random compact set
$$
   U(I) = \{U(x),\ x \in I\}.
$$
Hitting probabilities for $U$ refer to the probabilities of events of the form $\{U(I) \cap A \neq \emptyset\}$, $A \subset \R^d$. An interesting problem is to determine good upper and lower bounds on these probabilities, in terms of quantities related to the deterministic subset $A$. In particular, a point $z \in \R^d$ is said to be {\em polar} for $U$ if $\P\{\exists x \in I: U(x) = z\} = 0$.

Many results on this type of question are available for Brownian motion and classical Markov processes: see for instance \cite[Chapter 10]{khosh} and the references therein. For multiparameter processes, fewer results are available. For the $k$-parameter Brownian sheet $W = (W(x),\, x \in \R_+^k)$ with values in $\R^d$, Khoshnevisan and Shi \cite{khosh_shi1999} established the following result (see also \cite[Chapter 12]{khosh}).

\begin{theorem}\label{rd07_01t1}
Fix $M >0$ and let $I$ be a box.
There exists $0< c < \infty$ such that for all compact sets $A\subset [-M, M]^d$ $(\subset \R^d)$,
\begin{equation}\label{rd07_01e3}
   c^{-1}\, \mbox{\rm Cap}_{d-2k}(A) \leq \P\{W(I) \cap A \neq \emptyset \} \leq c\, \mbox{\rm Cap}_{d-2k}(A).
\end{equation}
\end{theorem}

In this theorem, $ \mbox{Cap}_{\beta}(A)$ denotes the {\em Bessel-Riesz capacity} of $A$ of index $\beta \in \R$ (see \cite[Chapter 10, Section 3, p.376]{khosh} for a definition). In particular, if $A = \{z\}$ is a single point in $\R^d$, then $\mbox{Cap}_{\beta}(\{z\}) = 0$ for $\beta\geq 0$ and $\mbox{Cap}_{\beta}(\{z\}) = 1$ for $\beta\ < 0$, therefore points are polar in dimensions $d \geq 2k$ and are non-polar when $d < 2k$.

Typically, for $d = 1, 2, \dots$, given $\R^d$-valued $k$-parameter random fields that consist of $d$ i.i.d.~copies of a given real-valued random field, there is a {\em critical dimension} $Q(k)$ such that points {are polar} if $d> Q(k)$ and are {\em non-polar} if $d < Q(k)$. Deciding whether or not points are polar in the critical dimension $Q(k)$ is usually a difficult problem. Theorem \ref{rd07_01t1} answers this question for the Brownian sheet.

 A more general class of {\em anisotropic} Gaussian random fields was considered in \cite{BLX}. Let $(V(x),\ x \in \R^k)$ be a centered continuous Gaussian random field with values in $\R^d$ and with i.i.d.~components: $V(x) = (V_1(x), \dots,V_d(x))$, and let $I$ be a box. Suppose that there exists $0<c<\infty$ and $H_1,\dots,H_k \in (0,1)$ such that for all $x,y \in I$,
\begin{equation}\label{rd07_01c1}
   c^{-1} \sum_{\ell =1}^k \vert x_\ell - y_\ell \vert^{H_\ell} \leq \Vert V_1(x) - V_1(y)\Vert_{L^2(\Omega)}
   \leq c \sum_{\ell =1}^k \vert x_\ell - y_\ell \vert^{H_\ell}.
\end{equation}
The following result was established in \cite{BLX}.

\begin{theorem}\label{rd07_01t2}
   Assume condition \eqref{rd07_01c1} (and some additional non-degeneracy conditions). Fix $M > 0$. Set
$
   Q = \sum_{\ell=1}^k \frac{1}{H_\ell}.
$
Then there is $0<c<\infty$ such that for every compact set $A\subset [-M, M]^d$,
\begin{equation}\label{rd07_01e4}
   c^{-1}\, \mbox{\rm Cap}_{d-Q}(A) \leq \P\{V(I) \cap A \neq \emptyset\} \leq c\, \mathscr{H}_{d-Q}(A).
\end{equation}
\end{theorem}

In this theorem, for $\beta \in \R$, $\mathscr{H}_\beta(A)$ denotes the $\beta$-dimensional Hausdorff measure of $A \subset \R^d$ (see \cite[Appendix C]{khosh}). If $V_1$ is a $k$-parameter Brownian sheet, then for $\ell \in \{1,\dots, k\}$, each coordinate mapping $x_\ell \mapsto W(x_1,\dots, x_\ell, \dots, x_k)$ is essentially a Brownian motion, therefore $H_\ell = \frac 1 2$ and $Q = 2k$, so the dimension $\beta = d - Q$ in \eqref{rd07_01e4} matches the index $\beta = d - 2k$ in \eqref{rd07_01e3}. However, in the critical dimension $d= Q$ (assuming that $Q$ is an integer), the $0$-dimensional Hausdorff measure of a singleton $A = \{z\}$, $z \in \R^d$, is $1$, therefore, \eqref{rd07_01e4} is not sufficient to decide whether or not points are polar in the critical dimension $d = Q$. Nevertheless, Theorem \ref{rd07_01t2} seems to be the sharpest result available for a wide class of Gaussian random fields.

We now recall the results that are available for the stochastic heat equation in spatial dimension $1$. Let $V = (V(t,x),\ (t,x) \in \R_+ \times \R)$ be an $\R^d$-valued random field that solves a {\em linear system} of stochastic heat equations, namely 
\begin{equation}\label{rd07_01e5}
  \frac{\partial}{\partial t}  V(t,x) = \frac12\frac{\partial^2}{\partial x^2} V(t,x) + \dot W(t,x), \qquad x \in \R,\ t >0,
\end{equation}
where $V(0, *): \R \to \R^d$ is given and $\dot W(t,x)$ denotes $\R^d$-valued (Gaussian) space-time white noise (see \cite[Chapter 1]{dss}). Mueller and Tribe \cite{muellertribe} established the following result.

\begin{theorem}\label{rd07_01t3}
For the random field $V$, the critical dimension for hitting points is $d=6$ and points are polar in this dimension.
\end{theorem}

A result analogous to Theorem \ref{rd07_01t2} was established for $V$ by Dalang, Khoshnevisan and E.~Nualart \cite{DKN07}.

\begin{theorem}\label{rd07_01t4}
Fix $M > 0$ and  two non-trivial compact intervals  $I \subset (0, T] $ and $J \subset \R$. There is $0<c<\infty$ such that for every compact set $A\subset [-M, M]^d$,
\begin{equation}\label{rd07_01e8}
   c^{-1}\, \mbox{\rm Cap}_{d-6}(A) \leq \P\{V(I \times J) \cap A \neq \emptyset\} \leq c\, \mathscr{H}_{d-6}(A).
\end{equation}
\end{theorem}

It is well-known \cite[Section 3.2.2]{dss} that $(t, x) \mapsto V(t, x)$ is jointly H\"older continuous with exponents $(1/4, 1/2)$, therefore $Q = (1/4)^{-1} + (1/2)^{-1} = 6$ and the presence of the dimension $d-6$  can also be deduced from Theorem \ref{rd07_01t2}.

We now consider the related {\em nonlinear system} of stochastic heat equations. Let $U = (U(t,x),\,(t,x) \in \R_+ \times \R)$ be an $\R^d$-valued random field that solves 
\begin{equation}\label{rd07_01e6}
   \frac{\partial}{\partial t}  U(t,x) = \frac12\frac{\partial^2}{\partial x^2} U(t,x) +\sigma(U(t, x)) \cdot \dot W(t,x), \qquad x \in \R,\ t >0,
\end{equation}
where $U(0, *): \R \to \R^d$ is given, $\dot W(t,x)$ is an $\R^d$-valued space-time white noise and ${\sigma}=(\sigma_{i, j})_{1\leq i, j\leq d}$: $\R^d \to \MM_{d \times d}$ is a smooth matrix-valued function.

Under the assumption that $\sigma$ is {\em uniformly elliptic,} the following result for $U$ was established by Dalang, Khoshnevisan and E.~Nualart \cite{DKN09}.

\begin{theorem}\label{rd07_01t5} 
 Fix  $\eta>0$, $M > 0$ and  two non-trivial compact intervals  $I \subset (0, T] $ and $J \subset \R$.
There exists $c>0$ (depending in particular on $\eta$)  such that for all compact sets $A \subseteq [-M, M]^d$,
\begin{equation}\label{rd07_01e7}
c^{-1} \,  \textnormal{Cap}_{d-6+\eta}(A) \leq \P \{ U(I \times J) \cap A \neq
\emptyset \}  \leq c \, \mathscr{H}_{d-6-\eta}(A).
\end{equation}
\end{theorem}

This result is similar to that of Theorem \ref{rd07_01t4} for the Gaussian solution to the linear stochastic partial differential equation (SPDE) \eqref{rd07_01e5}, but it is less sharp because of the additional $\pm \eta$ on the left- and right-hand sides of \eqref{rd07_01e7}. For many years, it remained unclear whether or not \eqref{rd07_01e7} could be improved. In particular, is the extra $\pm \eta$ in \eqref{rd07_01e7} needed for the nonlinear SPDE \eqref{rd07_01e6}?

For the lower bound in  \eqref{rd07_01e7}, this question was finally settled in \cite{FP} (see also \cite{DP1}), where the following result was established.

\begin{theorem}\label{rd07_01t6} 
Under the same hypotheses as in Theorem \ref{rd07_01t5} (nonlinear system of stochastic heat equations), there is $c>0$ such that for all compact sets $A \subset [-M, M]^d$, 
\begin{equation}\label{rd07_01e9}
c^{-1} \,  \textnormal{Cap}_{d-6}(A) \leq \P \{ U(I \times J) \cap A \neq
\emptyset \}.
\end{equation}
\end{theorem}

The main objective in this paper is to improve the upper bound in \eqref{rd07_01e7}, so as to obtain an upper bound that matches that in \eqref{rd07_01e8}.

A step in this direction was made in \cite[Chapter 4]{FP}, where it was shown that it would be sufficient to obtain suitable bounds on the joint probability density function of $(F_1^u, F_2^u)$,
where
\begin{equation*}
   F_1^u=u(t_0,x_0)\qquad\text{and}\qquad F_2^u=  \sup _{\substack {t_0\le t \le t_0+\zeta_1 \\ x_0\le x \le x_0+\zeta_2}}
(u(t,x)-u(t_0,x_0)).
\end{equation*}
where $u$ is the solution to the stochastic heat equation with multiplicative noise (see \eqref{eq:SHE}).
This program was (partly) carried out for $v$ instead of $u$ in \cite{DP1}, where $v$ is the solution to the stochastic heat equation with additive noise (see \eqref{rd03_27e2}). 

Here, we do not seek the required estimates for the joint probability density function of $(F_1^u, F_2^u)$, but we establish analogous estimates for $(F_1^u, F_2^v)$, where
\begin{equation*}
   F_2^v=  \sup _{\substack {t_0\le t \le t_0+\zeta_1 \\ x_0\le x \le x_0+\zeta_2}}
(v(t,x)-v(t_0,x_0)).
\end{equation*}
We derive a formula (involving iterated Skorohod integrals: see Remark \ref{rd07_02r1}) for the probability density function of $(F_1^u, F_2^v)$ based on a general criterion in \cite{fn}  for locally nondegenerate random variables and establish a Gaussian-type upper bound on this density function. 
Then we show that these estimates are sufficient to obtain the sharp upper bound
\begin{equation*}\label{rd04_08e9}
       \P\{U(I \times J) \in A \} \leq C \mathscr{H}_{d-6}(A).
\end{equation*}
Together with \eqref{rd07_01e9}, this improves \eqref{rd07_01e7} and gives sharp upper and lower bounds for hitting probabilities of $U$ that match those for $V$ in \eqref{rd07_01e8}. 

The fact that it is sufficient to obtain bounds on the joint density of $(F_1^u, F_2^v)$ comes from the fact that increments of $u$ and of $v$ are not very different. Indeed, we show that for $p \geq 1$ and $(t_0, x_0)$ in a compact subset $I\times J$ of $(0, T] \times \R$, 
\begin{align*}
    &\E\left[\left( \sup_{\substack {t_0\le t \le t_0+2^{-4n} \\ x_0\le x \le x_0+2^{-2n}}} |u(t,x)-u(t_0,x_0) - \sigma(u(t_0,x_0))(v(t,x)-v(t_0,x_0))| \right)^p\, \right]
    \leq C_{T,p}\,  2^{-\frac 32 pn}
\end{align*}
(see Proposition \ref{rd_prop1} below).

\section{Main results}

Let $u=(u(t, x),\, (t, x) \in [0, \infty) \times \R)$ be the solution (in the sense of \cite{spdewalsh}; see also \cite[Chapter 4]{dss}) of the nonlinear stochastic heat equation with multiplicative noise
\begin{align}\label{eq:SHE}
\begin{cases}
\frac{\partial}{\partial t} u(t\,,x) =\frac{1}{2}\frac{\partial^2}{\partial x^2}u(t\,,x) +\sigma(u(t\,,x)) \dot{W}(t\,,x),\quad  t>0,\ x\in \R,\\
u(0, x) = 0, \hskip 2.45in x \in \R,
\end{cases}
\end{align}
 where $\dot{W} = (\dot{W}(t, x))$ denotes a space-time white noise and $\sigma: \R\to \R$. The standing assumptions on $\sigma$ are that $\sigma$ is  
 Lipschitz continuous and such that $\inf_{z\in \R} \vert \sigma(z)\vert > 0$. 
 Since $z \mapsto \sigma(z)$ is a continuous function, we can assume that 
 \begin{equation}\label{rd07_02e7}
     \inf_{z\in \R}  \sigma(z) > 0. 
 \end{equation}
 This condition implies in particular that the solution $u$ belongs to $\bD^{1, p}$ (see \cite[Theorem 2.4.3 and Proposition 2.4.4]{nualart2006} for this result and for the definition of $\bD^{1, p}$). 
 
 Let $(U(t\,,x),\, (t, x)\in [0, \infty)\times \R)$ be an $\R^d$-valued stochastic process whose coordinates are independent copies of the solution to \eqref{eq:SHE}.  Our main result is the following theorem.

\begin{theorem}\label{th:hitting}
Fix $T>0$ and $M>0$. Let $I \subset (0, T]$ and $J\subset \R$ be two fixed compact intervals with positive length. Suppose in addition that $\sigma$ and its derivatives up to order 3 are bounded (in absolute value) by $\Sigma > 0$. Then there exists a finite constant $c = c(I, J,  M, \Sigma)>0$ such that for all compact sets $A\subset [-M, M]^d$, 
\begin{align}\label{rd03_27e3}
           \P\{U(I\times J) \cap A\neq \emptyset\} \leq c\, \mathscr{H}_{d-6}(A).
\end{align}
 \end{theorem}
 
 In the proof of this theorem (which is given at the end of this section), we will compare the process $u=(u(t, x))$ with the Gaussian process $v=(v(t, x))$ that is the solution of the linear stochastic heat equation
 \begin{align}\label{rd03_27e2}
\begin{cases}
\frac{\partial}{\partial t} v(t\,,x) =\frac{1}{2}\frac{\partial^2}{\partial x^2} v(t\,,x) + \dot{W}(t\,,x),\quad  t>0,\ x\in \R,\\
v(0, x)= 0, \hskip 1.8in x \in \R,
\end{cases}
\end{align}
with the same space-time white-noise $\dot W$ as in \eqref{eq:SHE}. It turns out that the increments of the process $v$ are a good approximation of the increments of $u$. Indeed, let $(t_0, x_0) \in I\times J$. If we define a random variable $L(t, x)$ by the relationship
\begin{equation}\label{rd03_28e1}
     u(t,x)-u(t_0,x_0) = \sigma(u(t_0,x_0))(v(t,x)-v(t_0,x_0)) + L(t,x),
\end{equation}
then $L(t, x)$ is small in the sense of the next proposition.

 \begin{proposition} \label{rd_prop1}
Suppose only that $\sigma$ is Lipschitz continuous. For $\zeta_1, \zeta_2 > 0$ and $(t_0, x_0)\in I \times J$, let 
\[
    R^{\zeta_1, \zeta_2}_{t_0,x_0} =  [t_0, t_0+\zeta_1]\times [x_0, x_0+\zeta_2].
\] 
Define
\begin{align*}
    \XX_{\zeta_1, \zeta_2} = \sup_{(t,x)\in R^{\zeta_1, \zeta_2}_{t_0,x_0}}  | L(t,x)|.
\end{align*}
Then for all $p\geq1$, there is a finite constant $C_{T,p}$ such that, for all $\zeta_1, \zeta_2 \in (0, 1]$ and $(t_0, x_0)\in I \times J$,
\begin{align*}\label{rd03_27e1}
   \Vert \XX_{\zeta_1, \zeta_2} \Vert_p \leq C_{T, p} \left[\max\left(\zeta_1^{1/4}, \zeta_2^{1/2}\right)\right]^{3/2}.
 \end{align*}
 \end{proposition}
 
 The proof of this proposition is given in Section \ref{rd08_04s1}. 
\medskip

\noindent{\em A general result on hitting probabilities}
\medskip

Hitting probabilities for $U$ are closely related to properties of increments of $u$, as one can see for instance from \cite[Proposition 6.5.1]{dss}. In fact, we have the following general result that can be used to relate upper bounds on hitting probabilities for the {\em non-Gaussian} random field $U$ to estimates on the probability density function of the supremum of certain increments of the {\em Gaussian} random field $v$.

Let $u=(u(t, x))$ and $v=(v(t, x))$ be arbitrary continuous real-valued random fields. Let 
\[
   U =(U_1, \ldots, U_d) = (U(t, x),\, (t, x) \in [0, \infty) \times \R )
\] 
be a random field with values in $\R^d$ and with i.i.d.~components, each with the law of $u$. 
Let $I \subset (0, T]$ and $J \subset \mathbb{R}$ be two compact intervals with positive length.

\begin{theorem}\label{th2017-11-23-1}
Fix $0 \leq C < \infty$. For $(t_0, x_0) \in I \times J$,  let $X_{t_0, x_0}$ be a random variable such that  $\vert X_{t_0, x_0}\vert \leq C$ a.s. 
Suppose that the following properties hold:

   (1) There are $H_1, H_2 > 0$ and $\alpha > 1$ such that for all large $p, n \in \N$ and $(t_0, x_0) \in I \times J$, there is $C_{T, p} < \infty$ with the following property: setting 
\[
    R_{t_0, x_0}^{(n)} = \left[t_0, t_0 + 2^{-n H_1^{-1}}\right] \times \left[t_0, t_0 + 2^{-n H_2^{-1}}\right], 
\]
\begin{equation}\label{rd04_03e2}
   L_{t_0, x_0}(t, x) = u(t,x)-u(t_0,x_0) -X_{t_0, x_0}\, (v(t,x)-v(t_0,x_0)),
\end{equation}
and 
\begin{equation}\label{rd04_01e1}
    \XX_n = \XX_n(t_0, x_0) = \sup_{(t,x)\in R^{(n)}_{t_0,x_0}}  | L_{t_0, x_0}(t,x)|,
\end{equation}
 we have
\begin{equation}\label{rd04_01e2}
       \E\left[ (\XX_n)^p\right] \le C_{T,p}\,  2^{- \alpha p n} ;
\end{equation}
   
   (2) for all $(t, x) \in I \times J$,  $u(t, x)$ has an absolutely continuous probability law, and for all $M \in \R_+$, the probability density function $p_{t, x}(z)$ of $u(t, x)$ is bounded uniformly over $(t, x) \in I \times J$ and $z \in [-M, M]$;

   (3) there exists a constant $c = c(I, J)$ such that for all $(t_0, x_0) \in I \times J$ and $\eta> 0$ sufficiently small, the $\R^2$-valued random vectors
\begin{align*}
\left(u(t_0, x_0), \sup_{(t, x) \in [t_0, t_0 + \eta^{H_1^{-1}}] \times [x_0, x_0 + \eta^{H_2^{-1}}]}(v(t, x) - v(t_0, x_0))\right),
\end{align*}
and \begin{align*} 
\left(-u(t_0, x_0), \sup_{(t, x) \in [t_0, t_0 + \eta^{H_1^{-1}}] \times [x_0, x_0 + \eta^{H_2^{-1}}]}(-v(t, x) + v(t_0, x_0))\right)
\end{align*}
have joint probability density functions, denoted by $p_{+}(\cdot, \cdot)$ and $p_{-}(\cdot, \cdot)$ respectively, that satisfy,   for all $z_1 \in \mathbb{R}$ and $z_2 \geq \eta$, 
\begin{align} \label{eq2017-11-23-5}
p_{\pm}(z_1, z_2) &\leq \frac{c}{\eta} 
   \exp\left(-\frac{z_2^2}{c\,  \eta^2} \right).
\end{align}
Then there exists a constant $C = C(I, J, M)$ such that for all compact sets $A \subset [-M, M]^d$,
\begin{align} \label{eq2017-11-23-11}
\P\{U(I \times J) \cap A \neq \emptyset\} \leq C \, \mathscr{H}_{d - H_1^{-1} - H_2^{-1}}(A).
\end{align}
\end{theorem}

This theorem is proved in Section \ref{rd08_11s1}.
\bigskip

\noindent{\em Back to the stochastic heat equation}
\medskip

In view of Theorem \ref{th2017-11-23-1}, we will establish the following result.

\begin{theorem} \label{thm1} 
Let $u$ and $v$ be respectively the solutions of the SPDEs \eqref{eq:SHE} and \eqref{rd03_27e2}. 
Suppose in addition that $\sigma$ and its derivatives up to order 3 are bounded (in absolute value) by $\Sigma > 0$.
There is a constant $c=c(I,J, \Sigma)$ with the following property:

For $(t_0,x_0) \in I\times J$ and  $\deltah_1, \deltah_2\in (0,1)$,  consider the random vector $F=(F_1, F_2)$ $($where $F_2 = F_2(\deltah_1, \deltah_2))$ defined by
\begin{equation}\label{rd07_30e1}
     F_1=u(t_0,x_0)\qquad\text{and}\qquad F_2=  \sup _{\substack {t_0\le t \le t_0+\deltah_1 \\ x_0\le x \le x_0+\deltah_2}}
(v(t,x)-v(t_0,x_0)).
\end{equation}
Then:

    (a)  The random vector $F$ admits a probability density function on $\R\times (0, \infty)$, denoted by $p_{F}(z_1,z_2)$, $z_1\in \R$, $z_2>0$.  
     
     (b)  For any  $(t_0,x_0)\in I\times J$,  for all sufficiently small  $\delta_1> 0,\zeta_2>0$, for all $z_1\in \R$ and $z_2 \ge \deltah^{1/2}:= \max(\deltah_1^{1/4}, \deltah_2^{1/2})$,
\begin{align}\label{rd07_02e8}
   p_{F}(z_1,z_2) \le \frac{c}{\sqrt{\deltah}} \exp\left(-\frac{1}{c} \left(z_1^2 + \frac{z_2^2}{\deltah}\right) \right).
\end{align}
\end{theorem}

This theorem is proved at the end of Section \ref{rd08_11s2}. The simplicity of the bound \eqref{rd07_02e8} hides quite a complex sequence of estimates of norms of various Malliavin derivatives, that begins with the extension to {\em locally nondegenerate} random vectors of the formula in \cite[Proposition 2.1.5]{nualart2006}, which applies to (globally) nondegenerate random vectors: see Proposition \ref{rd04_04p1} and Remark \ref{rd07_02r1}. The main Malliavin calculus estimate is in Proposition \ref{prop1}, which makes use of the technical estimates in Section \ref{rd07_02s1}. In particular, on the right-hand side of \eqref{equ3} in Lemma \ref{rd07_02l1}, the bound $\zeta^{5/2}$ appears. This exponent $5/2$ appears in the bounds of several different terms, such as \eqref{eq:u1}, \eqref{eq:u2}, \eqref{eq:u4}, \eqref{eq:u5}, \eqref{rd08_06e1}, \eqref{D2UA2} and \eqref{rd08_06e2}. In this last case, the value $5/2$ appears by combining numbers such as (one half of) $16$, $21$, $8$ and $13$ (see e.g. \eqref{rd07_02e9}). We do not have a clear explanation of why the estimates of all these terms ($35$ in all) match and lead to the bound \eqref{rd07_02e8}.
\bigskip

\noindent{\em Proving Theorem \ref{th:hitting}}
\medskip

 Assuming Proposition \ref{rd_prop1} and Theorems \ref{th2017-11-23-1} and \ref{thm1}, the proof of Theorem \ref{th:hitting} is quite straightforward.
 \medskip
 
 \noindent{\em Proof of Theorem \ref{th:hitting}.} By Proposition \ref{rd_prop1} and because $\sigma$ is assumed to be bounded, Hypothesis (1) of Theorem \ref{th2017-11-23-1} holds with $X_{t_0, x_0} = \sigma(u(t_0,x_0))$, $H_1 = 1/4$, $H_2 = 1/2$ and $\alpha = 3/2$. Hypothesis (2) of Theorem \ref{th2017-11-23-1} concerning the uniform boundedness of the probability density function $p_{t, x}(z)$ can be checked in the same way that \cite[Corollary 4.3]{DKN09} is deduced from \cite[Proposition 4.1]{DKN09} and \cite[Proposition 4.2]{DKN09}, but more simply, since \cite{DKN09} considers systems of SPDEs.
 
 
 By Theorem \ref{thm1}(b), setting $H_1 = 1/4$, $H_2 = 1/2$ and $\eta = \deltah^{1/2}$, we have $p_+(\cdot, \cdot) = p_F(\cdot, \cdot)$ and the bound \eqref{eq2017-11-23-5} holds for $p_+(\cdot, \cdot)$ by \eqref{rd07_02e8}. 
 In order to obtain \eqref{eq2017-11-23-5} for $p_-(\cdot, \cdot)$, we set 
 \begin{align*}
 \tilde{u}=-u,\quad  \tilde{v}=-v, \quad  \tilde{W}=-W
 \end{align*}
 and $\tilde{\sigma}(z)=\sigma(-z)$, for all $z\in \R$. Then $\tilde{u}$ and $\tilde{v}$ satisfy respectively 
 \begin{align*}
\begin{cases}
\frac{\partial}{\partial t} \tilde{u}(t,x) =\frac{1}{2} \frac{\partial^2}{\partial x^2} \tilde{u}(t,x) +\tilde{\sigma}(\tilde{u}(t\,,x)) \dot{\tilde{W}}(t\,,x),\quad  t>0,\ x\in \R,\\
\tilde{u}(0)\equiv 0,
\end{cases}
\end{align*}
and 
\begin{align*}
\begin{cases}
\frac{\partial}{\partial t}  \tilde{v}(t,x) =\frac{1}{2} \frac{\partial^2}{\partial x^2}  \tilde{v}(t,x) + \dot{\tilde{W}}(t\,,x),\quad  t>0,\ x\in \R,\\
\tilde{v}(0)\equiv 0.
\end{cases}
\end{align*}
Theorem \ref{thm1}  applies to the processes $\tilde{u}$ and $\tilde{v}$ and hence gives rise to the bound \eqref{eq2017-11-23-5} for  $p_-(\cdot, \cdot)$. This shows that Hypothesis (3) of Theorem \ref{th2017-11-23-1} holds.

Finally, we conclude from Theorem \ref{th2017-11-23-1} that \eqref{eq2017-11-23-11} holds with $H_1 = 1/4$ and $H_2 = 1/2$, and this is \eqref{rd03_27e3} since $(1/4)^{-1} + (1/2)^{-1} = 6$.
 \qed

\begin{remark}
It is also possible to consider systems of stochastic heat equations of the form
\begin{align*}
          \frac{\partial}{\partial t} U(t\,,x) =\frac{1}{2} \frac{\partial^2}{\partial x^2} U(t\,,x) + \sigma(U(t\,,x))\cdot \dot{W}(t\,,x) + b(U(t\,,x)),
\end{align*}
where $U=(U_1, \ldots, U_d)$, $\dot{{W}}$ denotes a $d$-dimensional space-time white noise, $\sigma: \R^d \to \MM_{d \times d}$ is a matrix-valued function, and $b$ is a function from $\R^d$ to $\R^d$. By extending our arguments, it should be possible to obtain upper bounds on hitting probabilities in terms of $\mathscr{H}_{d-6}(A)$ for $U$ in this case also. 
 \end{remark}
 
 We conclude this section with a brief overview of the notation used throughout the paper.
For $Y\in L^p(\Omega)$ with $p\in[1,\infty)$, we write $\|Y\|_p=(\E[|Y|^p])^{1/p}$.  Throughout the paper, we write ``$g_1(x)\lesssim g_2(x)$ for all $x\in X$'' when
there exists a real number $L$ such that $g_1(x)\le L\, g_2(x)$ for all $x\in X$.
Alternatively, we might write ``$g_2(x)\gtrsim g_1(x)$ for all $x\in X$.'' By
``$g_1(x)\asymp g_2(x)$ for all $x\in X$'' we mean that $g_1(x)\lesssim g_2(x)$
 and $g_2(x)\lesssim g_1(x)$ for all $x\in X$. For a bounded function $\phi$ on $\R$, we denote $\|\phi\|_\infty=\sup_{x\in \R}|\phi(x)|$.
 We also use the notations $C_{T,k}$, $C$ and $c$ to denote generic constants whose values may change from line to line.

 \section{Proving Proposition \ref{rd_prop1}}\label{rd08_04s1}

 We first develop the main arguments and then we state and prove two technical lemmas that are needed in the proof.
 
 \subsection{The main arguments}
 
 \noindent
 {\it Step 1:}  For $p \geq 2$, we are going to estimate the $L^p$-norm of the increments of the process $L$ defined in \eqref{rd03_28e1}. Let $(t,x), (s,y) \in R^{\zeta_1,\zeta_2}_{t_0,x_0}$ and define $\eta = \max(\zeta_1^{1/4},\zeta_2^{1/2})$, so that $\zeta_1 \zeta_2 \leq \eta^6$. We can write
 \begin{align*}
 L(t,x)-L(s,y) &= u(t,x)-u(s,y) -\sigma(u(t_0,x_0)) (v(t,x)-v(s,y))\\
     &= A_1+A_2+A_3+A_4,
 \end{align*}
 where
 \begin{align*}
 A_1 & =u(t,x)-u(t,y) -\sigma(u(t,y))\, [v(t,x)-v(t,y)], \\
  A_2&= u(t,y)-u(s,y) -\sigma(u(s,y))\, [v(t,y)-v(s,y)], \\
  A_3&= [\sigma(u(t,y)) - \sigma(u(t_0,x_0))]\, [v(t,x)-v(t,y)], \\
 A_4&= [\sigma(u(s,y)) - \sigma(u(t_0,x_0))]\, [v(t,y)-v(s,y)].
 \end{align*}
 From Lemma \ref{lem:spatial} below, replacing $x_0$ by $y$, we obtain
 \[
 \E [ |A_1|^p|] \le C_{T,p}\, |x-y|^{\frac {3p}4}  \leq C_{T,p}\, \zeta_2^{\frac p 4}\, |x-y|^{\frac {p}2} \le C_{T,p}\, \eta^{\frac p 2}\, |x-y|^{\frac p2}.
 \]
  From Lemma \ref{lem:tem} below,  replacing $t_0$ by $s$ and $x$ by $y$, we obtain 
 \[
 \E [ |A_2|^p|] \le C_{T,p}\,  |t-s|^{\frac {2p}5} \leq C_{T,p}\, \zeta_1^{\frac{3p}{20}}\, |t-s|^{\frac p4} \le C_{T,p}\, \eta^{\frac {3p}5 }\, |t-s|^{\frac p4}.
 \]
 For the terms $A_3$ and $A_4$, we use the Lipschitz continuity of $\sigma$, the Cauchy-Schwarz inequality and the H\"older continuity properties of $u$ and $v$ \cite[Proposition 4.3.2] {dss}, to see that
 \[
 \E [ |A_3|^p|] \le C_{T,p}   \left( |t-t_0|^{\frac 14}+|y-x_0|^{\frac 12}\right)^p  |x-y|^{\frac p2} \le  C_{T,p}\,  \eta^{p}\, |x-y|^{\frac p2} 
 \]
 and
  \[
 \E [ |A_4|^p|] \le C_{T,p}   \left( |s-t_0|^{\frac 14}+|y-x_0|^{\frac 12}\right)^p  |t-s|^{\frac p4} \le  C_{T,p}\,  \eta^{p}\, |t-s|^{\frac p4} .
 \]
 Putting these bounds together and using the notation 
 \begin{equation}\label{rd05_16e1}
    \Delta((t,x);(s,y))= \max\left(|t-s|^{\frac 14}, |x-y|^{\frac12}\right),
 \end{equation}
 we obtain
 \begin{equation} \label{k1}
 \E[ |L(t,x)-L(s,y)|^p] \le C_{T,p} \, \eta^{\frac{p}2} \Delta((t,x);(s,y))^{p}.
 \end{equation}
 
  \medskip

 \noindent
 {\it Step 2:}  We will use a version of the Garsia-Rodemich-Rumsey lemma (see \cite[Lemma A.3.5]{dss}). We apply this proposition to the process $L$ with $S=R_{t_0,x_0}^{\zeta_1, \zeta_2}$, $\rho((t,x);(s,y))=\Delta^2((t,x);(s,y))$, $\mu(\d t \d x)=\d t \d x$,
 $\Psi(x)=|x|^p$, $\Psi^{-1}(y)=y^{1/p}$ and $p(x)=|x|^{\alpha+\frac 6 p}$, where $\alpha\in \left(0,\frac 12-\frac 6p\right)$ (as in \cite[Corollary A.3]{DKN07}).  This requires $p>12$ (so that $\frac 12-\frac 6p > 0$), but it suffices to establish the proposition for $p\ge p_0$, for  some $p_0 \geq 1$.
 
 Define
 \[
 \mathscr{C}=\int_S \d t\d x\int_S \d s \d y\,  \frac { |L(t,x)-L(s,y)|^p}{ [ \Delta^2((t,x);(s,y)]^{6+\alpha p}}.
 \]
 By  (\ref{k1}), we can write
  \begin{align*}
 \E [\mathscr{C}]& \le C_{T,p}\, \eta^{\frac {p}2} \int_S \d t\d x\int_S \d s \d y\,   [ \Delta((t,x);(s,y)]^{ p - 12-2\alpha p} \\ 
 & \le C_{T,p}\, \eta^{\frac {p}2}\zeta_1\zeta_2 \int_0^{\zeta_1} \d u \int_0^{\zeta_2} \d v\, (u^{1/2}+v)^{\frac 1 2 ( p  - 12 - 2\alpha p)}.
 \end{align*}
 The double integral is bounded above by $C \zeta_1\zeta_2 \eta^{p - 2 \alpha p - 12} \leq C \eta^{p - 2 \alpha p - 6}$, therefore,
 \begin{equation}
      \E [\mathscr{C}] \le C_{T,p} \, \eta^{(\frac 32 - 2\alpha )p}.  
 \label{k2}
 \end{equation}
 From  \cite[Proposition A.3.5]{dss}, we deduce that for $(t, x), (s, y) \in R_{t_0,x_0}^{(n)}$,
 \begin{align*}
      |L(t,x)-L(s,y)| \le 10 \sup_{(r,z)\in R^{\zeta_1, \zeta_2}_{t_0,x_0}} \int_0^{2\Delta^2((t,x);(s,y))} \Psi^{-1}
 \left( \frac {\mathscr{C}}{[\mu(B_\rho((r,z),u/4)]^2} \right) u^{\alpha+\frac 6p -1}\, \d u,
 \end{align*}
 where $B_\rho((r,z),u/4)$ denotes the open ball in $R_{t_0,x_0}^{\zeta_1, \zeta_2}$ centered at $(r, z)$ of radius $u/4$ in the metric $\Delta^2$. One can check that there is a constant $c>0$ such that  $\mu(B_\rho((r,z),u/4)]) \ge cu^3$, for all $u \in [0, \eta^2]$ and
 $(r,z)\in  R^{\zeta_1, \zeta_2}_{t_0,x_0}$. Since $\Delta((t, x), (s, y)) \leq \eta$, 
  \begin{align*}
      |L(t,x)-L(s,y)|  &\lesssim \int_0^{2\Delta^2((t,x);(s,y))} \mathscr{C}^{1/p} u^{\alpha-1} \d u
 \asymp \mathscr{C}^{1/p}\, [\Delta^2((t,x);(s,y))]^\alpha.
  \end{align*}
  In particular, taking $(s,y)=(t_0,x_0)$, we obtain
  \[
  |L(t,x)| \lesssim  \mathscr{C}^{1/p}\, [\Delta^2((t,x);(t_0,x_0))]^\alpha \le     \mathscr{C}^{1/p} \, \eta^{2\alpha}.
\]
Using (\ref{k2}), this implies
\[
\E \left[ \sup_{(t,x)\in R^{\zeta_1, \zeta_2} _{t_0,x_0}} |L(t,x)|^p \right] \le C_{T,p}\, \eta^{\frac 32 p},
\]
 which completes the proof of the proposition.
 \qed

 \subsection{Two technical lemmas}
 Recall the mild form of the solution to \eqref{eq:SHE}:
 \begin{align}\label{rd07_23e1}
 u(t,x)= \int_0^t\int_\R G(t-s, x-y)\sigma(u(s\,,y))W(\d s\, \d y),
 \end{align}
 where $G(t,x)= (2\pi t)^{-1/2}\e^{-x^2/(2t)}$ and $W(\d s\, \d y)$ denotes the It\^o-Walsh stochastic integral (see \cite{spdewalsh,dss}).
 The solution to \eqref{rd03_27e2} is given by
 \begin{align*}
 v(t\,,x) = \int_0^t\int_\R G(t-s, x-y)W(\d s\, \d y).
 \end{align*}
Increments of $u$ be can approximated in various ways by those of $v$; see \cite{KSXZ13,FKM15,HP15}. The following lemma on the approximation of temporal increments of $u$ by those of $v$ is taken from \cite{KSXZ13}.

\begin{lemma}\label{lem:tem} 
{\rm \cite[Theorem 4.1]{KSXZ13}}
Assume only that $\sigma$ is Lipschitz. For $T>0$ and $p\geq2$, there exists $A_{T,p}>0$ such that for all $t, t_0\in[0, T]$ and $x\in \R$
           \begin{align*}
           \|u(t\,,x)-u(t_0\,,x)- \sigma(u(t_0\,,x))(v(t\,,x)-v(t_0\,,x))\|_p \le A_{T,p} |t-t_0|^{2/5}.
           \end{align*}
 \end{lemma}
 
 Using  ideas similar to those in \cite{KSXZ13,FKM15}, we obtain the following result on the approximation of spatial increments of $u$ by those of $v$.

\begin{lemma}\label{lem:spatial}
Assume only that $\sigma$ is Lipschitz. For $T> t_0 > 0$ and $p\geq2$, there exists a constant $C = C_{T,t_0,p}>0$ such that for all $t\in [t_0, T]$ and $x, x_0\in \R$,
\begin{align*}
          \left\Vert u(t\,,x)- u(t\,, x_0)-\sigma(u(t\,,x_0))(v(t,x)-v(t,x_0)) \right\Vert_p
           \leq C |x-x_0|^{\frac{3}{4}}.
 \end{align*}
 \end{lemma}
 \begin{proof}
Because $\Vert u(t, x) \Vert_p$ and $\Vert v(t, x) \Vert_p$ are bounded over $[0, T] \times \R$ \cite[Theorem 4.1]{dss} and because $\sigma$ has linear growth, it suffices to prove the above estimate for $x, x_0\in \R$ such that $|x-x_0|\leq 1\wedge t_0$.  Hence we assume $|x-x_0|\leq 1\wedge t_0$ in the following. 
           
According to \eqref{rd07_23e1}, we have
\begin{align*}
           u(t\,,x)- u(t\,,x_0)
           = \int_0^t\int_\R (G(t-s, x-y)- G(t-s, x_0-y))\sigma(u(s\,,y)) W(\d s\, \d y).
 \end{align*}
For $\varepsilon <\sqrt{t_0}$, let
\begin{align*}
           &\mathcal{A}(t, \varepsilon, x, x_0) 
           = \int_{t-\varepsilon^2}^t\int_{B(x_0, 2\varepsilon)} (G(t-s, x-y)- G(t-s, x_0-y))\sigma(u(s\,,y)) W(\d s\, \d y),
\end{align*}
where $B(x_0, 2\varepsilon):= [x_0-2\varepsilon, x_0+2\varepsilon]$. Then 
 \begin{align} \label{I1+I2}
          \|u(t\,,x)- u(t\,,x_0) - \mathcal{A}(t, \varepsilon, x, x_0)\|^2_p\lesssim I_1(t, \varepsilon, x, x_0)+I_2(t, \varepsilon, x, x_0),
\end{align}
where         
 \begin{align*}
       I_1(t, \varepsilon, x, x_0)     &=  \left\|\int_0^{t-\varepsilon^2}\int_\R (G(t-s, x-y)- G(t-s, x_0-y))\sigma(u(s\,,y)) W(\d s\, \d y)\right\|^2_p, \\
       I_2(t, \varepsilon, x, x_0)     &=\left\| \int_{t-\varepsilon^2}^t\int_{\R\setminus B(x_0, 2\varepsilon)} (G(t-s, x-y)- G(t-s, x_0-y))\sigma(u(s\,,y)) W(\d s\, \d y)\right\|^2_p.
\end{align*}
Using Burkholder's inequality, then Minkowski's inequality, and using the facts that $\sigma$ has linear growth and $\Vert u(t, x) \Vert_p$ is bounded \cite[Theorem 4.2.1]{dss}, we obtain
\begin{align*}
           I_1(t, \varepsilon, x, x_0) &\lesssim   \int_0^{t-\varepsilon^2}\int_\R (G(t-s, x-y)- G(t-s, x_0-y))^2\, \d y \d s\\
           &=  \int_{\varepsilon^2}^{t}\int_\R (G(s, x-y)- G(s, x_0-y))^2\,  \d y \d s.
\end{align*}
Using  Plancherel's identity, we see that
\begin{align*}
           I_1(t, \varepsilon, x, x_0) &\lesssim \int_{\varepsilon^2}^t\int_\R \left| \e^{ixz}\e^{-\frac12sz^2} - 
            \e^{ix_0z}\e^{-\frac12sz^2}\right|^2\d z \d s\\
            &= 2 \int_{\varepsilon^2}^t\int_\R \e^{-sz^2} (1- \cos((x-x_0)z))\, \d z \d s\\
            &\leq 2\int_\R \frac{\e^{-\varepsilon^2 z^2}}{z^2} (1- \cos((x-x_0)z))\, \d z.
\end{align*}
With the change of variables $y = (x-x_0) z$, we obtain
 \begin{align*}
           I_1(t, \varepsilon, x, x_0) &\lesssim 2 |x-x_0|  \int_\R \e^{- \ep^2y^2/(x-x_0)^2} \left(1 - \cos y\right) \frac{\d y}{y^2} .
 \end{align*}           
Using the inequality $1-\cos y \leq \frac12 y^2$ for all $y\in \R$, we see that
\begin{align*}
           I_1(t, \varepsilon, x, x_0) &\lesssim  |x-x_0| \int_\R \e^{-\varepsilon^2 y^2/(x-x_0)^2} \d y \lesssim \frac{(x-x_0)^2}{\ep} .         
\end{align*}
Letting $\varepsilon = |x-x_0|^\half$,  we obtain 
\begin{align}\label{eq:I1}
   I_1(t, |x-x_0|^\half, x, x_0) &\lesssim  |x-x_0|^{3/2}.
\end{align}

Using again Burkholder's inequality, then Minkowski's inequality, and the facts that $\sigma$ has linear growth and $\Vert u(t, x) \Vert_p$ is bounded, we obtain 
\begin{align*}
           I_2(t, \varepsilon, x, x_0) 
           & \lesssim \int_{t-\varepsilon^2}^t\int_{|y-x_0|>2\varepsilon} (G(t-s, x-y)- G(t-s, x_0-y))^2\, \d y \d s\\
           &= \int_0^{\ep^2} \int_{\vert z \vert > 2 \ep} (G(s, x - x_0 + z) - G(s, z))^2\, \d z \d s. 
\end{align*}
Do the change of variables $s = \ep^2 r$, $z= \ep v$, to see that
\begin{align*}
   I_2(t, \varepsilon, x, x_0) &\lesssim  \varepsilon^3 \int^{1}_{0}\int_{|v|>2} \left(G\left(\varepsilon^2 r, x-x_0 +\varepsilon v\right)- G(\varepsilon^2 r, \varepsilon v)\right)^2 \d v \d r \\
           &=  \varepsilon \int^{1}_{0}\int_{|v|>2} \left(G\left(r, \frac{x-x_0}{\varepsilon} + v\right)- G(r, v)\right)^2 \d v \d r,
\end{align*}
where in the second equality, we have used the identity 
\begin{align}\label{eq:scale}
           G(\varepsilon^2 s, \varepsilon y)= \varepsilon^{-1} G(s, y)\quad  \text{for all $\varepsilon, s>0$ and $y\in \R$}.
\end{align}
Set $\beta = (x - x_0)/\varepsilon = (x-x_0) |x-x_0|^{-\half}$. Then 
\begin{align*}
           I_2(t, |x-x_0|^\half, x, x_0) &\lesssim   
           |x-x_0|^\half \int^{1}_{0}\int_{|v|>2} (G(r, \beta + v)- G(r, v))^2\, \d v \d r .
\end{align*}
Notice that $\vert \beta \vert \leq 1$ because $|x-x_0| \leq 1$, and $\vert \beta + v \vert \geq 1$ when $|v|>2$. Note that the heat kernel $G(r, y)$ has a partial derivative in $y$ that is uniformly bounded over the set $\R_+ \times \{\vert y\vert \geq 2 \}$ and satisfies
\begin{equation*}
    \left\vert \frac{\partial}{\partial y} G(r, y) \right\vert = \frac{\vert y \vert}{r} G(r, y).
\end{equation*} 
Therefore,
\begin{align*}
           I_2(t, |x-x_0|^\half, x, x_0) &\lesssim  |x-x_0|^\half   \int^{1}_{0}\int_{|v|>2} \left(\int_v^{v + \beta} \frac{\partial}{\partial z} G(r, z)\,  \d z \right)^2 \d v \d r\\
             &\lesssim  |x-x_0|^{\half}  \beta^2 \int^{1}_{0}\int_{|v|>2} \left(\frac{\vert v \vert + 1}{r} \, G(r, \vert v \vert - 1 )\right)^2\d v\d r \\
             & \lesssim |x-x_0|^{3/2}.
\end{align*}
Thus, 
\begin{align}\label{eq:I2}
          I_2(t, |x-x_0|^\half, x, x_0) &\lesssim |x-x_0|^{3/2}.
\end{align} 
Therefore, we conclude from \eqref{I1+I2}, \eqref{eq:I1} and \eqref{eq:I2} that 
\begin{align}\label{eq:u-A}
           \left\|u(t\,,x)- u(t\,,x_0) - \mathcal{A}(t,|x-x_0|^\half, x, x_0)\right\|^2_p \lesssim |x-x_0|^{3/2}.
\end{align}
       
For $\varepsilon <\sqrt{t_0}$, define
\begin{align*}
           \mathcal{B}_\sigma(t, \varepsilon, x, x_0)&:= \sigma(u(t-\varepsilon^2, x_0)) \\
           &\qquad \times \int_{t-\varepsilon^2}^t\int_{B(x_0, 2\varepsilon)} (G(t-s, x-y)- G(t-s, x_0-y)) W(\d s\, \d y).
\end{align*}
If $\sigma \equiv 1$, we write $\cB_1$ instead of $\cB_\sigma$, and we have 
\begin{align*}
     \cB_\sigma(t, \varepsilon, x, x_0) = \sigma(u(t-\varepsilon^2,x_0)) \cB_1(t, \varepsilon, x, x_0) .
\end{align*}

Because $\sigma$ has linear growth, $\Vert u(t - \varepsilon^2, x_0) \Vert_p$ is bounded above and \eqref{rd07_02e7} holds, we have
\begin{align}\label{eq:B1}
\|\mathcal{B}_1(t, \varepsilon, x, x_0)\|^2_p& \asymp \int_{t-\varepsilon^2}^t\int_{B(x_0, 2\varepsilon)} (G(t-s, x-y)- G(t-s, x_0-y))^2\, \d y\d s\nonumber\\
 &\leq \int_{0}^t\int_{\R} (G(t-s, x-y)- G(t-s, x_0-y))^2 \, \d y\d s \lesssim |x-x_0|,
\end{align}
by \cite[(4.3.10)]{dss}. Using Burkholder's and Minkowski's inequalities, we see that
\begin{align*}
           &\|\mathcal{A}(t, \varepsilon, x, x_0)-\mathcal{B}_\sigma(t, \varepsilon, x, x_0)\|_p^2 \\
           &\qquad \lesssim \int_{t-\varepsilon^2}^t\int_{B(x_0, 2\varepsilon)} (G(t-s, x-y)- G(t-s, x_0-y))^2  \\
           &\qquad\qquad\qquad\qquad \qquad\qquad\times  \| \sigma(u(t-\varepsilon^2\,,x_0)) -\sigma(u(s\,,y))\|_p^2 \, \d y\d s\\
           &\qquad \lesssim \int_{t-\varepsilon^2}^t\int_{B(x_0, 2\varepsilon)} (G(t-s, x-y)- G(t-s, x_0-y))^2   \\
           & \qquad\qquad\qquad\qquad \qquad\qquad\times  (|t-\varepsilon^2-s|^{1/2}+|x_0-y|)\d y\d s\\
           &\qquad \lesssim \varepsilon\,   \int_{0}^t\int_{\R} (G(t-s, x-y)- G(t-s, x_0-y))^2  \d y\d s \\
           &\qquad \lesssim \varepsilon\, |x-x_0|,
           \end{align*}
           where in the second inequality, we use the Lipschitz property of $\sigma$ and the estimate
           \begin{align*}
           \sup_{(t,x), (s,y)\in [0, T]\times \R}\frac{\|u(t\,,x)-u(s\,,y)\|_p}{|t-s|^{1/4}+|x-y|^{1/2}}<\infty 
\end{align*}
(see \cite[Proposition 4.3.2]{dss}),           and in the fourth inequality, we use again \cite[(4.3.10)]{dss}. We take $\ep = |x-x_0|^\half$, to obtain 
           \begin{align}\label{eq:A-B}
           \|\mathcal{A}(t, |x-x_0|^\half, x, x_0)-\mathcal{B}_\sigma(t, |x-x_0|^\half, x, x_0)\|_p^2 \lesssim |x-x_0|^{3/2}.
           \end{align}
Hence, we deduce from \eqref{eq:u-A} and \eqref{eq:A-B} that
          \begin{align}\label{eq:u-B}
           \|u(t\,,x)- u(t\,,x_0) - \mathcal{B}_\sigma(t,|x-x_0|^\half, x, x_0)\|^2_p \lesssim |x-x_0|^{3/2}.
           \end{align}
           
           This inequality also applies to the case $\sigma\equiv 1$, which yields that
\begin{align}\label{eq:V-B}
            \left\|v(t\,,x)- v(t\,,x_0)  - \cB_1(t, |x-x_0|^\half, x, x_0) \right\Vert_p^2 
             \lesssim |x-x_0|^{3/2}.
\end{align}

Because $\sigma$ is Lipschitz, \cite[Theorem 4.3.4]{dss} implies that
\begin{align*}
           \|\sigma(u(t-|x-x_0|, x_0))- \sigma(u(t\,,x_0))\|_p \lesssim |x-x_0|^{1/4}.
 \end{align*}
Therefore,
\begin{align*}
    &  \left\Vert u(t\,,x)- u(t\,, x_0)-\sigma(u(t\,,x_0))(v(t,x)-v(t,x_0)) \right\Vert_p \\
       &\qquad\qquad \leq \Vert u(t\,,x)- u(t\,,x_0) - \mathcal{B}_\sigma(t,|x-x_0|^\half, x, x_0)\Vert_p \\
        &\qquad \qquad\qquad +  \Vert \mathcal{B}_\sigma(t,|x-x_0|^\half, x, x_0) - \sigma(u(t, x_0)) \cB_1(t,|x-x_0|^\half, x, x_0)\Vert_p \\
         &\qquad \qquad\qquad+   \Vert  \sigma(u(t, x_0)) [\cB_1(t,|x-x_0|^\half, x, x_0) - (v(t,x)-v(t,x_0))] \Vert_p.
\end{align*}
By \eqref{eq:u-B}, the first term is bounded by a constant multiple of $|x-x_0|^{3/4}$. By the Cauchy-Schwarz inequality, the second term is bounded above by a constant multiple of 
\begin{align*}
     & \Vert \sigma(u(t - |x-x_0|, x_0))   - \sigma(u(t, x_0)) \Vert_{2p}\ \Vert \cB_1(t,|x-x_0|^\half, x, x_0)\Vert_{2p} \\
     &\qquad \lesssim |x-x_0|^{\frac{1}{4}}\, |x-x_0|^{\frac12} = |x-x_0|^{\frac{3}{4}},
\end{align*}
where the inequality holds by \eqref{eq:B1}.
The last term is bounded above by a constant multiple of 
\begin{align*}
    \Vert \cB_1(t,|x-x_0|^\half, x, x_0) - (v(t,x)-v(t,x_0)) \Vert_{2p} \lesssim |x-x_0|^{\frac{3}{4}}
\end{align*}
by \eqref{eq:V-B}.
The proof is complete.                
 \end{proof}

 \section{Proof of Theorem \ref{th2017-11-23-1}}\label{rd08_11s1}

This section does not make use of the material in Section \ref{rd08_04s1}.
 
\begin{proof}[Proof of Theorem \ref{th2017-11-23-1}]
For all nonnegative integers $n$, $m$ and $\ell$, set $t_m^n:=m\, 2^{-nH_1^{-1}}$, $x_{\ell}^n:=\ell\, 2^{-nH_2^{-1}}$,
and
\begin{equation*} 
	I^n_m = [t_m^n,t_{m+1}^n],\qquad J^n_{\ell} = [x_{\ell}^n,x_{\ell+1}^n],\qquad R^n_{m,\ell} = I^n_m \times J^n_{\ell}.
\end{equation*}
Proceeding as in the proof of \cite[Theorem 3.1]{DKN07}, we will show below that it suffices to check that there is $c > 0$ such that for all  $z\in [-M, M]$, for all sufficiently large $n$ and for all $m,\ell$ such that $R^n_{m,\ell} \cap  (I \times J) \neq \emptyset$,  
\begin{align}\label{eq:smallp}
         \mathcal{P}:= \P\left\{\inf_{(t,x)\in R^n_{m,\ell}} |u(t,x)-z|\leq 2^{-n} \right\} \leq c\, 2^{-n}.
\end{align}
Since the coordinates of $U$ are independent copies of the solution to \eqref{eq:SHE}, this will imply that for $z_0 \in \R^d$ with $\vert z_0 \vert \leq M$,
\begin{align}\label{rd04_01e3}
          \P\left\{\inf_{(t,x)\in R^n_{m,\ell}} |U(t,x)-z_0 |\leq 2^{-n} \right\} \leq c^d\, 2^{-nd},
\end{align}
and this corresponds to $\beta = d$ in \cite[Theorem 3.1]{DKN07}.

Fix $n$, $m$ and $\ell$ such that $R^n_{m,\ell} \cap ( I \times J) \neq \emptyset$. By the reverse triangle inequality, 
\begin{align}\label{rd04_03e1}
           \mathcal{P} \leq \P\left\{|u(t_m^n, x_{\ell}^n)-z| \leq 2^{-n} + \sup_{(t,x)\in R^n_{m,\ell}} |u(t,x)-u(t_m^n, x_\ell^n)|\right\}. 
 \end{align}
        Denote $L_{m, \ell}^n(t, x) := L_{t_m^n, x_\ell^n}(t, x)$, where $L_{t_0,x_0}(t, x)$ is defined in \eqref{rd04_03e2}, and set
 \[
     \XX_{m, \ell}^n := \XX_n(t_m^n, x_\ell^n) =\sup_{(t, x) \in R^n_{m,\ell}} | L_{m, \ell}^n(t, x)|, 
 \]
as in \eqref{rd04_01e1}. Then
\[
    \mathcal{P} \leq  \P\left\{|u(t_m^n, x_{\ell}^n)-z| \leq 2^{-n} + \XX_{m, \ell}^n +  \vert X_{t_m^n, x_\ell^n}\vert \sup_{(t,x)\in R^n_{m,\ell}} |v(t,x)-v(t_m^n, x_\ell^n)|\right\} .
\]
 Since $\vert X_{t_m^n, x_\ell^n}\vert \leq  C$ by assumption, 
 \begin{align*}
 \mathcal{P} &\leq  \P\left\{|u(t_m^n, x_{\ell}^n)-z| \leq 2^{-n} + \XX_{m, \ell}^n + C \sup_{(t,x)\in R^n_{m,\ell}} |v(t,x)-v(t_m^n, x_\ell^n)|\right\} .
 \end{align*}
 For $p \geq 1$ large enough, $(\alpha - 1) p > 1$ because  $\alpha > 1$. Therefore, for large $p$, by Markov's inequality and Property \eqref{rd04_01e2}, 
 \begin{align*}
 \P\{\XX_{m, \ell}^n > 2^{-n} \} \leq 2^{p n}\, \E\left[(\XX_{m, \ell}^n)^p \right] \leq C_{T, p}\, 2^{-(\alpha  - 1) p n} \leq 2^{-n}.
 \end{align*}
 It follows that for large $n$,
  \begin{align}\nonumber
     \mathcal{P}  &\leq 2^{-n} + \P\left\{|u(t_m^n, x_{\ell}^n)-z| \leq 2^{-n+1} + C \sup_{(t,x)\in R^n_{m,\ell}} |v(t,x)-v(t_m^n, x_\ell^n)|\right\} \\
     &\leq 2^{-n} + \mathcal{P}_1+ \mathcal{P}_2,
 \label{rd07_26e1}
 \end{align}
 where
 \begin{align*}
\mathcal{P}_1&=   \P\left\{|u(t_m^n, x_{\ell}^n)-z| \leq 2^{-n+1} + C \sup_{(t,x)\in R^n_{m,\ell}} (v(t\,,x)-v(t_m^n, x_\ell^n))\right\},\\
\mathcal{P}_2&=   \P\left\{|u(t_m^n, x_{\ell}^n)-z| \leq 2^{-n+1}  + C \sup_{(t,x)\in R^n_{m,\ell}} (-v(t\,,x)+v(t_m^n, x_\ell^n))\right\}.
\end{align*} 

Using \eqref{eq2017-11-23-5} for $p_{+}(\cdot, \cdot)$, we will show that 
$\mathcal{P}_1 \leq c_+\, 2^{-n}$, for some constant $c_+$, and a similar estimate for $\mathcal{P}_2$ can be obtained by using \eqref{eq2017-11-23-5} for $p_{-}(\cdot, \cdot)$.

For this, notice that
\begin{align}
\mathcal{P}_1&\leq \mbox{P}\Bigg(\left\{|u(t_m^n, x_\ell^n) - z| \leq 2^{-n+1} + C \sup_{(t, x) \in R_{k, l}^n} (v(t, x) - v(t_m^n, x_\ell^n)) \right\} \nonumber \\
&\qquad \qquad\qquad \bigcap \left\{\sup_{(t, x) \in R_{k, l}^n}v(t, x) - v(t_m^n, x_\ell^n) \geq 2^{-n}\right\}\Bigg)\nonumber \\
& \qquad\qquad + \mbox{P}\left\{|u(t_m^n, x_\ell^n) - z| \leq (C+2)\, 2^{-n}\right\}. \label{eq2017-11-23-4}
\end{align}
For the first term in \eqref{eq2017-11-23-4}, we apply assumption \eqref{eq2017-11-23-5} with $\eta = 2^{-n}$
to see that it is equal to
\begin{align} \nonumber
& \int_{2^{-n}}^{\infty} \d z_2\int_{z - 2^{-n+1} - C z_2}^{z + 2^{-n+1} + C z_2} \d z_1\, p_+(z_1, z_2)   \leq c\int_{2^{-n}}^{\infty} \d z_2\int_{z - C z_2 - 2^{-n+1}}^{z + C z_2 + 2^{-n+1} }\d z_1 \, 2^{n}\, \exp\left(- \frac{z_2^2}{c\, 2^{-2n}}\right)  \\ 
 &\qquad\qquad \asymp \int_{2^{-n}}^{\infty} \d z_2\, (z_2 + 2^{-n}) \, 2^{n}\, \exp(- 2^{2n}\, z_2^2/c)= c'\, 2^{-n},
 \label{eq2017-11-23-6}
\end{align}
where, in the last equality, we have used the change of variables $z_2 = 2^{-n}\, \tilde{z}_2$.

Since the density of $u(t_m^n, x_\ell^n)$ is bounded uniformly for $(t_m^n, x_\ell^n) \in I \times J$ by hypothesis (2), the second term in \eqref{eq2017-11-23-4} is bounded above by $c'' \, 2^{-n}$, for all $n \in \N$. Together with \eqref{eq2017-11-23-4} and \eqref{eq2017-11-23-6}, we deduce that
\begin{align*}
   \mathcal{P}_1 \leq c2^{-n}.
\end{align*}
Similarly, $\mathcal{P}_2 \leq c2^{-n}$, so we conclude from \eqref{rd04_03e1} and \eqref{rd07_26e1} that
\begin{align*}
\mbox{P}\left\{\inf_{(t, x) \in R_{k, \ell}^n}|u(t, x) - z| \leq 2^{-n}\right\} \leq c2^{-n},
\end{align*}
which is \eqref{eq:smallp} and which implies \eqref{rd04_01e3}.

Now we use the estimate in \eqref{rd04_01e3} to prove Theorem \ref{th2017-11-23-1}, using the arguments in the proof of  \cite[Theorem 3.1]{DKN07}. Assume that $d - H_1^{-1} - H_2^{-1} > 0$ and $\mathscr{H}_{d - H_1^{-1} - H_2^{-1}}(A) < \infty$, otherwise there is nothing to prove.
Fix $\ep \in (0, 1)$ and $n \in \mathbb{N}$ such that $2^{-n - 1} < \ep \leq 2^{-n}$, and write
\begin{align*}
\mbox{P}\left\{U(I \times J) \cap B(z_0, \ep) \neq \emptyset \right\} \leq \sum_{(m, \ell): R_{m, \ell}^n \cap (I \times J) \neq \emptyset}\mbox{P}\left\{U(R_{m, \ell}^n) \cap B(z_0, 2^{-n}) \neq \emptyset \right\},
\end{align*}
where $B(z_0, \ep)$ denotes the Euclidean ball centered at $z_0$ with radius $\ep$.

 The number of pairs $(m, \ell)$ involved in the sum is at most $2^{n(H_1^{-1} + H_2^{-1})}$ times a constant. The bound \eqref{rd04_01e3} implies that for all $z \in A$ and large $n$,
\begin{align}\label{eq2017-08-18-50}
\mbox{P}\left\{U(I \times J) \cap B(z_0, \ep) \neq \emptyset \right\} &\leq \tilde{C}\, 2^{n(H_1^{-1} + H_2^{-1})} \, 2^{-nd}
 \leq C \ep^{d - H_1^{-1} - H_2^{-1}}.
\end{align}
Note that $C$ only depends on $I$ and $J$ (and does not depend on $n$ or $\ep$). Therefore, \eqref{eq2017-08-18-50} is valid for all $\ep \in (0, 1)$.

Now we use a covering argument: Choose $\tilde{\ep} \in (0, 1)$ and let $(B_i,\, i \geq 1)$ be a sequence of open balls in $\mathbb{R}^d$ with respective radii $r_i \in (0, \tilde{\ep})$ such that
\begin{align}\label{eq2017-08-18-60}
A \subseteq \bigcup_{i = 1}^{\infty}B_i \quad \mbox{and} \quad \sum_{i = 1}^{\infty}(2r_i)^{d - H_1^{-1} - H_2^{-1}} \leq \mathscr{H}_{d - H_1^{-1} - H_2^{-1}}(A) + \tilde{\ep}.
\end{align}
Because $\mbox{P}\left\{U(I \times J) \cap A \neq \emptyset \right\}$ is at most $\sum_{i = 1}^{\infty}\mbox{P}\left\{U(I \times J) \cap B_i \neq \emptyset \right\}$, the bounds in \eqref{eq2017-08-18-50} and \eqref{eq2017-08-18-60} together imply that
\begin{align*}
\mbox{P}\left\{U(I \times J) \cap A \neq \emptyset \right\} \leq C \sum_{i = 1}^{\infty} r_i^{d - H_1^{-1} - H_2^{-1}} \leq \tilde C \left(\mathscr{H}_{d - H_1^{-1} - H_2^{-1}}(A) + \tilde{\ep}\right).
\end{align*}
Letting $\tilde{\ep} \rightarrow 0^{+}$, we obtain \eqref{eq2017-11-23-11}. This completes the proof of Theorem \ref{th2017-11-23-1}.
\end{proof}

 \section{Proof of Theorem \ref{thm1}}\label{rd08_11s2}
 
\subsection{Controlling the supremum of $v$ by means of H\"older seminorms}

Choose an integer $p_0$ and a real number $\gamma_0$ such that
\begin{equation}\label{rd07_26e2}
p_0>\gamma_0 >4.
\end{equation}
Fix $\theta \in \left(0,\frac 12 \right)$ and set
\begin{align}\label{eq:theta12}
\theta_1=\frac 12-\theta, \qquad \theta_2=2\theta.
\end{align}
We assume that $p_0$ is large enough so that there exist $\gamma_1$ and $\gamma_2$ such that
\[
\frac 1{2p_0} <\gamma_1 <\frac {\theta_1}2 - \frac 1{2p_0}, \qquad
\frac 1{2p_0} <\gamma_2 <\frac {\theta_2}2 - \frac 1{2p_0}
\]
and
\begin{equation}\label{rd07_01e1}
2\gamma_1+\gamma_2=\frac {\gamma_0-1}{2p_0}.
\end{equation}

Let $(t_0, x_0) \in I \times J$.  For $r\ge 0$ and $z \in \R$, we define
 \begin{align*}
Y_1(r) &=\int_{[t_0,r]^2} \frac {(v(t,x_0)-v(s,x_0))^{2p_0}}{|t-s|^{\gamma_0/2}}\, \d s\d t, \\
Y_2(z)&=\int_{[x_0,z]^2} \frac {(v(t_0,x)-v(t_0,y))^{2p_0}}{|x-y|^{\gamma_0-2} }\, \d x\d y, \\
Y_3(r,z)&=\int_{[t_0,r]^2} \d s\d t \int_{[x_0,z]^2} \d x\d y\, \frac {(v(t,x)+v(s,y)-v(t,y)-v(s,x))^{2p_0}}{|t-s|^{1+2p_0\gamma_1}|x-y|^{1+2p_0\gamma_2}}
\end{align*}
 (here, the interval $[t_0, r]$ is defined as $[r, t_0]$ if $r<t_0$
 ).
 
\begin{lemma}\label{rd07_26l1}
(a) $Y_1(r)<\infty$ a.s.~and  for any $p\ge 1$,
\begin{align}\label{eq:moment1}
   \E[Y_1(r)^p] \lesssim  |r-t_0|^{p \left(2+ \frac { p_0-\gamma_0}2 \right)}.
\end{align}

(b)  $Y_2(z)<\infty$ a.s. and for any $p\ge 1$,
\begin{align}\label{eq:moment2}
   \E[Y_2(z)^p] \lesssim  |z-x_0|^{p(4+p_0-\gamma_0)}.
\end{align}

(c) $Y_3(r, z)<\infty$ a.s. and for any $p\ge 1$,
 \begin{align}\label{eq:moment3}
   \E[Y_3(r,z)^p]  &\lesssim |r-t_0|^{p(1+p_0(\theta_1-2\gamma_1))}\, |z-x_0|^{p(1+p_0(\theta_2-2\gamma_2))}.
\end{align}
 \end{lemma}

 \begin{proof}
 (a) In order to check that $Y_1(r)<\infty$ a.s., by the H\"older-continuity properties of $t \mapsto v(t, x_0)$, it suffices to notice that $p_0/2 > \gamma_0/2$ by \eqref{rd07_26e2}. Further,
 \begin{align*}
\E[Y_1(r)^p] \lesssim  |r-t_0|^{2(p-1)}\int_{[t_0,r]^2} \frac {|t-s|^{pp_0/2}}{|t-s|^{p\gamma_0/2}} \, \d s\d t
\le  |r-t_0|^{p(2+ \frac { p_0-\gamma_0}2)},
\end{align*}
 which is \eqref{eq:moment1}.

(b)  Similarly, $Y_2(z)<\infty$ a.s.~because of the H\"older-continuity properties of $x \mapsto u(t_0, x)$ and the fact that $p_0 - \gamma_0 + 2 > 0$ by \eqref{rd07_26e2}. Further,
\begin{align*}
\E[Y_2(z)^p] \lesssim |z-x_0|^{2(p-1)}\int_{[x_0,z]^2} \frac {|x-y|^{pp_0}}{|x-y|^{p(\gamma_0-2)}}\, \d s\d t
\le  |z-x_0|^{p(4+p_0-\gamma_0)},
\end{align*}
which is \eqref{eq:moment2}.

(c) Because  $v$ is Gaussian, for $p\geq1$, we have 
\begin{align}\label{eq:rectangle}
&\|v(t,x)+v(s,y)-v(t,y)-v(s,x)\|_p^2 \nonumber\\
&\qquad \asymp \|v(t,x)+v(s,y)-v(t,y)-v(s,x)\|_2^2 \lesssim |t-s|^{\theta_1}|x-y|^{\theta_2},
\end{align}
where the second inequality holds by the same arguments as in \cite[Lemma 3.1]{DP1} with $\theta_1, \theta_2$ given in \eqref{eq:theta12}. 
Thus, for any $p\ge 1$, 
\begin{align*}
   \E[Y_3(r,z)^p]  &\lesssim (|r-t_0||z-x_0|)^{2(p-1)}\nonumber \\
&\qquad\qquad \times \int_{[t_0,r]^2}\d s\d t \int_{[x_0,z]^2} \d x\d y\, \frac {|t-s|^{p_0 p \theta_1} |x-y|^{p_0p\theta_2}}{|t-s|^{p(1+2p_0\gamma_1)}|x-y|^{p(1+2p_0\gamma_2)}} \nonumber\\
&\le |r-t_0|^{p(1+p_0(\theta_1-2\gamma_1))}|z-x_0|^{p(1+p_0(\theta_2-2\gamma_2))},
\end{align*}
which is \eqref{eq:moment3}. This implies in particular that $Y_3(r, z) < \infty$ a.s.
\end{proof}

For $r\ge 0$ and $z \in \R$,  we set 
 \begin{equation}\label{rd04_04e1}
 Z_{r,z} = Z_{r, z; t_0, x_0} :=  Y_1(r) + Y_2(z) + Y_3(r,z).
 \end{equation}

\begin{proposition}\label{rd04_03p1}
There is a constant $ c>0$ such that for all $\deltah > 0$, $a>0$, $r\in [t_0,t_0+\deltah^2]$ and $z\in [x_0,x_0+\deltah]$,
 \begin{equation} \label{R}
 Z_{r,z} \le R:=c\,a^{2p_0}\, \deltah^{4-\gamma_0}  \Longrightarrow
 \sup_{(t,x)\in [t_0,r] \times [x_0, z]} |v(t,x)-v(t_0,x_0)| \le a.
 \end{equation}
\end{proposition}
 
 \begin{proof}
 By the triangle inequality,  
\begin{equation} \label{a1}
   \sup_{(t,x)\in [t_0,r] \times [x_0, z]} |v(t,x)-v(t_0,x_0)|  \le M_1(r)+M_2(z) +M_3(r, z),
\end{equation}
 where
\begin{align*}
M_1(r) &=   \sup_{t_0\le t \le r } \vert v(t,x_0)-v(t_0,x_0)\vert,   \\
M_2(z) &=    \sup _{x_0\le x \le z} \vert v(t_0,x)-v(t_0,x_0)\vert, \\
M_3(r, z) &= \sup _{\substack {t_0\le t \le r \\ x_0\le x \le z}} \vert v(t,x)-v(t_0,x)-v(t,x_0)+v(t_0,x_0)\vert.
\end{align*}
We will compare $M_1(r)$ with $Y_1(r)$, $M_2(z)$ with $Y_2(z)$ and $M_3(r, z)$ with $Y_3(r, z)$.

By the Garsia-Rodemich-Rumsey lemma \cite[Lemma A.3.1]{nualart2006} (with $p(t) = t^{\gamma_0/(4 p_0)}$, $\Psi(x) = x^{2 p_0}$, $d=1$), for $s,t\in [t_0,r]$ such that $t_0\le r\le t_0+\deltah^2$,
\begin{align*}
|v(t,x_0)-v(s,x_0)| & \lesssim |t-s|^{\frac {\gamma_0-4}{4p_0} } \, Y_1(r) ^{\frac 1{2p_0}} \le c\, \deltah^{\frac {\gamma_0-4}{2p_0} }
Y_1(r)^{\frac 1{2p_0}}.
\end{align*}
Taking $s=t_0$, we see that for all $t_0 \leq r \le t_0+\deltah^2$ and  $t\in [t_0,r]$,
   \begin{equation} \label{EQ2}
|v(t,x_0)-v(t_0,x_0)| \lesssim \deltah^{\frac {\gamma_0-4}{2p_0} } \,Y_1(r) ^{\frac 1{2p_0}}\quad\text{i.e.}\quad M_1(r) \lesssim \deltah^{\frac {\gamma_0-4}{2p_0} } \,Y_1(r) ^{\frac 1{2p_0}}.
\end{equation}

By the Garsia-Rodemich-Rumsey lemma \cite[Lemma A.3.1]{nualart2006} (with $p(t) = t^{(\gamma_0 - 2)/(2 p_0)}$, $\Psi(x) = x^{2 p_0}$, $d=1$),  assuming that $x_0\le z\le x_0+\deltah$, for $x,y\in [x_0,z]$, we obtain
\[
|v(t_0,x)-v(t_0,y)| \lesssim |x-y|^{\frac {\gamma_0-4}{2p_0}}  Y_2(z) ^{\frac 1{2p_0}} \lesssim\deltah^{\frac {\gamma_0-4}{2p_0}} 
Y_2(z) ^{\frac 1{2p_0}}.
\]
Taking $y=x_0$ and assuming that $x_0 \leq z \le x_0+\deltah$,  for all $x\in [x_0,z]$, we obtain
   \begin{equation} \label{EQ1}
|v(t_0,x)-v(t_0,x_0)| \lesssim \deltah^{\frac {\gamma_0-4}{2p_0}} \, Y_2(z) ^{\frac 1{2p_0}}\quad\text{i.e.}\quad M_2(z) \lesssim \deltah^{\frac {\gamma_0-4}{2p_0}} \, Y_2(z) ^{\frac 1{2p_0}}.
\end{equation}

 For any $z \in \R$, $p \geq 1$, $\gamma\in \left(1/(2p),1\right)$ and continuous function $f$ defined on $[x_0,z]$, we define
 \[
     \|f\|_{p,\gamma}:=\left( \int_{[x_0,z]^2} \frac{|f(x)-f(y)|^{2p}}{|x-y|^{1+2p\gamma}}\, \d x\d t \right)^{1/(2p)}.
 \]
 Then, for $r\in \R$,
 \begin{align*}
 Y_3(r,z)&=\int_{[t_0,r]^2}\d s\d t \int_{[x_0,z]^2} \d x\d y\, \frac {(v(t,x)-v(s,x) - (v(t,y) - v(s,y)))^{2p_0}}{|t-s|^{1+2p_0\gamma_1}|x-y|^{1+2p_0\gamma_2}}\\
 &=\int_{[t_0,r]^2} \frac { \| v(t,*) - v(s,*) \|^{2p_0}_{p_0,\gamma_2}}
 {|t-s|^{1+2p_0\gamma_1}}\, \d s\d t.
 \end{align*}
By the Garsia-Rodemich-Rumsey lemma \cite[Lemma A.3.1]{nualart2006} (with $p(t) = t^{(1 + 2 p_0 \gamma_1)/(2 p_0)}$, $\Psi(x) = x^{2 p_0}$, d=1), assuming that  $t_0\le r\le t_0+\deltah^2$, for $t,s\in [t_0,r]$, we obtain
 \begin{align*}
 \|v(t,*)-v(s,*)\|_{p_0,\gamma_2}& \lesssim |t-s|^{\gamma_1-\frac 1{2p_0}} \,
 Y_3(r,z)^{\frac 1{2p_0}}\le  \deltah^{2\gamma_1-\frac 1{p_0}}\, Y_3(r,z)^{\frac 1{2p_0}}.
 \end{align*}
 Putting $s=t_0$,  assuming that   $t_0\le r\le t_0+\deltah^2$, for $t\in [t_0,r]$,  this yields
  \[
 \|v(t,*)-v(t_0,*)\|_{p_0,\gamma_2} \lesssim  \deltah^{2\gamma_1-\frac 1{p_0}}\, Y_3(r,z)^{\frac 1{2p_0}}.
 \]
 Using again the Garsia-Rodemich-Rumsey lemma \cite[Lemma A.3.1]{nualart2006} (with $p(t) = t^{(1 + 2 p_0 \gamma_2)/(2 p_0)}$, $\Psi(x) = x^{2 p_0}$, assuming that  $t_0\le r\le t_0+\deltah^2$ and $x_0\le z\le x_0+\deltah$, for $t\in [t_0,r]$,  
 and $x,y \in [x_0,z]$,  we obtain
  \begin{align*}
 |v(t,x)-v(t_0,x)-v(t,y)+v(t_0,y)| &\lesssim  |x-y|^{\gamma_2-\frac1{2p_0}} \, \deltah^{2\gamma_1-\frac 1{p_0}}\,  Y_3(r,z)^{\frac 1{2p_0}} \\
 &\le  \deltah ^{\frac {\gamma_0-4}{2p_0}} \, Y_3(r,z)^{\frac 1{2p_0}} .
 \end{align*}
 Putting $y=x_0$, we get
   \begin{equation*} 
 |v(t,x)-v(t_0,x)-v(t,x_0)+v(t_0,x_0)| \lesssim
  \deltah ^{\frac {\gamma_0-4}{2p_0}} \, Y_3(r,z)^{\frac 1{2p_0}},
 \end{equation*}
 that is,
    \begin{equation} \label{EQ3}
    M_3(r, z) \lesssim \deltah ^{\frac {\gamma_0-4}{2p_0}} \, Y_3(r,z)^{\frac 1{2p_0}}.
 \end{equation}
 From \eqref{a1}--(\ref{EQ3}),  assuming that   $t_0 \le r\le t_0+\deltah^2$ and $x_0\le z\le x_0+\deltah$, we obtain
  \[
 \sup _{\substack {t_0\le t \le r \\
x_0\le x \le z}} |v(t,x)-v(t_0,x_0)| \lesssim     \deltah ^{\frac {\gamma_0-4}{2p_0}}\,
 Z_{r,z}^{\frac 1{2p_0}}.
 \]
 This establishes \eqref{R}.
 \end{proof}

 \subsection{Density formula for the random vector $F$}
 
 We begin by recalling some basic elements of Malliavin calculus, following \cite{nualart2006}. 
 We associate to the space-time white noise $\dot W$ an isonormal Gaussian process $(W(h),\, h \in \HH)$, where the Hilbert space $\HH = L^2(\R_+ \times \R)$ is endowed with its usual inner product $\langle \cdot, \cdot \rangle_\HH$, defined by
 \begin{align*}
      W(h) = \int_{\R_+ \times \R} h(t, x) W(\d t\, \d x). 
 \end{align*}
 The (Malliavin) derivative $D S$ of a smooth random variable 
 \begin{align*}
    S = f(W(h_1),\dots,W(h_n)), 
\end{align*}
where $f$ is a $C^\infty$ function such that $f$ and all of its derivatives have polynomial growth, is defined as in \cite[Section 1.2 \& 1.2.1]{nualart2006}, as well as higher order derivatives $D^k S$ and the seminorms
 \begin{align*}
    \Vert S \Vert_{k, p} = \left[ \E[\vert S \vert^p] + \sum_{j=1}^k  \E\left[\Vert D^j S \Vert_{\HH^{\otimes j}} \right]\right]^{\frac{1}{p}},
 \end{align*}
 $p \geq 1$, $k\geq 1$. The completion of the family of smooth random variables with respect to $\Vert \cdot \Vert_{k, p} $ is denoted $\bD^{k, p}$.
 
 The adjoint of the operator $D$, which we denote $\delta$, is an unbounded operator on $L^2(\R_+ \times \R \times \Omega)$. The Skorohod integral of an element $\tilde u$ of $L^2(\R_+ \times \R \times \Omega)$ is defined in \cite[Section 1.3.2]{nualart2006} and is denoted $\delta(\tilde u)$. If $\tilde u$ is adapted to the filtration generated by the space-time white noise, then the Skorohod integral is equal to the It\^o-Walsh integral:
 \begin{align*}
     \delta(\tilde u) = \int_{\R_+ \times \R} \tilde u(t, x) W(\d t\, \d x).
 \end{align*}
 
 Concerning the solution $(u(t, x))$ to \eqref{eq:SHE}, when $\sigma$ is $k$ times continuously differentiable, then $u(t, x) \in \bD^{k, p}$, for all $k \geq 1$, $p \geq 1$ and $(t, x) \in \R_+^* \times \R$ (see \cite[Theorem 2.4.3 \& Proposition 2.4.4]{nualart2006} and \cite[Proposition 4.1]{DKN09}).
 \medskip
 
 \noindent{\em The density of a random vector as an iterated Skorohod integral}
 \medskip
 
 Recall that $F = (F_1, F_2)$, where
 \[
     F_1=u(t_0,x_0),\qquad\ F_2=  \sup _{\substack {t_0\le t \le t_0+\deltah_1 \\ x_0\le x \le x_0+\deltah_2}}
(v(t,x)-v(t_0,x_0)),
\]
and $(t_0, x_0) \in I \times J$.
 In this subsection, we will provide a formula for the probability density function $p_F: \R \times \R_+ \to \R$ of $F$ (see Proposition \ref{rd04_04p1}) and for the probability density function of a more general random vector $H = (H^1,\dots, H^k)$ (see Remark \ref{rd07_02r1}). We first introduce some notation.

Suppose that $I = [s_0, s_1]$ and fix $\rho>0$ strictly smaller than $s_0 \wedge \frac 9 2$ (the $\frac 9 2$ is used in the beginning of the proof of Lemma \ref{lem:gamma_11}).  Let $T_1 \in \R$ be such that $T_1 > s_1$, and
 let  $f:\R\to [0,1]$ be the function defined by 
\begin{equation}\label{rd05_21e1}
     f(r)= \frac{ r-t_0+\rho}{\rho}\, 1_{(t_0 - \rho, t_0)}(r) + 1_{[t_0, T_1)}(r) + (2 - \frac{r}{T_1})\, 1_{[T_1, 2T_1)}(r).
 \end{equation}
  Fix a large real number $K$ such that $J \subset [1-K, K-1]$. Let $g:\R\to [0,1]$  be the continuously differentiable function defined by
 \begin{align} \nonumber
 g(z) &= \frac 34 \e^{\frac 23(z+K)}\, 1_{(-\infty, -K]}(z) + \left(1 - [z-1+K]^2/4 \right) 1_{(-K, 1-K]}(z) \\ \nonumber
    &\qquad + 1_{(1-K, K-1]}(z) + \left(1- [z-K+1]^2/4 \right) 1_{(K-1, K]}(z) \\
    &\qquad + \frac 34\e^{-\frac 23 (z-K)}\, 1_{(K, \infty)}(z).
\label{rd05_21e2}
 \end{align}
 Notice in particular that $f = 1$ on $I$, $g=1$ on $J$, $f$ is continuous and piecewise $C^1$ and $g$ is $C^1$ and piecewise $C^2$.

 Recall that $\HH=L^2(\R_+\times \R)$.
  Define the $\HH$-valued random variable $u_{A,1}$ which, evaluated at $(r,z)$, is given by
 \begin{equation} \label{uA1}
 u_{A,1}(r,z)= \left ( \frac \partial{\partial r} -\frac 12\frac {\partial^2}{\partial z^2} \right)  [f(r)g(z)],
 \end{equation}
 where the derivatives are well-defined except at a few points.

 Let $\psi_0:\mathbb{R}_+ \rightarrow [0,1]$ be an infinitely differentiable nonincreasing function such that
 \[
 \psi_0(x)= 1 \quad\text{if } x \in [0, 1/2] \qquad\text{and}\qquad \psi_0(x)= 0 \quad\text{if } x \geq 1.
 \]
 
 For $z_2\in (0,\infty)$, set 
 \begin{equation}\label{rd04_04e2}
      a=z_2/2\qquad \text{and}\qquad A=\R \times (a,\infty).
 \end{equation}
 Let $\zeta > 0$ and
 \begin{equation}\label{rd08_04e2}
    R=R(z_2,\deltah) := c \left(\frac{z_2}{2}\right)^{2 p_0}\, \zeta^{4 - \gamma_0}
  \end{equation}
 be defined as in (\ref{R}) for $a=z_2/2$ (in Proposition \ref{rd04_04p1} below, we will take $\zeta=\max(\zeta_1^{1/2}, \zeta_2)$).
 
 Define $\psi(x) = \psi_{z_2,\deltah}(x) := \psi_0(x/R(z_2, \deltah))$, $x\in \mathbb{R}$ so that
  \[
 \psi(x)= 1 \quad\text{if } x \in [0, R/2], \qquad \psi(x)= 0 \quad\text{if } x \geq R,
 \]
 and 
 \begin{equation}\label{rd07_02e1}
    \|\psi'\|_\infty \le cR^{-1}
\end{equation}
for a certain constant $c$ not depending on $z_2$ or $\deltah$.

Let $Z_{r, z}$ be defined as in \eqref{rd04_04e1}, and for $r\geq0$ and $z\in \R$, define $H_A(r, z) = H_A(r, z; z_2, \deltah)$ by
 \begin{align}\label{eq:H}
 H_A(r,z)= \int_{t_0}^{t_0+\deltah^2}\int_{x_0}^{x_0+\deltah}\psi_{z_2, \deltah}(Z_{s,y})\, d y\d s- \int_{r}^{t_0+\deltah^2}\int_{z}^{x_0+\deltah}\psi_{z_2, \deltah}(Z_{s,y})\, \d y\d s. 
 \end{align}
 Let $\phi$ an infinitely differentiable function with support on $[-1,2]$ such that $\phi(x)=1$ for $x \in [0,1]$.
Define 
\begin{align}\label{eq:eta}
\eta(r, z)= \phi\left((r-t_0)/\deltah^2\right) \phi\left((z-x_0)/\deltah\right), \quad r\geq0,\ z\in \R,
\end{align}
and define $u_{A,2}(r, z) = u_{A,2}(r, z; z_2, \deltah)$ by
\begin{align}\label{rd07_01e2}
 u_{A,2}(r, z)= \left(\frac{\partial }{\partial r}-\frac 12 \frac{\partial^2}{\partial z^2}\right)[\eta(r,z)H_A(r,z)], \quad r\geq0,\ z\in \R.
 \end{align} 
 
 Finally, define a $2 \times 2$ matrix $ \gamma_A = \gamma_A(t_0,x_0; z_2, \deltah)$ by
 \[
 \gamma_A = \left(\begin{array}{cc} 
         \gamma_{A} ^{1,1} & \gamma_{A} ^{1,2} \\ [7pt]
         \gamma_{A} ^{2,1} & \gamma_{A} ^{2,2}\end{array} \right),
 \]
 where, letting $D$ denote the Malliavin derivative,
 \begin{align}\notag
 \gamma_{A} ^{1,1} &:= \int_0^{t_0}\int_\R D_{r,z} u(t_0,x_0)\,  ( f'(r) g(z)- \frac 12f(r) g''(z) )\, \d z \d r, \\
  \gamma_{A} ^{1,2} &:=  \int_0^{t_0}\int_\R   D_{r,z} u(t_0,x_0)\, u_{A,2}(r,z)\, \d z \d r,  \notag \\
   \gamma^{2,1}_A &:= 0, \notag \\
 \gamma^{2,2}_A &:= \int_{t_0}^{t_0+\deltah^2}\int_{x_0}^{x_0+\deltah}\psi(Z_{s,y})\, \d y\d s.
 \label{rd08_04e1}
 \end{align}
 Note that $ \gamma_{A} ^{1,2}$ and $ \gamma_{A} ^{2,2}$ depend on $z_2$  and $\deltah$ (via the functions $\eta$ and $\psi = \psi_{z_2, \deltah}$).
  
   \begin{proposition}\label{rd04_04p1}
  Let $F = (F_1, F_2)$ be the random vector defined in \eqref{rd07_30e1}. Let $z_2 \in (0, \infty)$, set $\zeta=\max\{\zeta_1^{1/2}, \zeta_2\}$ and, using $z_2$ and this value of $\zeta$, define
  \[
     G_1=\delta\left (  (\gamma_{A} ^{2,2})^{-1} u_{A,2}   - (\gamma_{A} ^{2,2})^{-1}   \gamma_{A} ^{1,2}\,
       ( \gamma_{A} ^{1,1})^{-1}u_{A,1} \right).
  \]
For $z_1 \in \R$, the probability density function $p_F$ of $F$ at $(z_1, z_2)$ is given by
  \begin{equation} \label{for9}
    p_F(z_1,z_2)= \E \left[ \mathbf{1}_{\{F_1>z_1\}} \mathbf{1}_{\{F_2>z_2\}}\, \delta( u_{A,1}
  (\gamma_A^{1,1})^{-1}   G_1) \right],
 \end{equation}
  and also by the alternate equivalent formula 
 \begin{equation} \label{for10}
      p_F(z_1,z_2)= -\E \left[ \mathbf{1}_{\{F_1<z_1\}} \mathbf{1}_{\{F_2>z_2\}} \, \delta( u_{A,1}
  (\gamma_A^{1,1})^{-1}   G_1)\right].
 \end{equation}
Here, $\delta$ denotes the Skorohod integral, and the Skorohod integrals that appear in these formulas are well-defined by Proposition \ref{prop1}.
\end{proposition}

\begin{remark}\label{rd07_02r1}
For a random (column) vector $H= (H^1,\dots,H^k)$ in $\R^k$, and a $k \times k$ matrix $M$, we obtain another $k \times k$ matrix $M \cdot H$, where ``$\cdot$" denotes the matrix product. We define the iterated integral $I_k(H)$ by
\begin{equation}\label{rd04_07e1}
I_k(H) = \delta(H^1 \delta (H^2 \cdots \delta(H^k) \cdots)).
\end{equation}
Interpret $DH = (DH^1,\dots,DH^k)$ as a column vector in $\R^k$, that is, a $k \times 1$ matrix.

Let $A \subset R^k$ be an open set.  Suppose that there exist a random (column) vector $u_A = (u_A^1,\dots, u_A^k) \in (\bD^\infty(\HH))^k$ and a  $k \times k$ random matrix $(\gamma_A^{i, j})$ such that $\vert \text{det } \gamma_A \vert^{-1} \in L^p(\Omega)$, for all $p\geq 1$, and $\langle DH^i, u_A^j \rangle_\HH = \gamma_A^{i, j}$ on $\{H \in A \}$, for $i, j = 1,\dots, k$.
Formula \eqref{for9}  is a special case of the generic formula
\begin{equation}\label{rd04_06e1}
     p_H(z_1,\dots,z_k) = \E\left[1_{\{H^1 > z_1,\dots, H^k > z_k \}}\, I_k( (\gamma_A^{-1})^\rT \cdot  u_A)\right],
\end{equation}
for $(z_1,\dots,z_k) \in A$, which extends \cite[(2.34)]{nualart2006} to locally nondegenerate random vectors. Alternatively, for $\ell = 1,\dots, k$,
\begin{equation*}
     p_H(z_1,\dots,z_k) = (-1)^\ell\, \E\left[1_{\{H^1 < z_1,\dots, H^\ell < z_\ell, H^{\ell +1} > z_{\ell + 1},\dots, H^k > z_k  \}}\, I_k( (\gamma_A^{-1})^\rT \cdot  u_A)\right].
\end{equation*}

\end{remark}
 
 \begin{proof}[Proof of Proposition \ref{rd04_04p1}]
We determine the density of $F$ on the set $A$ introduced in \eqref{rd04_04e2}.  We carry out most of the proof without using the fact that $\gamma_A^{2, 1} = 0$, so as to obtain \eqref{rd04_06e1} for $k=2$. Then we particularize to the case where $\gamma_A^{2, 1} = 0$, so as to obtain \eqref{for9} and \eqref{for10}.

 For the first component $F_1$, by (\ref{uA1}),
 \begin{align} \nonumber
&  \langle DF_1, u_{A,1} \rangle_\HH =\int_0^{t_0}\int_\R D_{r,z} u(t_0,x_0) u_{A,1}(r, z)\, \d z \d r\\
  &\qquad=\int_0^{t_0}\int_\R D_{r,z} u(t_0,x_0) \label{for2}
 ( f'(r) g(z)- \frac 12f(r) g''(z) )\,  \d z \d r = \gamma_{A} ^{1,1}.
  \end{align}
   Moreover,
 \begin{equation} \label{for4a}
   \langle DF_1, u_{A,2} \rangle_\HH=
    \int_0^{t_0}\int_\R   D_{r,z} u(t_0,x_0) u_{A,2}(r,z) \d z \d r=  \gamma_{A} ^{1,2}.
    \end{equation}

Clearly, $F_2>0$ a.s.~by \cite[(3.9)]{Kho14}. Moreover, since $\E[(v(t,x)-v(s,y))^2]\neq 0$ for $(t,x)\neq (s,y)$,
 by \cite[Lemma 2.6]{KiP}, $t \mapsto v(t, x) - v(t_0, x_0)$ attains its maximum in $[t_0, t_0 + \deltah_1]\times[x_0, x_0 + \deltah_2]$ at a unique random point
 $(S,X)\in [t_0,t_0+\deltah_1]\times [x_0,x_0+\deltah_2]$.
  By the same argument as in \cite[Lemma 4.3]{DP1}, we obtain
 \begin{align*}
   D_{r,z} F_2 &= G(S-r,X-z) \mathbf{1}_{\{r<S\}}- G(t_0-r, x-x_0) \mathbf{1}_{\{r<t_0\}}.
 \end{align*}
Therefore,
 \begin{align*}
 \langle DF_2, u_{A,2} \rangle_\HH& = \int_0^{\infty}\int_\R\left[G(S-r, X-z)\bm{1}_{\{r<S\}} -  G(t_0-r, x_0-z)\bm{1}_{\{r<t_0\}} \right] u_{A,2}(r, z)\, \d z \d r\\
 &= \eta(S, X)\, H_A(S, X) - \eta(t_0, x_0)\, H_A(t_0, x_0) \\
 &= \int_{t_0}^{t_0+\deltah^2}\int_{x_0}^{x_0+\deltah}\psi(Z_{s,y})\, \d y\d s- \int_{S}^{t_0+\deltah^2}\int_{X}^{x_0+\deltah}\psi(Z_{s,y})\, \d y\d s,
 \end{align*}
 where the second equality holds by Lemma \ref{lem:heat} below and the third equality holds because $\eta(S, X) = 1$ and $H_A(t_0, x_0) = 0$. 
 
 We claim that on the set $ \{F_2>a\}$, if $S\leq s \leq t_0+\deltah^2$ and $X\leq y\leq x_0+\deltah$, then $\psi(Z_{s,y})=0$. Indeed, if $\psi(Z_{s,y}) > 0$,
 then $Z_{s,y}\leq R = R(z_2, \deltah)$ and by \eqref{R}, this implies that 
 \begin{align*}
 a<F_2= v(S,X)-v(t_0, x_0) \leq  \sup_{t_0\leq t\leq s, x_0\leq x\leq y}(v(t,x)-v(t_0, x_0)) \leq a,
 \end{align*}
 which is a contradiction. Therefore, on the set $\{F\in A\}$,  
 \begin{align}\label{for3a}
  \langle DF_2, u_{A,2}\rangle_{\mathcal{H}} = \int_{t_0}^{t_0+\deltah^2}\int_{x_0}^{x_0+\deltah}\psi(Z_{s,y})\d y\d s = \gamma^{2,2}_A,
 \end{align}
    and
   \begin{align}\nonumber
    \langle DF_2, u_{A,1} \rangle_\HH& =\int_0^{\infty}\int_\R \left[G(S-r, X-z)\bm{1}_{\{r<S\}} -  G(t_0-r, x_0-z)\bm{1}_{\{r<t_0\}} \right]  \\   \nonumber 
   &\qquad \qquad\quad \times \left(\frac{\partial }{\partial r}-\frac 12 \frac{\partial^2}{\partial z^2}\right)[f(r)g(z)] \d z \d r\\
   &= f(S)g(X) -f(t_0)g(x_0)= 1-1=0 = \gamma_A^{2,1}, \label{for1}
   \end{align}
   where the second  equality holds by Lemma \ref{lem:heat} below. 
  
  With these ingredients, we can now proceed to derive \eqref{for9} and \eqref{for10}.   Let $h$ be an infinitely differentiable function with compact support on $\R^2$. Set
  \begin{equation}\label{rd04_07e2}
  \varphi(y_1,y_2)= \int_{-\infty}^{y_1} \d z_1 \int_{-\infty}^{y_2} \d z_2 \, h(z_1,z_2) .
  \end{equation}
  On $\{F\in A\}$, using (\ref{for1}) and (\ref{for2}), we see that
  \begin{equation*}
     \langle D\partial_1 \varphi(F), u_{A,1}\rangle_\HH= \partial_{11} \varphi(F)\, \gamma_{A}^{1,1} + \partial_{21} \varphi(F)\, \gamma_{A}^{2,1}.
  \end{equation*}
  Similarly,
   \begin{equation*}
     \langle D\partial_1 \varphi(F), u_{A,2}\rangle_\HH= \partial_{11} \varphi(F)\, \gamma_{A}^{1,2} + \partial_{21} \varphi(F)\, \gamma_{A}^{2,2}.
   \end{equation*}
  We can rewrite this in matrix form
    \[
    \left(\begin{array}{cc}
    \langle D\partial_1 \varphi(F), u_{A,1}\rangle_\HH \\ \langle D\partial_1 \varphi(F), u_{A,2}\rangle_\HH
           \end{array} \right) = \gamma_A^\rT \cdot \left(\begin{array}{cc}
           \partial_{11} \varphi(F) \\  \partial_{21} \varphi(F) \end{array} \right),
   \]
  where $\gamma_A^\rT$ denotes the transpose of $\gamma_A$.   Since $\gamma_A$ is invertible a.s. (see Lemmas \ref{lem:gamma_11} and \ref{lem:gamma22} below), it follows that
    \[
   \left(\begin{array}{cc}
           \partial_{11} \varphi(F) \\  \partial_{21} \varphi(F) \end{array} \right) = (\gamma_A^\rT)^{-1} \cdot 
            \left(\begin{array}{cc} 
            \langle D\partial_1 \varphi(F), u_{A,1}\rangle_\HH \\ \langle D\partial_1 \varphi(F), u_{A,2}\rangle_\HH
           \end{array} \right).
 \]
As a consequence,
 \begin{equation}\label{rd05_20e1}
   \partial_{21} \varphi(F) =  (\gamma_{A}^{-1})^{1,2}\langle D\partial_1 \varphi(F), u_{A,1}\rangle_\HH +  (\gamma_{A}^{-1})^{2,2} \langle D\partial_1 \varphi(F), u_{A,2}\rangle_\HH.
\end{equation}

       For $z_2 > 0$, let $\kappa_{z_2}:\R\rightarrow [0,1]$ be an infinitely differentiable function such that $\kappa_{z_2}(x)=0$ for $x\le \frac {2z_2}3$ and $\kappa_{z_2}(x)=1 $ for $x\ge \frac {3z_2}4$. Define $\bar{\kappa}(y_1,y_2)=\kappa_{z_2}(y_2)$
   and $H=\bar{\kappa}(F)$. Multiplying both sides of (\ref{rd05_20e1}) by $H$, taking expectations on both sides and using the duality relationship between the derivative and the divergence operator \cite[(1.42) p.~37]{nualart2006} (the fact that this relationship can be applied is justified at the beginning of the proof of Proposition \ref{prop1}) yields
\begin{align} \nonumber 
   \E[H h(F)] &=  \E\left[H\, \partial_{2 1} \varphi(F)\right] \\ \nonumber
   &=   \E\left[H\,\left( (\gamma_{A}^{-1})^{1,2}\, \langle D\partial_1 \varphi(F), u_{A,1}\rangle_\HH +  (\gamma_{A}^{-1})^{2,2}\, \langle D\partial_1 \varphi(F), u_{A,2}\rangle_\HH\right)\right] \\ \nonumber
   &= \E\left[\left\langle D\partial_1 \varphi(F), H\left( (\gamma_{A}^{-1})^{1,2}\, u_{A,1} + (\gamma_{A}^{-1})^{2,2}\,  u_{A,2}\right)\right\rangle_\HH\right]\\ 
    & =\E \left[\partial_1 \varphi(F)\,  \delta\left (  H\left( (\gamma_{A}^{-1}) ^{1,2}\, u_{A,1}  +  (\gamma_{A}^{-1}) ^{2,2}\, u_{A,2} \right) \right)\right].
    \label{for7a}
 \end{align}
 
  We denote
\begin{equation}\label{rd04_07e3}
   G_0=\delta\left ( \left((\gamma_{A}^{-1}) ^{1,2} \, u_{A,1}  + (\gamma_{A}^{-1}) ^{2,2}\, u_{A,2} \right) H \right).
\end{equation}
We notice that $H=1$ on $\{F_2\ge \frac 34 z_2\}$, and since $H=0$ on the set  $\{F \not \in A\}$, the local property of the divergence operator \cite[Proposition 1.3.15]{nualart2006} implies that $G_0=0$ on $\{F \not \in A\}$.
    
 By the chain rule, \eqref{for2} and \eqref{for1},
\begin{align*}
    \langle D\varphi(F), u_{A,1} \rangle_\HH & = \partial_1 \varphi(F)\, \langle  DF_1, u_{A,1} \rangle_\HH + \partial_2\varphi(F) \, \langle  DF_2, u_{A,1} \rangle_\HH \\
     &=  \gamma_A^{1,1}\,  \partial_1 \varphi(F) + \gamma_A^{2,1}\,  \partial_2\varphi(F) .
 \end{align*}
Similarly, by \eqref{for4a} and \eqref{for3a}, on the set $\{F\in A\}$,
\begin{align*}
    \langle D\varphi(F), u_{A,2} \rangle_\HH & = \partial_1 \varphi(F)\, \langle  DF_1, u_{A,2} \rangle_\HH + \partial_2\varphi(F) \, \langle  DF_2, u_{A,2} \rangle_\HH \\
     &=  \gamma_A^{1,2}\,  \partial_1 \varphi(F) + \gamma_A^{2,2}\,  \partial_2\varphi(F) .
 \end{align*}
 In matrix form, these two equalities become
   \[
    \left(\begin{array}{cc}
    \langle D \varphi(F), u_{A,1}\rangle_\HH \\ \langle D  \varphi(F), u_{A,2}\rangle_\HH
           \end{array} \right) = \gamma_A^\rT \cdot \left(\begin{array}{cc}
           \partial_{1} \varphi(F) \\  \partial_{2} \varphi(F) \end{array} \right).
   \]
Equivalently, 
   \[  
  \left(\begin{array}{cc}
           \partial_{1} \varphi(F) \\  \partial_{2} \varphi(F) \end{array} \right) = (\gamma_A^\rT)^{-1} \cdot 
           \left(\begin{array}{cc}  \langle D \varphi(F), u_{A,1}\rangle_\HH \\ \langle D  \varphi(F), u_{A,2}\rangle_\HH
           \end{array} \right).
   \]
This implies that
\begin{equation} \label{for8a}
 \partial_1\varphi(F) = (\gamma_A^{-1})^{1,1} \langle D\varphi(F), u_{A,1} \rangle_\HH + (\gamma_A^{-1})^{2,1} \langle D\varphi(F), u_{A,2} \rangle_\HH\, .
 \end{equation}
 Substituting (\ref{for8a}) into (\ref{for7a}), we obtain
 \begin{align*}
   \E[H h(F)] &=
   \E \left[ \left((\gamma_A^{-1})^{1,1} \langle D\varphi(F), u_{A,1} \rangle_\HH + (\gamma_A^{-1})^{2,1} \langle D\varphi(F), u_{A,2} \rangle_\HH\ \right)  G_0 \right] \\
   & = \E \left[ \left\langle D\varphi(F), \left((\gamma_A^{-1})^{1,1} u_{A,1} + (\gamma_A^{-1})^{2,1} u_{A,2} \right) G_0 \right\rangle_\HH \right] \\
    & = \E\left[ \varphi(F) \delta \left(\left( (\gamma_A^{-1})^{1,1}  u_{A,1} + (\gamma_A^{-1})^{2,1} u_{A,2}\right) G_0\right) \right].
   \end{align*}
Using formula \eqref{rd04_07e3}, this equality can be written
   \[  
    \E[H \partial_{12} \varphi(F)] = \E\left[ \varphi(F) \delta \left([(\gamma_A^{-1})^\rT \cdot u_A ]_1\, \delta\left([(\gamma_A^{-1})^\rT \cdot u_A ]_2\, H \right) \right) \right],
  \]
  where
\[  
         (\gamma_A^{-1})^\rT \cdot u_A = \left(\begin{array}{c}
           \left((\gamma_A^{-1})^\rT \cdot u_A \right)_1 \\  [6pt]
           \left((\gamma_A^{-1})^\rT \cdot u_A \right)_2 \end{array} \right) \qquad\text{and}\qquad u_A = \left(\begin{array}{c} u_{A,1} \\  u_{A,2} \end{array} \right).
\]

 By the definition of $H$ and of $\varphi$ in \eqref{rd04_07e2} and Fubini's theorem,
  \begin{align}\nonumber
      \E[\bar{\kappa}(F) h(F)] &=  \int_{-\infty}^\infty \d z_1 \int_{-\infty}^\infty \d z_2\, h(z_1, z_2) \\
       &\qquad \times \E\left[1_{\{F_1 > z_1\}}\, 1_{\{F_2 > z_2 \}}\,  \delta \left([(\gamma_A^{-1})^\rT \cdot u_A ]_1\, \delta\left([(\gamma_A^{-1})^\rT \cdot u_A ]_2\, H \right) \right)\right].
\label{rd04_08e4}
\end{align} 
Recalling that $H =1$ on the set  $\{F\in \mathbb{R}\times [\frac 34 z_2,\infty)\}$, using the local property of $\delta$ and the notation in \eqref{rd04_07e1}, equality \eqref{rd04_08e4} becomes
 \begin{equation}\label{rd04_08e5}
      \E[\bar{\kappa}(F) h(F)] =   \int_{-\infty}^\infty \d z_1 \int_{-\infty}^\infty \d z_2\, h(z_1, z_2)  \E\left[1_{\{F_1 > z_1\}}\, 1_{\{F_2 > z_2 \}}\,  I_2\left((\gamma_A^{-1})^\rT\cdot u_A \right) \right].
 \end{equation}
If the support of $h$ is contained in $\R \times [\frac{3}{4} z_2, \infty)$, then $h(F) > 0$ implies $\bar{\kappa}(F) = 1$, so for such $h$, \eqref{rd04_08e5} becomes
\[
     \E[h(F)] =   \int_{-\infty}^\infty \d z_1 \int_{-\infty}^\infty \d z_2\, h(z_1, z_2) \, \E\left[1_{\{F_1 > z_1\}}\, 1_{\{F_2 > z_2 \}}\,   I_2\left((\gamma_A^{-1})^\rT \cdot u_A \right) \right].
\]
In particular, 
the density of $F$ at $(z_1,z_2) \in \R \times \R_+^*$ is
\begin{equation}\label{rd07_28e1}
 p(z_1,z_2)= \E \left[ \mathbf{1}_{\{F_1>z_1\}} \mathbf{1}_{\{F_2>z_2\}}\, I_2\left((\gamma_A^{-1})^\rT \cdot u_A \right) \right],
\end{equation}
 which is \eqref{rd04_06e1} for $k=2$. 
 
 We now particularize formula \eqref{for9} to the case where $\gamma_A^{2, 1} = 0$. In this case,
 \[
     \gamma_A^{-1} = \left(\gamma_A^{1, 1}\, \gamma_A^{2, 2} - \gamma_A^{1, 2}\, \gamma_A^{2, 1}\right)^{-1} \left(\begin{array}{cc}
         \gamma_A^{2, 2} & - \gamma_A^{1, 2}  \\[7pt]
          -\gamma_A^{2, 1} &  \gamma_A^{1, 1}  \end{array}  \right) 
          = \left(\gamma_A^{1, 1}\, \gamma_A^{2, 2} \right)^{-1} \left(\begin{array}{cc}
         \gamma_A^{2, 2} & - \gamma_A^{1, 2}  \\[7pt]
          0 &  \gamma_A^{1, 1}  \end{array}  \right) ,
 \]
that is,
  \begin{align*}
     (\gamma_A^{-1})^{1,1} &= (\gamma_A^{1,1})^{-1},\quad (\gamma_A^{-1})^{1,2} = - (\gamma_A^{2,2})^{-1}\, \gamma_A^{1,2}\,(\gamma_A^{1,1})^{-1},\\
      (\gamma_A^{-1})^{2,1} &= 0,\qquad\qquad (\gamma_A^{-1})^{2,2} = (\gamma_A^{2,2})^{-1}.
  \end{align*}
 In particular, defining $I_2(\cdot)$ as in \eqref{rd04_07e1}, we have
 \[
    I_2\left((\gamma_A^{-1})^\rT\cdot u_A \right) = \delta\left( (\gamma_A^{1,1})^{-1}\, u_{A, 1}\,  \delta\left( (\gamma_A^{2,2})^{-1}\, u_{A, 2} - (\gamma_A^{2,2})^{-1}\, \gamma_A^{1,2}\,(\gamma_A^{1,1})^{-1} u_{A, 1}\right) \right).
 \] 
  Together with \eqref{rd07_28e1}, this establishes \eqref{for9}. 
  
 Replacing $\varphi$ in \eqref{rd04_07e2} by
 \[
     \tilde\varphi(y_1, y_2) =  \int_{y_1}^{\infty} \d z_1 \int_{-\infty}^{y_2} \d z_2 \, h(z_1,z_2),
 \]
 we would obtain the alternate formula
 \[
 p(z_1,z_2)= - \E \left[ \mathbf{1}_{\{F_1<z_1\}} \mathbf{1}_{\{F_2>z_2\}}\, I_2\left(\left(\gamma_A^{-1}\right)^{\rT} \cdot u_A \right) \right],
\]
which would lead to \eqref{for10}. This completes the proof of Proposition \ref{rd04_04p1}.
 \end{proof} 


 The following lemma was used in the proof of Proposition \ref{rd04_04p1}.
 
\begin{lemma}\label{lem:heat}
Suppose that $(r, z) \mapsto \tilde{H}(r, z)$ is a continuous function on $[0, \infty) \times \R$ that has one generalized derivative in $r$ and two  in $z$, and that
there exists a constant $C>0$ such that for all $(r, z)\in [0, \infty)\times\R$,
\begin{align*}
  \max\left[ |\tilde{H}(r, z)|, \left|\frac{\partial \tilde{H}(r, z)}{\partial r}\right|, \left|\frac{\partial \tilde{H}(r, z)}{\partial z}\right|, 
\left|\frac{\partial^2 \tilde{H}(r, z)}{\partial z^2}\right| \right] \leq C\e^{-C|z|}. 
\end{align*}
 Then for every $t>0$ and $x\in \R$, 
\begin{align*}
\int_0^t\int_\R G(t-r, x-z)\left(\frac{\partial }{\partial r}-\frac12\frac{\partial^2}{\partial z^2}\right)\tilde{H}(r,z)\, \d z\d r= \tilde{H}(t,x).
\end{align*}
\end{lemma}

\begin{proof}
Let $\varphi \in C_0^\infty((0, T) \times \R)$ and let $\LL$ denote the heat operator: 
\[
     \LL \varphi (r, z)= \frac{\partial }{\partial r} \varphi(r, z) -\frac12\frac{\partial^2}{\partial z^2} \varphi(r, z). 
\]
Then
\begin{align*}
    \int_0^T \d r \int_\R \d z\,  \int_r^T \d t \int_\R \d x\, \vert \LL \tilde H(r,z) \varphi(t,x) G(t-r, x-z) \vert < \infty,
\end{align*}
so we can apply Fubini's theorem to see that
\begin{align}
\label{coronilla(*1)rd}
 & \int_0^\infty \d t \int_\R \d x\, \varphi(t,x) \left(\int_0^t \d r \int_\R \d z\, G(t-r, x-z) \LL \tilde H(r,z)\right)  \notag\\
 &\qquad = \int_0^T \d r \int_\R \d z\,  \LL \tilde H(r,z) \left(\int_r^T \d t \int_\R \d x\, \varphi(t,x) G(t-r, x-z) \right).
  \end{align}   
Because $\psi$ has compact support, and vanishes near $0$ and $T$ and at $x = \pm \infty$, we can integrate by parts (twice in $x$, and once in $t$), with no contribution of boundary terms, to see that \eqref{coronilla(*1)rd} is equal to
\begin{align*}
    &\int_0^\infty \d t \int_\R \d x\,  \tilde H(t,x)\, \LL \left(\int_0^t \d s \int_\R \d y\, G(t-s, x-y) \varphi(s,y)\right)\\
    &\qquad =  \int_0^\infty \d t \int_\R \d x\, \tilde H(t,x) \varphi(t,x),
    \end{align*}
where the equality holds since the space-time convolution between heat kernel $G$ and $\varphi$ satisfies the inhomogenuous heat equation. This implies that for a.a.~$(t, x)$, $\tilde H(t, x)$ is equal to 
\begin{align*}
    \int_0^t \d r \int_\R \d z\,  G(t-r, x-z) \LL \tilde H(r,z) .
\end{align*}
Since both functions are continuous in $(t, x)$, the lemma is proved.
\end{proof}

  \subsection{Gaussian-type upper bounds on the density of $F$}
  
  First, 
 we apply twice the Cauchy-Schwarz inequality in \eqref{for9} to obtain for any $z_1\in \R$ and $z_2 > 0$,
\begin{align} \nonumber
   p_F(z_1,z_2) &\le      \P\{F_1 > z_1,  F_2>z_2\}^{1/2} \, \| \delta( u_{A,1} (\gamma_A^{1,1})^{-1}     G_1)\|_2 \\
      &\leq P\{F_1 > z_1 \}^{1/4}\, P\{F_2>z_2 \}^{1/4}\,  \| \delta( u_{A,1} (\gamma_A^{1,1})^{-1}     G_1)\|_2.  
 \label{for5}
 \end{align}
 Similarly, we can use \eqref{for10} instead to see that
 \begin{align*}
    p_F(z_1,z_2) &\le P\{F_1 < z_1 \}^{1/4}\, P\{F_2>z_2 \}^{1/4}\,  \| \delta( u_{A,1} (\gamma_A^{1,1})^{-1}     G_1)\|_2. 
  \end{align*}
 For $z_1 > 0$, $F_1 > z_1$ implies $\vert F_1\vert > \vert z_1\vert$, and for $z_1 < 0$,  $F_1 < z_1$ implies that $\vert F_1\vert > \vert z_1\vert$, so we combine this last inequality with \eqref{for5} to see that
  \begin{align}\label{rd08_08e1}
    p_F(z_1,z_2) &\le P\{\vert F_1\vert > \vert z_1\vert \}^{1/4}\, P\{F_2>z_2 \}^{1/4}\,  \| \delta( u_{A,1} (\gamma_A^{1,1})^{-1}     G_1)\|_2. 
  \end{align}

The following proposition provides an upper bound for the factor $\|\delta( u_{A,1}
  (\gamma_A^{1,1})^{-1}     G_1)\|_2$.

\begin{proposition} \label{prop1}
Assume that $z_2 \ge \deltah^{1/2}$. Then,
\begin{equation}
\label{equ1}
\|  \delta( u_{A,1} (\gamma_A^{1,1})^{-1}    G_1)\|_2 \lesssim \deltah^{-\frac 12}.
\end{equation}
\end{proposition}

\begin{proof}
In this proof, we assume for the moment the results of Section \ref{rd07_02s1}.

Recall that $u(t_0, x_0) \in \bD^{1,2}$ because $\sigma$ is Lipschitz \cite[Proposition 2.4.4]{nualart2006}. One easily deduces from \cite[Proposition 4.1]{DKN09} that $u(t_0, x_0) \in \bD^{3,2}$ because $\sigma \in C^3(\R)$, and since $\sigma$ and its first three derivatives are bounded, 
\begin{align}\label{rd08_06e3}
   \E\left[\Vert D u(t_0, x_0) \Vert_\HH^p\right], \quad \E\left[\Vert D^2 u(t_0, x_0) \Vert_{\HH\otimes 2}^p\right] \quad\text{and}\quad \E\left[\Vert D^3 u(t_0, x_0) \Vert_{\HH\otimes 3}^p \right]
\end{align}
are uniformly bounded over $(t_0, x_0) \in I \times J$.

In order to check that $u_{A,1} (\gamma_A^{1,1})^{-1} G_1$ belongs to the domain of $\delta$, we observe that Lemma \ref{lem:Dgamma12} below implies that $\gamma_A^{1, 2} \in \bD^{2, 2}$. In addition, because $p_0 > 4$, we have $Z_{r, z} \in \bD^{4, p}$, for all $p \geq 1$, and therefore $\gamma_A^{2, 2} \in \bD^{4,p}$. Together with Lemma \ref{lem:gamma22}, this shows that $(\gamma_A^{2, 2})^{-1} \in \bD^{4,2}$. Further, by Lemmas \ref{lem:uA2mom}--\ref{D2u_A2} below, $u_{A, 2} \in \bD^{2,2}$. Because $u_{A, 1}$ is deterministic, we have $u_{A, 1} \in \bD^{1, 2}$. These considerations, together with \cite[Propositions 1.3.3 and 1.3.8]{nualart2006}, show that $G_1$ is well-defined and $u_{A,1} (\gamma_A^{1,1})^{-1} G_1$ belongs to the domain of $\delta$.

In order to compute the divergence  $\delta( u_{A,1}
  (\gamma_A^{1,1})^{-1}    G_1)$, we use the property \cite[Proposition 1.3.3]{nualart2006} of the Skorohod integral to see that
\begin{align*}
   \delta( u_{A,1} (\gamma_A^{1,1})^{-1}  G_1) &=  \delta(u_{A,1})\, (\gamma_A^{1,1})^{-1}\,  G_1 - \left\langle u_{A,1}, D((\gamma_A^{1,1})^{-1}  G_1) \right\rangle_\HH \\
     & = \delta(u_{A,1})\, (\gamma_A^{1,1})^{-1}\,  G_1
  - G_1 \left\langle u_{A,1} ,D\left[(\gamma_A^{1,1})^{-1}\right]  \right\rangle_\HH \\
   &\qquad \qquad- (\gamma_A^{1,1})^{-1} \left\langle u_{A,1} ,D G_1 \right\rangle_\HH\\
  &= \delta(u_{A,1})\, (\gamma_A^{1,1})^{-1}\,  G_1
  + (\gamma_A^{1,1})^{-2}\, G_1\left\langle u_{A,1} ,D\gamma_A^{1,1}  \right\rangle_\HH \\
   &\qquad\qquad - (\gamma_A^{1,1})^{-1} \langle u_{A,1} ,D G_1 \rangle_\HH.
  \end{align*}
  Therefore, by H\"older's inequality (and since $u_{A,1}$ is deterministic),
  \begin{align} \nonumber
  \left\| \delta( u_{A,1}
  (\gamma_A^{1,1})^{-1}    G_1)\right\|_2 &\lesssim  \| \delta(u_{A,1})\|_6 \left \| (\gamma_A^{1,1})^{-1} \right\|_6\| G_1\|_6 \\  \nonumber
  &\qquad + \left\| (\gamma_A^{1,1})^{-2} \right\|_6\| G_1\|_6 \left\| u_{A,1} \right\|_\HH \left\| D\gamma_A^{1,1}\right\|_{L^6(\Omega; \HH)}\\ 
  & \qquad  + \left\| (\gamma_A^{1,1})^{-1} \right\|_4 \left\| \langle u_{A,1} ,D G_1 \right\rangle_\HH\|_4.
 \label{for16}
  \end{align}
 Again because $u_{A,1}$ is deterministic, the divergence $\delta(u_{A,1})$ is Gaussian, and therefore, for any $p\ge 2$,
 \[
  \|\delta(u_{A,1})\|_p \le c_p \| u_{A,1}\|_\HH.
  \]
  From the definition of $u_{A,1}$ in \eqref{uA1}, it follows that $\| u_{A,1}\|_\HH \le C$, which implies that
  \begin{equation} \label{for12}
   \|\delta(u_{A,1})\|_p \le C.
  \end{equation}
  Observe that the random variable $\gamma^{1,1}_A$ satisfies
  \begin{equation} \label{for20}
  \|  D\gamma_A^{1,1}\|_{L^p(\Omega;\HH)} \le C
  \end{equation}
  for all $p\ge 2$. Indeed,
  \begin{align*}
   \left\|  D\gamma_A^{1,1} \right\|_\HH &=  \left\| D \int_0^{t_0} \int_\R D_{r,z} u(t_0,x_0) [ f'(r)g(z)-\frac 12 f(r)g''(z)] \, \d z \d r \right\|_\HH\\
  &\lesssim \| D^2u(t_0,x_0)\|_{\HH^{\otimes2} }
  \end{align*}
by \eqref{rd08_06e3}.   Similarly, the uniform boundedness of the moments of $ \| D^3u(t_0,x_0)\|_{\HH^{\otimes3}}$ implies that
    \begin{equation} \label{for202}
  \|  D^2\gamma_A^{1,1}\|_{L^p(\Omega;\HH^{\otimes 2})} \le C.
  \end{equation}

  Using the estimates (\ref{for12}), (\ref{for13}) and (\ref{for20}) in the inequality (\ref{for16}) yields
\[
  \| \delta( u_{A,1}
  (\gamma_A^{1,1})^{-1}    G_1)\|_2 \lesssim    \| G_1\|_6 +   
\| \langle u_{A,1} ,D G_1 \rangle_\HH\|_4 .
\]

In order to conclude the proof, it suffices to show that for $z_2\geq \zeta^{\frac12}$
  \begin{equation} \label{for17}
  \| G_1\|_6 \lesssim \deltah^{-\frac 12}
  \end{equation}
  and
  \begin{equation} \label{for18}
    \left\| \langle u_{A,1} ,D G_1 \rangle_\HH\right\|_4 \lesssim \deltah^{-\frac 12}.
  \end{equation}
 
 \medskip
 \noindent
 {\it Proof of (\ref{for17}):} 
Using the property \cite[Proposition 1.3.3]{nualart2006} of the Skorohod integral, the random variable $G_1$ can be expressed as follows:
  \begin{align*}
  G_1&=  (\gamma_{A} ^{2,2})^{-1}\, \delta(u_{A,2})+(\gamma_{A} ^{2,2})^{-2}  \left\langle u_{A,2} , D\gamma_A^{2,2} \right\rangle_\HH \\
  &\quad
   - (\gamma_{A} ^{2,2})^{-1} \,  \gamma_{A} ^{1,2} \,
     ( \gamma_{A} ^{1,1})^{-1}\, \delta(u_{A,1} )+\left \langle u_{A,1} ,D\left[(\gamma_{A} ^{2,2})^{-1}   \gamma_{A} ^{1,2}
     ( \gamma_{A} ^{1,1})^{-1} \right] \right\rangle_\HH.
  \end{align*}
  From (\ref{for7}), the last term simplifies and it follows that
  \begin{align}
  G_1&=  (\gamma_{A} ^{2,2})^{-1}\,  \delta(u_{A,2})+(\gamma_{A} ^{2,2})^{-2}  \left\langle u_{A,2} ,D\gamma_A^{2,2} \right\rangle_\HH \nonumber\\ 
  &\qquad
   - (\gamma_{A} ^{2,2})^{-1}  \, \gamma_{A} ^{1,2}\, 
     ( \gamma_{A} ^{1,1})^{-1}\, \delta(u_{A,1} ) 
     + (\gamma_{A} ^{2,2})^{-1} \left\langle u_{A,1} ,D\left[  \gamma_{A} ^{1,2}\,
     ( \gamma_{A} ^{1,1})^{-1} \right] \right\rangle_\HH.
  \label{rd05_20e2}
  \end{align}
  As a consequence,
    \begin{align*}
  \|G_1\|_6&\leq  \left\|(\gamma_{A} ^{2,2})^{-1}\right\|_{12} \left\| \delta(u_{A,2})\right\|_{12} 
   +\left\| (\gamma_{A} ^{2,2})^{-2}\right\|_{18} \left\| u_{A,2}\right\| _{L^{18}(\Omega;\HH)}
 \left \|   D\gamma_A^{2,2}\right\| _{L^{18}(\Omega;\HH)} \\
  &\qquad
   + \left\|(\gamma_{A} ^{2,2})^{-1} \right\|_{24} \left\|  \gamma_{A} ^{1,2}\right\|_{24}
  \left \|  ( \gamma_{A} ^{1,1})^{-1}\right\|_{24} \left\|\delta(u_{A,1} )\right\|_{24}\\
   &\qquad + \left\|  (\gamma_{A} ^{2,2})^{-1} \left\langle u_{A,1} ,D\left[  \gamma_{A} ^{1,2}
     ( \gamma_{A} ^{1,1})^{-1} \right] \right\rangle_\HH \right\|_6.
  \end{align*}
  Using the estimates (\ref{equ2}),  (\ref{for21}),  (\ref{equ10}), (\ref{Dgamma})  (\ref{A12}), (\ref{for13})  and  (\ref{for12}), we can write
      \begin{align*}
  \|G_1\|_6  \lesssim  \deltah^{-3} \deltah^{\frac 52} +\deltah^{-6} \deltah^{\frac 52} \deltah^3   
   +\deltah^{-3} \left\|     \langle u_{A,1} ,D[  \gamma_{A} ^{1,2}
     ( \gamma_{A} ^{1,1})^{-1} ] \rangle_\HH \right\|_{12}.
  \end{align*}
  From (\ref{for13}), (\ref{DA12}),  (\ref{A12}), (\ref{for20})  and \eqref{for13}, it follows that
  \[
  \|G_1\|_6 \lesssim \deltah^{-3} \deltah^{\frac 52}  + \deltah^{-3}  \deltah^{\frac 52}   \lesssim \deltah^{-\frac 12},
  \] 
  which completes the proof of  (\ref{for17}).
    \medskip
    
 \noindent
 {\it Proof of (\ref{for18}):} 
  From \eqref{rd05_20e2} and using \cite[Proposition 1.3.8]{nualart2006}, we have
  $$
  DG_1 = Y_1 + \cdots + Y_4,
  $$
  where
  \begin{align*}
    Y_1&=  -(\gamma_{A} ^{2,2})^{-2} \delta(u_{A,2}) D\gamma_A^{2,2}+(\gamma_{A} ^{2,2})^{-1}  u_{A,2} +
    (\gamma_{A} ^{2,2})^{-1}  \delta(Du_{A,2} ), \\[6pt]
    Y_2&= -2
    (\gamma_{A} ^{2,2})^{-3}  \left\langle u_{A,2} ,D\gamma_A^{2,2} \right\rangle_\HH\, D\gamma_A^{2,2} 
    + (\gamma_{A} ^{2,2})^{-2}   \left\langle Du_{A,2} (*),D_*\gamma_A^{2,2} \right\rangle_\HH \\
    &\qquad\qquad +
    (\gamma_{A} ^{2,2})^{-2}   \langle u_{A,2} (*),DD_*\gamma_A^{2,2} \rangle_\HH , \\[6pt]
  Y_3& =  +(\gamma_{A} ^{2,2})^{-2} \,  \gamma_{A} ^{1,2}\,
     ( \gamma_{A} ^{1,1})^{-1}\delta(u_{A,1} )\,  D\gamma_A^{2,2} 
     -(\gamma_{A} ^{2,2})^{-1}  
     ( \gamma_{A} ^{1,1})^{-1}\delta(u_{A,1} )\,  D\gamma_{A} ^{1,2} \\
     &\qquad\qquad + ( \gamma_{A} ^{2,2})^{-1}\, \gamma_{A} ^{1,2}\, ( \gamma_{A} ^{1,1})^{-2}\, \delta(u_{A,1} ) \, D\gamma_A^{1,1} 
      -(\gamma_{A} ^{2,2})^{-1}   \gamma_{A} ^{1,2}
     ( \gamma_{A} ^{1,1})^{-1}u_{A,1} \\
     &\qquad\qquad -     (\gamma_{A} ^{2,2})^{-1}   \gamma_{A} ^{1,2}
     ( \gamma_{A} ^{1,1})^{-1} \, \delta(D u_{A,1}) , \\[6pt]
   Y_4  &= - (\gamma_{A} ^{2,2})^{-2} \left\langle u_{A,1} ,D\left[  \gamma_{A} ^{1,2}
     ( \gamma_{A} ^{1,1})^{-1} \right] \right\rangle_\HH \, D\gamma_A^{2,2} 
     + (\gamma_{A} ^{2,2})^{-1} 
    \left \langle Du_{A,1}(*) ,D_*[  \gamma_{A} ^{1,2}
     ( \gamma_{A} ^{1,1})^{-1} ] \right\rangle_\HH \\
    &\qquad\qquad +(\gamma_{A} ^{2,2})^{-1} 
    \left \langle u_{A,1}(*) ,DD_*[  \gamma_{A} ^{1,2}
     ( \gamma_{A} ^{1,1})^{-1} ] \right\rangle_\HH .
  \end{align*}
In the above , ``$*$" denotes the variable involved in the inner product. 

 Clearly,
\begin{equation}\label{rd07_29e1}
   \langle u_{A,1}, DG_1\rangle_\HH = \langle u_{A, 1}, Y_1 \rangle_\HH + \cdots + \langle u_{A, 1}, Y_4 \rangle_\HH.
\end{equation}
The derivative $D u_{A, 1}$ vanishes because $u_{A, 1}$ is deterministic, and in view of  (\ref{for7}),  when we project $DG_1$ on $u_{A,1}$, some of the terms vanish (in particular, those containing $D\gamma_A^{2,2}$ and $D_*\gamma_A^{2,2} $). 
We evaluate the four terms in \eqref{rd07_29e1} using in the first equality the stochastic Fubini's theorem for Skorohod integrals (see \cite[Proposition 6.5]{KRT}, for instance) to write $\langle u_{A, 1}(*), \delta(D_* u_{A, 2}) \rangle_\HH =  \delta(\langle u_{A,1}(*), D_{*}u_{A,2}\rangle_{\HH}  )$: 
\begin{align*}
    \langle u_{A,1}, Y_1\rangle_\HH&=   (\gamma_{A} ^{2,2})^{-1}  \langle u_{A,1}, u_{A,2} \rangle_\HH 
     -(\gamma_A^{2,2})^{-1} \delta(\langle u_{A,1}(*), D_{*}u_{A,2}\rangle_{\HH}  ) ,\\[6pt]
\langle u_{A, 1}, Y_2 \rangle_\HH & =(\gamma_A^{2,2})^{-2}\langle u_{A,1}, \langle Du_{A,2}(*), D_*\gamma_A^{2,2}\rangle_\HH\rangle_\HH, \\[6pt]
   \langle u_{A, 1}, Y_3 \rangle_\HH &=  -(\gamma_{A} ^{2,2})^{-1} 
     ( \gamma_{A} ^{1,1})^{-1}\delta(u_{A,1} )  \left\langle u_{A,1},  D\gamma_{A} ^{1,2} \right\rangle_\HH\\
     &\quad+ ( \gamma_{A} ^{2,2})^{-1}\gamma_{A} ^{1,2}( \gamma_{A} ^{1,1})^{-2}  \delta(u_{A,1} ) \left\langle u_{A,1}, D\gamma_A^{1,1} \right\rangle_\HH 
      -(\gamma_{A} ^{2,2})^{-1}   \gamma_{A} ^{1,2}
     ( \gamma_{A} ^{1,1})^{-1} \|u_{A,1} \|^2_\HH, \\[6pt]
    \langle u_{A, 1}, Y_4 \rangle_\HH   &= (\gamma_{A} ^{2,2})^{-1} \left \langle u_{A,1}\otimes u_{A1} ,D^2[  \gamma_{A} ^{1,2}
     ( \gamma_{A} ^{1,1})^{-1} ] \right\rangle_{\HH\otimes \HH} .
  \end{align*}
 By Lemmas \ref{lem:gamma22} and \ref{lem:uA2mom},
         \begin{align*}
         \|   (\gamma_{A} ^{2,2})^{-1}  \langle u_{A,1}, u_{A,2} \rangle_\HH\|_4\lesssim \deltah^{-3+5/2}=\deltah^{-1/2},
         \end{align*}
and by Lemmas \ref{lem:gamma22} and \ref{lem:uA1uA2},
         \begin{align*}
         \|(\gamma_A^{2,2})^{-1}\delta(\langle u_{A,1}(*), D_{*}u_{A,2}\rangle_{\HH})\|_4\lesssim \deltah^{-3+5/2}=\deltah^{-1/2},
         \end{align*}
 therefore $ \| \langle u_{A,1}, Y_1\rangle_\HH  \|_4 \lesssim \deltah^{-1/2}$. By Lemmas \ref{lem:gamma22} and \ref{lem:uA2gamma},
         \begin{align*}
         \| (\gamma_A^{2,2})^{-2}\langle u_{A,1}, \langle Du_{A,2}(*), D_*\gamma_A^{2,2}\rangle_\HH\rangle_\HH \|_4
         \lesssim \deltah^{-6+11/2}=\deltah^{-1/2},
         \end{align*}
therefore,  $ \| \langle u_{A,1}, Y_2\rangle_\HH  \|_4 \lesssim \deltah^{-1/2}$.    By Lemmas \ref{lem:gamma22}, \ref{lem:Dgamma12} and \ref{lem:gamma_11},
         \begin{align*}
         \|(\gamma_{A} ^{2,2})^{-1} 
     ( \gamma_{A} ^{1,1})^{-1}\delta(u_{A,1}  \langle u_{A,1},  D\gamma_{A} ^{1,2} \rangle_\HH)\|_4          \lesssim \deltah^{-3+5/2}=\deltah^{-1/2},
         \end{align*}
 by Lemmas \ref{lem:gamma22}, \ref{lem:Dgamma12}, \ref{lem:gamma_11} and \eqref{for12}--\eqref{for20},
        \begin{align*}
        \| ( \gamma_{A} ^{2,2})^{-1}\gamma_{A} ^{1,2}( \gamma_{A} ^{1,1})^{-2}\, \delta(u_{A,1}\, \langle u_{A,1}, D\gamma_A^{1,1}\rangle_\HH )\|_4\lesssim \deltah^{-3+5/2}=\deltah^{-1/2},
        \end{align*}
and by  Lemmas \ref{lem:gamma22}, \ref{lem:Dgamma12} and \ref{lem:gamma_11},
        \begin{align*}
        \|(\gamma_{A} ^{2,2})^{-1}   \gamma_{A} ^{1,2}
     ( \gamma_{A} ^{1,1})^{-1} \|u_{A,1} \|^{2}_\HH\|_4 \lesssim  \deltah^{-3+5/2} =\deltah^{-1/2},
        \end{align*}  
therefore,  $ \| \langle u_{A,1}, Y_3 \rangle_\HH  \|_4 \lesssim \deltah^{-1/2}$.    By  (\ref{A12})--(\ref{DDA12}), and using the facts that $\Vert u_{A, 1} \Vert_\HH$ is bounded and $(\gamma_A^{1,1})^{-1}$ has bounded moments in $\mathbb{D}^{2,p} $ for all $p\ge 2$ (see Lemma \ref{lem:gamma_11} below, \eqref{for20} and \eqref{for202}), we see that
 \begin{align*}
   \left\| (\gamma_{A} ^{2,2})^{-1} \left \langle u_{A,1}\otimes u_{A1} ,D^2[  \gamma_{A} ^{1,2}
     ( \gamma_{A} ^{1,1})^{-1} ] \right\rangle_{\HH\otimes \HH} \right\|_4 \lesssim  \deltah^{-3 +5/2}  =\deltah^{-1/2},
 \end{align*}
 therefore, $ \| \langle u_{A,1}, Y_4 \rangle_\HH  \|_4 \lesssim \deltah^{-1/2}$.  

   Together with \eqref{rd07_29e1}, these estimates  imply \eqref{for18}.
  \end{proof}

 \medskip
  \noindent
 {\it Proof of Theorem \ref{thm1}}. 
 Taking into account  the inequality (\ref{rd08_08e1}) and the bound (\ref{equ1}),  it suffices to estimate the probabilities $\P\{\vert F_1 \vert > \vert z_1 \vert\}$ and
 $\P\{F_2> z_2\}$. Because $\sigma$ is bounded, the existence of a positive constant $c_1$ such that $\P\{\vert F_1 \vert > \vert z_1 \vert \} \leq c_1 \exp(-z_1^2/c_1)$, for all $z_1 \in \R$, follows directly from the exponential martingale inequality \cite[Appendix A.2, (A.5)]{nualart2006}. We also have
 \begin{align*}
 \E[F_2] & \le \E \left[  \sup_{\substack {t_0\le t \le t_0+\deltah_1 \\
x_0\le x \le x_0+\deltah_2}}
 |v(t,x) - v(t_0,x_0)| \right]   \\
 & \le  \E \left[  \sup_{ |t-t_0|^{1/2}+|x-x_0| \le 2\deltah}
 |v(t,x) - v(t_0,x_0)| \right] 
  \le c\,  \deltah^{1/2},
 \end{align*}
 where the last inequality holds by \cite[Lemma 4.5]{DKN07}.
 Define 
 \[
 \sigma^2 = \sup_{\substack {t_0\le t \le t_0+\deltah_1 \\
x_0\le x \le x_0+\deltah_2}} \E[ (v(t,x)- v(t_0,x_0)^2] \le C\deltah.
\]
 Applying Borell's inequality \cite[(2.6)]{Adl90}, for all $z_2 \geq  \deltah^{1/2}$,
 \begin{align*}
 \P [F_2>z_2] & \le 2 \exp\left(-(z_2-\E[F_2])^2/(2\sigma^2)\right) \le 2 \exp\left(-(z_2-\E[F_2])^2/(2C\deltah)\right)\\
 & \le 2 \exp\left(-\left(2z^2_2/3-2\E[F_2]^2\right)/(2 C\deltah)\right) \\
 &= 2 \exp\left (-z^2_2/(3C\deltah^2) \right) \exp \left(\E[F_2]^2/( C\deltah)\right) \\
 & \le 2 \e^{c^2/C}\exp\left (-z^2_2/(3C\deltah) \right) 
 = c_1 \exp\left (-z^2_2/(3C\deltah) \right),
 \end{align*}
 This completes the proof.
 \hfill $\Box$

  \section{Technical results}\label{rd07_02s1}
  In this section, we state and prove several  lemmas that provide upper bounds for the $L^p$-norm of various terms, all of which were used in the proof of Proposition  \ref{prop1}. 
 
 \begin{lemma}\label{lem:gamma_11}
 The random variable $\gamma^{1,1}_A$ defined in \eqref{rd08_04e1}  satisfies
 \begin{equation} \label{for13}
\sup_{(t_0, x_0)\in I\times J} \E\left[\vert\gamma_A^{1,1}\vert^{-p}\right] <\infty,
\end{equation}
 for all $p\ge 2$.
 \end{lemma}
 
 \begin{proof}
 On the one hand, we know that $D_{r,z}u(t_0,x_0)\ge 0$ a.s.~(see \cite[(A.4)]{CKNP23}\footnote{\cite{CKNP23} considers the initial data $u(0)\equiv1$ and the result of (A.4) also applies to $u(0)\equiv0$.}). On the other hand, let $\rho$ and $f$ be as in \eqref{rd05_21e1} and let $K$ and $g$ be as in \eqref{rd05_21e2}.  For $0\leq r < t_0 - \rho$, we have $f(r) = 0$, and for $r \in [t_0-\rho, t_0]$ and $z \in  (1-K, K-1)$,  we have $f'(r)=\frac{1}{\rho}$, $g(z) = 1$ and $g''(z) = 0$. Therefore, for these $(r, z)$,
 \begin{align*}
     f'(r)g(z)-\frac 12 f(r) g''(z) 
 &= \frac 1 \rho 
     \ge 0.
  \end{align*}
For $r \in [t_0-\rho, t_0]$ and $\vert z \vert \in (K-1, K)$, $g''(z) = -1/2$, so $f'(r)g(z)-\frac 12 f(r) g''(z) \geq \frac 1 \rho g(z) \geq 0$. Finally, when  $r \in [t_0-\rho, t_0]$ and $\vert z \vert > K$, $g''(z) = \frac 4 9 g(z)$, so $f'(r)g(z)-\frac 12 f(r) g''(z) = \left( \frac 1 \rho - \frac 2 9 \right) g(z) \geq 0$ because we have chosen $\rho < \frac 9 2$. W conclude that the integrand in the definition of $\gamma_A^{1,1}$ is always nonnegative.

It follows that
\begin{align*}
         \gamma_A^{1,1}&= \int_0^{t_0}\int_\R D_{r,z} u(t_0,x_0) \left( f'(r) g(z)- \frac 12f(r) g''(z)\right) \d z \d r\\
         &\geq  \frac1\rho \int_{t_0-\rho}^{t_0}\int_{1-K}^{K-1} D_{r,z} u(t_0,x_0)\, \d z\d r. 
\end{align*}
By \cite[Proposition 2.4.4]{nualart2006},  the Malliavin derivative of $u$ satisfies  the following equation:
\begin{align*}
         D_{r,z}u(t_0\,,x_0)&= G(t_0-r, x_0-z)\sigma(u(r\,,z))\\
          &\qquad + \int_r^{t_0}\int_\R G(t_0-\theta, x_0-\eta)\sigma'(u(\theta\,, \eta))
         D_{r,z}u(\theta\,, \eta)W(\d\theta\, \d \eta),
 \end{align*}
so for all $0<\varepsilon<\rho$,
\begin{align}\label{gamma11-1}
         \gamma_A^{1,1}& \geq \frac1 \rho \int_{t_0-\varepsilon}^{t_0}\int_{1-K}^{K-1} D_{r,z} u(t_0,x_0)\d z\d r = I_{1, \varepsilon} + I_{2, \varepsilon},
\end{align}
where
\begin{align*}
        I_{1, \varepsilon} &= \frac1\rho \int_{t_0-\varepsilon}^{t_0}\int_{1-K}^{K-1} G(t_0-r, x_0-z)\sigma(u(r\,,z))\d z\d r , \\
     I_{2, \varepsilon}   &= \frac1\rho \int_{t_0-\varepsilon}^{t_0}\int_{1-K}^{K-1}\Bigg[\int_r^{t_0}\int_\R G(t_0-\theta, x_0-\eta)\sigma'(u(\theta\,, \eta)) \\
      &\qquad\qquad\qquad\qquad\qquad\qquad
       \times  D_{r,z}u(\theta\,, \eta)W(\d\theta\, \d \eta)\Bigg]\d z\d r.
\end{align*}
We can assume that $K$ is large enough so that the interval $[x_0-(K-1), x_0+ K -1)]$ contains $[0, 1]$, for all $x_0\in J$.  By \eqref{rd07_02e7}, there exists a constant $c>0$ such that for all $0<\varepsilon<\rho/2$,
\begin{align}\label{I1}
         I_{1, \varepsilon} &\geq c \int_{t_0-\varepsilon}^{t_0}\int_{1-K}^{K-1} G(t_0-r, x_0-z)\d z \d r \nonumber\\
         &\geq c \int_0^{\varepsilon} \int_0^1G(r, z) \d z \d r = c\, \varepsilon \int_0^{1} \int_0^{1/\sqrt{\varepsilon}}
           \sqrt{\ep}\, G(\varepsilon\, s, y\, \sqrt{\varepsilon})\d y\d s\nonumber\\
         &= c\, \varepsilon  \int_0^{1} \int_0^{1/\sqrt{\varepsilon}}
         G(s, y)\d y\d s \asymp \varepsilon,
\end{align}
where the second equality holds because of the identity \eqref{eq:scale}.
         
We proceed to estimate the moments of $I_{2, \varepsilon}$, $0 < \varepsilon < 1$. First, by the stochastic Fubini's theorem \cite[Theorem 2.4.1]{dss}, we write
\begin{align*}
      I_{2, \varepsilon} &= \int_{t_0-\varepsilon}^{t_0}\int_\R G(t_0-\theta, x_0-\eta) \\
        &\qquad \times \left[\int_{t_0-\varepsilon}^{\theta}\int_{1-K}^{K-1}  \sigma'(u(\theta\,, \eta))D_{r,z}u(\theta\,,\eta)\d z \d r 
         \right] W(\d \theta\, \d \eta).
\end{align*}
Since $\sigma'$ is bounded, by Burkholder's inequality, then Minkowsky's inequality,
\begin{align*}
         \|I_{2, \varepsilon}\|_p^2& \lesssim 
         \int_{t_0-\varepsilon}^{t_0}\int_\R G^2(t_0-\theta, x_0-\eta) \\
       &\qquad \times \left\|\int_{t_0-\varepsilon}^{\theta}\int_{-K+1}^{K-1}
         \sigma'(u(\theta\,, \eta))D_{r,z}u(\theta\,,\eta) \, \d z \d r  \right\|_p^2\d \eta\d \theta\\
         &\lesssim  \int_{t_0-\varepsilon}^{t_0}\int_\R G^2(t_0-\theta, x_0-\eta) 
          \left[\int_{t_0-\varepsilon}^{\theta}\int_{-K+1}^{K-1}
         \|D_{r,z}u(\theta\,,\eta)\|_p \, \d z \d r  \right]^2\d \eta\d \theta.
 \end{align*}
Lemma A.1 of \cite{HNV20}\footnote{\cite{HNV20} considers the initial data $u(0)\equiv1$ and the result of Lemma A.1 also applies to $u(0)\equiv0$.} ensures that
\begin{align*}
         \|D_{r,z}u(\theta\,,\eta)\|_p \lesssim G(\theta-r, \eta-z).
\end{align*}
Thus, it follows that
\begin{align}
                  \|I_{2, \varepsilon}\|_p^2& \lesssim
                    \int_{t_0-\varepsilon}^{t_0}\int_\R G^2(t_0-\theta, x_0-\eta) 
         \left[\int_{t_0-\varepsilon}^{\theta}\int_{-K+1}^{K-1}
         G(\theta-r, \eta-z)\, \d z \d r \right]^2\d \eta\d \theta\nonumber\\
         &\lesssim             \int_{t_0-\varepsilon}^{t_0}\int_\R G^2(t_0-\theta, x_0-\eta)\,
        (\theta - t_0 + \ep)^2 \, \d \eta\d \theta\nonumber\\
         &\leq \varepsilon^2               \int_{t_0-\varepsilon}^{t_0}\int_\R G^2(t_0-\theta, x_0-\eta) \,
         \d \eta\d \theta 
          \asymp \varepsilon^{5/2},
  \label{I2}
\end{align}
where we have used \cite[(A.5)]{bp1998}. 
         Hence, we conclude from \eqref{gamma11-1}, \eqref{I1} and \eqref{I2} that for all $(t_0\,,x_0)\in I\times J$ and 
         for all $0<\varepsilon<\rho$,
         \begin{align*}
         \gamma_A^{1,1}\geq c'\varepsilon - |I_{2, \varepsilon}|,
         \end{align*}
         where 
         \begin{align*}
         \|I_{2, \varepsilon}\|_p \lesssim \varepsilon^{5/4}, \qquad 0 < \varepsilon < 1.
         \end{align*}
Therefore, we deduce from \cite[Proposition 3.5]{DKN09} that for all $p\geq 2$, \eqref{for13} holds.
 \end{proof}

\begin{lemma} \label{lem:gamma22}
The random variable $\gamma_A^{2,2}$ has finite negative moments of all orders. Moreover, for $z_2 \ge \deltah^{1/2}$ and $p\ge 1$, we have
 \begin{equation}   \label{equ2}
 \| (\gamma_A^{2,2})^{-1}\|_{p} \lesssim \deltah^{-3}.
 \end{equation}
 \end{lemma}
 
 \begin{proof}
 Let $R$ be as in \eqref{rd08_04e2}.  We can write 
 \begin{align*}
 \gamma_A^{2,2} \geq  \int_{t_0}^{t_0+\deltah^2}\int_{x_0}^{x_0+\deltah}\bm{1}_{\{Z_{s,y}\leq R/2\}}\d y\d s=:\tilde{X}.
 \end{align*}
 Then, for $0<\epsilon \leq \deltah^3$, since $(s, y) \mapsto Z_{s, y}$ is a nondecreasing function of $s$ and $y$, we have
 \begin{align*}
 \P\{\tilde{X}<\epsilon\} \leq \P\left\{Z_{t_0+\epsilon^{2/3},\, x_0+\epsilon^{1/3}}\geq R/2\right\}\lesssim R^{-q}\, \E\left[Z_{t_0+\epsilon^{2/3},\, x_0+\epsilon^{1/3}}^q\right],
 \end{align*}
 for all $q > 0$. We see from \eqref{eq:moment1}, \eqref{eq:moment2}, \eqref{eq:moment3}, \eqref{eq:theta12} and \eqref{rd07_01e1}  that
 \begin{align*}
 \E\left[Z_{t_0+\epsilon^{2/3},\, x_0+\epsilon^{1/3}}^q\right] \lesssim \epsilon^{\frac{q}{3}(4+p_0-\gamma_0)},
 \end{align*}
 and therefore, for $0<\epsilon \leq \deltah^3$, 
 \begin{align*}
  \P\{\tilde{X}<\epsilon\} \lesssim R^{-q}\epsilon^{\frac{q}{3}(4+p_0-\gamma_0)}.
 \end{align*}
 This implies that $\tilde{X}$ has finite negative moments of all orders. Indeed,
 \begin{align*}
 \E[\tilde{X}^{-p}]&= p\int_0^\infty y^{p-1}\P\{\tilde{X}^{-1}>y\}\d y\\
 &=p\int_0^{\deltah^{-3}} y^{p-1}\P\{\tilde{X}^{-1}>y\}\d y + p\int_{\deltah^{-3}}^{\infty} y^{p-1}\P\{\tilde{X}<1/y\}\d y\\
 &\lesssim \deltah^{-3p} + R^{-q}\int_{\deltah^{-3}}^\infty y^{p-1}y^{-\frac{q}{3}(4+p_0-\gamma_0)}\d y\\
 &\asymp \deltah^{-3p} + R^{-q}\deltah^{-3p+q(4+p_0-\gamma_0)},
 \end{align*}
 provided we choose $q$ so that $p<\frac{q}{3}(4+p_0-\gamma_0)$.
Recall that $z_2\geq \deltah^{1/2}$. Therefore, by \eqref{rd08_04e2}, 
 \begin{align}\label{eq:R}
 R^{-1} 
     = c_1 z_2^{-2p_0}\deltah^{\gamma_0-4} \leq c_1 \deltah^{-p_0+\gamma_0-4}.
 \end{align}
 Hence, for $z_2\geq \deltah^{1/2}$,
  \begin{align*}
 \E[\tilde{X}^{-p}]&\lesssim  \deltah^{-3p} + \deltah^{(-p_0+\gamma_0-4)q -3p+q(p_0-\gamma_0+4)} \asymp \deltah^{-3p}.
 \end{align*}
Therefore, we conclude that (\ref{equ2}) holds for $z_2\geq \deltah^{1/2}$.
 \end{proof}

 \begin{lemma}\label{rd07_02l1}
 For $z_2\geq \deltah^{1/2}$,  we have
 \begin{equation} \label{equ3}
 \|\E[u_{A,2}]\|_{\mathcal{H}} \lesssim \deltah^{5/2}.
 \end{equation}
 \end{lemma}
 
 \begin{proof}
 Observe from \eqref{rd07_01e2} that for $r>0$ and $z\in \R$,
 \begin{align*} 
 u_{A,2}(r,z)& :=\sum_{k=1}^5 u_{A,2,k}(r,z),
 \end{align*}
 where
 \begin{align} \nonumber
& u_{A,2,1}(r,z) =  \frac{\partial \eta(r,z)}{\partial r}H_A(r,z) , \qquad\quad u_{A,2,2}(r,z) = \eta(r, z)\frac{\partial H_A(r,z)}{\partial r} , \\  \nonumber
&u_{A,2,3}(r,z) = -\frac12 \frac{\partial^2\eta(r,z)}{\partial z^2}H_A(r,z), \quad
 u_{A,2,4}(r,z) = -\frac{\partial \eta(r,z)}{\partial z}\,  \frac{\partial H_A(r,z)}{\partial z}, \\
 &u_{A,2,5}(r,z) = -\frac12\, \eta(r,z)\frac{\partial^2H_A(r,z)}{\partial z^2}.
 \label{for8}
 \end{align}

In order to simplify the notation, we let
\begin{align}\label{eq:E}
E=[t_0-\deltah^2, t_0+2\deltah^2]\times [x_0-\deltah, x_0+2\deltah].
\end{align}
 Since $\eta$ is supported on $E$, it suffices to estimate 
\begin{align*}
            \|\E[u_{A,2}]\|_{L^2(E)}=\left( \int_{t_0-\deltah^2}^{t_0+2\deltah^2}\int_{x_0-\deltah}^{x_0+2\deltah}\left(\E[u_{A,2}(r,z)]\right)^{2}\d z\d r\right)^{1/2}. 
 \end{align*}
For $(r,z)\in E$, we see from \eqref{eq:H} that
           \begin{align}\label{eq:HA}
           |H_A(r,z)|\leq \|\psi\|_{\infty}(\deltah^3+ |t_0+\deltah^2-r|\cdot|x_0+\deltah-z|) \lesssim \deltah^3,
           \end{align}
and by \eqref{eq:eta}, that
 \begin{align}\label{eq:etabound}
           \left|\frac{\partial\eta(r,z)}{\partial r}\right| \lesssim \deltah^{-2}\|\phi'\|_{\infty}\|\phi\|_{\infty}\quad \text{and}\quad 
                      \left|\frac{\partial^2\eta(r,z)}{\partial z^2}\right| \lesssim \deltah^{-2} \|\phi\|_{\infty}\|\phi''\|_{\infty}.
           \end{align}
Thus, 
\begin{align}\label{eq:u1}
    \vert u_{A,2,1}(r, z) \vert \lesssim \zeta^{-2} \zeta^3 = \zeta \quad\text{and}\quad       \|\E[u_{A,2,1}]\|_{L^2(E)}\lesssim  \left[(\deltah^{-2}\deltah^3)^2\deltah^3\right]^{1/2}= \deltah^{5/2},
\end{align}
and similarly, 
\begin{align}\label{eq:u3}
      \vert u_{A,2,3}(r, z) \vert \lesssim  \zeta \quad\text{and}\quad       \|\E[u_{A,2,3}]\|_{L^2(E)}\lesssim \deltah^{5/2}.
\end{align}
Moreover, by \eqref{eq:H}, we have for $(r,z)\in E$,
           \begin{align}\label{eq:HA2}
           \left|\frac{\partial H_A(r,z)}{\partial r} \right|= \left|\int_z^{x_0+\deltah}\psi(Z_{r, y})\, \d y\right|
            \leq \|\psi\|_{\infty}|x_0+\deltah-z| \lesssim \|\psi\|_{\infty}\deltah.
           \end{align}
Using the fact that for $(r,z)\in E$,
           \begin{align}\label{eq:eta2}
           |\eta(r,z)|\leq \|\phi\|_{\infty}^2,
           \end{align}
 we obtain that
 \begin{align}\label{eq:u2}
      \vert u_{A,2,2}(r, z) \vert \lesssim  \zeta \quad\text{and}\quad       \|\E[u_{A,2,2}]\|_{L^2(E)}\lesssim  [\deltah^2\deltah^3]^{1/2}=\deltah ^{5/2}.
\end{align}
           
We proceed to estimate $\|\E[u_{A,2,4}]\|_{L^2(E)}$. First, notice that for $(r,z)\in E$, by \eqref{eq:eta}, 
           \begin{align}\label{eq:eta3}
           \left|\frac{\partial \eta(r,z)}{\partial z}\right| \leq \deltah^{-1}\|\phi\|_{\infty}\|\phi'\|_{\infty}.
           \end{align}
By \eqref{eq:H}, for $(r,z)\in E$,
           \begin{align}\label{eq:HA3}
           \left|\frac{\partial H_A(r,z)}{\partial z} \right|= \left|\int_r^{t_0+\deltah^2}\psi(Z_{s, z})\, \d s\right|
            \leq \|\psi\|_{\infty}|t_0+\deltah^2-r| \lesssim \|\psi\|_{\infty}\deltah^2.
           \end{align}
           Hence,
\begin{align}\label{eq:u4}
      \vert u_{A,2,4}(r, z) \vert \lesssim  \zeta \quad\text{and}\quad       \|\E[u_{A,2,4}]\|_{L^2(E)}\lesssim  [(\deltah^{-1}\deltah^{2})^2\deltah^3]^{1/2}=\deltah^{5/2}.
\end{align}
           
It remains to estimate $\|\E[u_{A,2,5}]\|_{L^2(E)}$. By \eqref{eq:H},  for $(r,z)\in E$,
           \begin{align}\label{eq:HAderi}
           \frac{\partial^2 H_A(r,z)}{\partial z^2}= \int_r^{t_0+\deltah^2}\psi'(Z_{s, z}) \frac{\partial Z_{s,z}}{\partial z}\, \d s,
           \end{align}
and by \eqref{rd04_04e1},
\begin{align}
             \frac{\partial Z_{s,z}}{\partial z}&=2\int_{x_0}^z \frac {(v(t_0,x)-v(t_0,z))^{2p_0}}{|x-z|^{\gamma_0-2} }\d x \nonumber\\
             &\quad +2 \int_{[t_0, s]^2}\d \tilde{s}\d t \int_{x_0}^z \d x\, \frac {(v(t,x)+v(\tilde{s},z)-v(t,z)-v(\tilde{s},x))^{2p_0}}{|t-\tilde{s}|^{1+2p_0\gamma_1}|x-z|^{1+2p_0\gamma_2}}. \label{eq:partialZ}
            \end{align}
By the H\"older continuity of $v$ and \eqref{eq:rectangle},
for $s\in[r, t_0+\deltah^2]$ with $r\in [t_0-\deltah^2, t_0+2\deltah^2]$ and $z\in [x_0-\deltah, x_0+2\deltah]$, 
          \begin{align*}
          \E\left[\left| \frac{\partial Z_{s,z}}{\partial z}\right|\right] &\lesssim  \int_{x_0}^z |x-z|^{p_0-\gamma_0+2}\d x \\
          &\quad + \int_{[t_0, s]^2}\d \tilde{s}\d t \int_{x_0}^z \d x \,
          |t-\tilde{s}|^{p_0\theta_1-1-2p_0\gamma_1}|x-z|^{p_0\theta_2-1-2p_0\gamma_2}.
          \end{align*}
Using \eqref{eq:theta12} and \eqref{rd07_01e1}, we deduce that
            for these $(s, z)$, 
          \begin{align*}
          \E\left[\left| \frac{\partial Z_{s,z}}{\partial z}\right|\right] &\lesssim \deltah^{p_0-\gamma_0+3}.
          \end{align*}
Furthermore, since $\psi_0$ is a bounded function, we conclude using \eqref{rd07_02e1} that for $(r,z)\in E$,
          \begin{align*}
          |\E[u_{A,2,5}(r,z)]| \lesssim \|\psi'\|_{\infty}\deltah^{p_0-\gamma_0+5} \leq  \|\psi_0'\|_{\infty}R^{-1}\deltah^{p_0-\gamma_0+5}.
          \end{align*}
           Under the condition $z_2\geq \deltah^{1/2}$, we apply the inequality \eqref{eq:R} to see that for $(r,z)\in E$,
                     \begin{align*}
          |\E[u_{A,2,5}(r,z)]| \lesssim \deltah.
          \end{align*}
          Therefore, we obtain that for $z_2\geq \deltah^{1/2}$,
          \begin{align}\label{eq:u5}
          \|\E[u_{A,2,5}]\|_{L^2(E)} \lesssim  [\deltah^2\deltah^3]^{1/2}=\deltah^{5/2}.
          \end{align}
          
We combine \eqref{eq:u1}, \eqref{eq:u2}, \eqref{eq:u3}, \eqref{eq:u4} and \eqref{eq:u5} to complete the proof. 
 \end{proof}

  \begin{lemma}\label{lem:uA2mom}
 For $z_2\geq \deltah^{1/2}$,
 \begin{equation} \label{equ10}
 \|u_{A,2}\|_{L^p(\Omega;\mathcal{H})} \lesssim \deltah^{5/2}.
 \end{equation}
 \end{lemma}
 
 \begin{proof}
           Recall the definition of $E$ in \eqref{eq:E}.
           Since $\eta$ is supported on $E$, we have
           \begin{align*}
           \|u_{A,2}\|_{L^p(\Omega;\mathcal{H})} &= \left(\E\left[\left(\int_{t_0-\deltah^2}^{t_0+2\deltah^2}\int_{x_0-\deltah}^{x_0+2\deltah}
           u_{A,2}^{2}(r,z)\d z\d r\right)^{p/2}\right]\right)^{1/p}\\
           &\leq  \sum_{k=1}^5 \left(\E\left[\left(\int_{t_0-\deltah^2}^{t_0+2\deltah^2}\int_{x_0-\deltah}^{x_0+2\deltah}
           u_{A,2,k}^{2}(r,z)\d z\d r\right)^{p/2}\right]\right)^{1/p}
            =: \sum_{k=1}^5I_k ,
           \end{align*}
 where the $u_{A, 2, k}$ are defined in \eqref{for8}. 
           By \eqref{eq:u1} and \eqref{eq:u3},  we have for $(r,z)\in E$,
           \begin{align*}
           |u_{A,2,1}(r,z)| \lesssim \deltah \quad \text{and} \quad |u_{A,2,3}(r,z)| \lesssim \deltah,
           \end{align*}
           which implies that
           \begin{align*}
           I_1\lesssim \deltah^{5/2} \quad \text{and} \quad I_3\lesssim \deltah^{5/2}.
           \end{align*}
           According to \eqref{eq:u2} and \eqref{eq:u4},  for $(r,z)\in E$,
           \begin{align*}
           |u_{A,2,2}(r,z)| \lesssim \deltah \qquad \text{and} \qquad   |u_{A,2,4}(r,z)| \lesssim \deltah,
           \end{align*}
           which implies that
           \begin{align*}
           I_2\lesssim \deltah^{5/2} \qquad \text{and} \qquad  I_4\lesssim \deltah^{5/2}.
           \end{align*}
           
           It remains to estimate $I_5$. Using the boundedness of $\eta$ and \eqref{eq:HAderi}, we get
 \begin{align}\label{rd07_02e2}
           I_5^p\lesssim \|\psi'\|^p_{\infty}\E\left[\left(\int_{t_0-\deltah^2}^{t_0+2\deltah^2}\int_{x_0-\deltah}^{x_0+2\deltah} \left(\int_r^{t_0+\deltah^2}\left|\frac{\partial Z_{s,z}}{\partial z}\right|\d s\right)^2\d z \d r\right)^{p/2}\right].
\end{align}
 Using \eqref{eq:partialZ}, we see that
\begin{align}
           \left| \frac{\partial Z_{s,z}}{\partial z} \right| &\lesssim \int_{x_0-\deltah}^{x_0+2\deltah} \frac {(v(t_0,x)-v(t_0,z))^{2p_0}}{|x-z|^{\gamma_0-2} }\, \d x\nonumber\\
           &\quad +
            \int_{[t_0-\deltah^2, t_0+2\deltah^2]^2}\d \tilde{s}\d t \int_{x_0-\deltah}^{x_0+2\deltah} \d x\, \frac {(v(t,x)+v(\tilde{s},z)-v(t,z)-v(\tilde{s},x))^{2p_0}}{|t-\tilde{s}|^{1+2p_0\gamma_1}|x-z|^{1+2p_0\gamma_2}}\nonumber\\
            &:=\tilde{Z}_1 +\tilde{Z}_2. 
 \label{z1z2}
\end{align}
By \eqref{rd07_02e2} and \eqref{rd07_02e1}, 
           \begin{align*}
           I_5^p \lesssim R^{-p}\deltah^{7p/2}\E\left[(\tilde{Z}_1+\tilde{Z}_2)^p\right].
           \end{align*}
           By Minkowski's inequality and H\"older continuity of $v$,
           \begin{align}
           \|\tilde{Z}_1\|_p &\lesssim   \int_{x_0-\deltah}^{x_0+2\deltah} \frac {\left\|(v(t_0,x)-v(t_0,z))^{2p_0}\right\|_{p}}{|x-z|^{\gamma_0-2} }\, \d x
           \lesssim \deltah^{p_0-\gamma_0+3}. \label{eq:z1}
           \end{align} 
           By Minkowski's inequality and \eqref{eq:rectangle},
           \begin{align}
           \|\tilde{Z}_2\|_{p} &\lesssim    \int_{[t_0-\deltah^2, t_0+2\deltah^2]^2}\d \tilde{s}\d t \int_{x_0-\deltah}^{x_0+2\deltah} \d x\, 
               \frac {\left\|(v(t,x)+v(\tilde{s},z)-v(t,z)-v(\tilde{s},x))^{2p_0}\right\|_{p}}{|t-\tilde{s}|^{1+2p_0\gamma_1}|x-z|^{1+2p_0\gamma_2}}\nonumber\\
           &\lesssim \deltah^{5}\deltah^{2(p_0\theta_1-1-2p_0\gamma_1)}\deltah^{p_0\theta_2-1-2p_0\gamma_2}
           =\deltah^{p_0-\gamma_0+3}, \label{eq:z2}
           \end{align}
           where the last equality is due to \eqref{eq:theta12} and \eqref{rd07_01e1}.
           Therefore, 
           \begin{align*}
           I_5\lesssim R^{-1}\deltah^{p_0-\gamma_0+13/2}.
           \end{align*}
           Now, under the condition $z_2\geq \deltah^{1/2}$, \eqref{eq:R} holds. Therefore, 
           if $z_2\geq \deltah^{1/2}$, then 
           \begin{align}\label{rd08_06e1}
           I_5\lesssim \deltah^{5/2}.
           \end{align}
           The proof is complete.
\end{proof}

 \begin{lemma}\label{lem:DuA2}
 For $z_2\geq \deltah^{1/2}$,  
 \begin{equation} \label{UA}
 \|Du_{A,2}\|_{L^p(\Omega;\mathcal{H}^{\otimes2})} \lesssim \deltah^{5/2}.
 \end{equation}
 \end{lemma}
 
 \begin{proof} The proof of this lemma is based on the decomposition (\ref{for8}) and it will be done in several steps. Recall the definition of the set $E$ in \eqref{eq:E}.
 
 \medskip
 \noindent
 {\it Step 1:}
            For $(s,y)\in E$ and $p\geq1$,  we claim that
            \begin{equation} \label{Z}
             \E\left[\|DZ_{s,y}\|_{\mathcal{H}}^p \right] \lesssim \deltah^{p(4+p_0-\gamma_0)},
           \end{equation}
 where $Z_{s,y}$ is defined in \eqref{rd04_04e1}.  In order to establish (\ref{Z}), we compute 
           \begin{align}\label{eq:sum123}
           DZ_{s,y} = DY_1(s) + DY_2(y) + DY_3(s,y),
           \end{align}
           where
\begin{align*}
           DY_1(s) &= 2p_0\int_{[t_0, s]^2} \frac{(v(t,x_0)-v(\tilde{s},x_0))^{2p_0-1} D(v(t,x_0)-v(\tilde{s},x_0))}{|t-\tilde{s}|^{\gamma_0/2}}  \, \d \tilde{s} \d t,\\[6pt]
           DY_2(y)&=2p_0\int_{[x_0, y]^2} \frac {(v(t_0,x)-v(t_0,\tilde{y}))^{2p_0-1}  D\left(
           v(t_0,x)-v(t_0,\tilde{y})
           \right)}{|x-\tilde{y}|^{\gamma_0-2} }\, \d x\d \tilde{y},\\[6pt]
            DY_3(s,y)&=2p_0\int_{[t_0, s]^2}\d \tilde{s}\d t \int_{[x_0, y]^2} \d x\d \tilde{y} \, 
             \frac {(v(t,x)+v(\tilde{s},\tilde{y})-v(t,\tilde{y})-v(\tilde{s},x))^{2p_0-1}
            }{|t-\tilde{s}|^{1+2p_0\gamma_1}|x-\tilde{y}|^{1+2p_0\gamma_2}}\\
                       &  \qquad\qquad\qquad \times D\left(v(t,x)+v(\tilde{s},\tilde{y})-v(t,\tilde{y})-v(\tilde{s},x)\right).
\end{align*}
Since
\begin{align}
           \|D\left(v(t,x)-v(s,y)\right)\|_{\mathcal{H}} =
                      \Vert v(t, x) - v(s, y) \Vert_{2}  \lesssim |t-s|^{1/4}+|x-y|^{1/2},
\label{eq:increment2}           
\end{align}
and
           \begin{align}\label{eq:rect2}
           &\left\|D\left(v(t,x)+v(\tilde{s},\tilde{y})-v(t,\tilde{y})-v(\tilde{s},x)\right)  \right\|_{\mathcal{H}}\nonumber\\
           &\qquad=\left\|v(t,x)+v(\tilde{s},\tilde{y})-v(t,\tilde{y})-v(\tilde{s},x) \right\|_{2}
            \lesssim |t-\tilde{s}|^{\theta_1/2}|x-\tilde{y}|^{\theta_2/2}
           \end{align}
           (see \eqref{eq:rectangle}),
we deduce that 
           \begin{align*}
           \|DY_1(s)\|_{\mathcal{H}} &\lesssim \int_{[t_0-\deltah^2, t_0+2\deltah^2]^2} \frac {|v(t,x_0)-v(\tilde{s},x_0|^{2p_0-1}
           }{|t-\tilde{s}|^{\frac{\gamma_0}{2}-\frac14}}\,  \d \tilde{s} \d t,\\
           \|DY_2(y)\|_{\mathcal{H}}&\lesssim \int_{[x_0-\deltah, x_0+2\deltah]^2} \frac {|v(t_0,x)-v(t_0,\tilde{y}|^{2p_0-1}}{|x-\tilde{y}|^{\gamma_0-\frac52} } \, \d x\d \tilde{y},\\
            \|DY_3(s,y)\|_{\mathcal{H}}&\lesssim \int_{[t_0-\deltah^2, t_0+2\deltah^2]^2}\d \tilde{s}\d t \int_{[x_0-\deltah, x_0+2\deltah]^2} \d x\d \tilde{y} 
            \\
            &\qquad\qquad \times \frac {|v(t,x)+v(\tilde{s},\tilde{y})-v(t,\tilde{y})-v(\tilde{s},x)|^{2p_0-1}
            }{|t-\tilde{s}|^{1+2p_0\gamma_1 -\frac{\theta_1}{2}}|x-\tilde{y}|^{1+2p_0\gamma_2-\frac{\theta_2}{2}}}.
           \end{align*}
           By Minkowski' inequality, for $p\geq 1$,
           \begin{align}\label{eq:Y1}
           \|DY_1(s)\|_{L^p(\Omega, \mathcal{H})} &\lesssim  \int_{[t_0-\deltah^2, t_0+2\deltah^2]^2} \frac {\left\||v(t,x_0)-v(\tilde{s},x_0)|^{2p_0-1}\right\|_{p}
           }{|t-\tilde{s}|^{\frac{\gamma_0}{2}-\frac14}}\, \d \tilde{s} \d t \nonumber\\
           &\lesssim \zeta^4 \zeta^{2(p_0/2 - 1/4 - \gamma_0/2 + 1/4)} \lesssim \deltah^{4+p_0-\gamma_0}.
           \end{align}
            Similarly, 
           \begin{align}\label{eq:Y2}
           \|DY_2(y)\|_{L^p(\Omega, \mathcal{H})} &\lesssim 
           \int_{[x_0-\deltah, x_0+2\deltah]^2} \frac {\left\||v(t_0,x)-v(t_0,\tilde{y})|^{2p_0-1}\right\|_{p}}{|x-\tilde{y}|^{\gamma_0-\frac52} }\d x\d \tilde{y}\nonumber\\
           &\lesssim \zeta^2 \zeta^{p_0 - 1/2 - \gamma_0 + 5/2} \lesssim \deltah^{4+p_0-\gamma_0}.
           \end{align}
Moreover, by \eqref{eq:rectangle}, 
           \begin{align}\label{eq:Y3}
                       \|DY_3(s,y)\|_{L^p(\Omega, \mathcal{H})}&\lesssim \int_{[t_0-\deltah^2, t_0+2\deltah^2]^2}\d \tilde{s}\d t \int_{[x_0-\deltah, x_0+2\deltah]^2} \d x\d \tilde{y} \\
                       &\qquad \times  
             \frac {\left\|\, |v(t,x)+v(\tilde{s},\tilde{y})-v(t,\tilde{y})-v(\tilde{s},x)|^{2p_0-1}\right\|_{p}
            }{|t-\tilde{s}|^{1+2p_0\gamma_1 -\frac{\theta_1}{2}}|x-\tilde{y}|^{1+2p_0\gamma_2-\frac{\theta_2}{2}}} \nonumber\\
            &\lesssim \deltah^{6}\deltah^{2\left[(2p_0-1)\frac{\theta_1}{2}-1-2p_0\gamma_1+\frac{\theta_1}{2}\right]}
            \deltah^{(2p_0-1)\frac{\theta_2}{2}-1-2p_0\gamma_2+\frac{\theta_2}{2}} \nonumber \\
            &= \deltah^{4+p_0-\gamma_0},  \nonumber
           \end{align}
where, in the equality, we have used \eqref{rd07_01e1}. Therefore, the bound  (\ref{Z})  follows from the estimates \eqref{eq:Y1}--\eqref{eq:Y3}.
  \medskip
  
 \noindent
 {\it Step 2:} For $k=1,\dots, 5$, $z_2 \geq \zeta^{1/2}$ and $p \geq 1$, we claim that 
 \begin{equation} \label{rd07_02e5}
           \E\left[\left(\int_0^\infty\int_\R \|Du_{A,2,k}(r,z)\|_\mathcal{H}^2\, \d z\d r\right)^{p/2}\right] \lesssim \deltah^{5p/2},
\end{equation}
 where the $u_{A,2,k}(r, z)$ are defined in \eqref{for8}. This will prove \eqref{UA}.
 
  \medskip
 \noindent
 {\it Step 2a:}  We claim that \eqref{rd07_02e5} holds for $k=1$.

 Indeed, for $(r,z)\in E$, we use \eqref{eq:H} to write
           \begin{align*}
           Du_{A,2,1}(r,z) &= \frac{\partial \eta(r,z)}{\partial r}DH_A(r,z)\\
           &=\frac{\partial \eta(r,z)}{\partial r} \Big(\int_{t_0}^{t_0+\deltah^2}\int_{x_0}^{x_0+\deltah}\psi'(Z_{s,y})DZ_{s,y}\, \d y\d s \\
           & \qquad \qquad \qquad - \int_{r}^{t_0+\deltah^2}\int_{z}^{x_0+\deltah}\psi'(Z_{s,y})DZ_{s,y}\, \d y\d s\Big).
           \end{align*}
 By \eqref{eq:etabound} and \eqref{rd07_02e1},  for $(r,z)\in E$, we have
           \begin{align*}
           \|Du_{A,2,1}(r,z)\|_{\mathcal{H}}&\lesssim \deltah^{-2}R^{-1}\Big(\int_{t_0}^{t_0+\deltah^2}\int_{x_0}^{x_0+\deltah}\|DZ_{s,y}\|_{\mathcal{H}}\, \d y\d s \\
      &\qquad\qquad\qquad +\int_{t_0-\deltah^2}^{t_0+2\deltah^2}\int_{x_0-\deltah}^{x_0+2\deltah}\|DZ_{s,y}\|_{\mathcal{H}}\, \d y\d s\Big).
           \end{align*}
Hence, for $(r,z)\in E$, by \eqref{Z}, 
           \begin{align}\label{eq:UA1}
           \|Du_{A,2,1}(r,z)\|_{L^p(\Omega, \mathcal{H})} &\lesssim \deltah^{-2}R^{-1}
\int_{t_0-\deltah^2}^{t_0+2\deltah^2}\int_{x_0-\deltah}^{x_0+2\deltah}\|DZ_{s,y}\|_{L^p(\Omega,\mathcal{H})}\d y\d s\nonumber\\
           &\lesssim R^{-1}\deltah^{5+p_0-\gamma_0}.
           \end{align}
By Minkowski's inequality, 
           \begin{align*}
           &\left\|\int_0^\infty\int_\R \|Du_{A,2,1}(r,z)\|_\mathcal{H}^2\, \d z\d r\right\|_{\frac p2} \\
           &\qquad\qquad =\left\|\int_{t_0-\deltah^2}^{t_0+2\deltah^2}\int_{x_0-\deltah}^{x_0+2\deltah} \|Du_{A,2,1}(r,z)\|_\mathcal{H}^2\, \d z\d r\right\|_{\frac p2}\\
           &\qquad\qquad \leq \int_{t_0-\deltah^2}^{t_0+2\deltah^2}\int_{x_0-\deltah}^{x_0+2\deltah} \|Du_{A,2,1}(r,z)\|_{L^p(\Omega, \mathcal{H})}^2\, \d z\d r
           \lesssim R^{-2}\deltah^{3+2(5+p_0-\gamma_0)},
           \end{align*}
where the second inequality is due to \eqref{eq:UA1}. Using \eqref{eq:R}, 
we deduce that for $z_2\geq \deltah^{1/2}$,
           \begin{align*}
           \left\|\int_0^\infty\int_\R \|Du_{A,2,1}(r,z)\|_\mathcal{H}^2\, \d z\d r\right\|_{\frac p2} \lesssim \deltah^{5},
           \end{align*}
           which completes the proof of (\ref{rd07_02e5}) for $k=1$.
  \medskip
  
 \noindent{\it Step 2b:}
           We claim that \eqref{rd07_02e5} holds for $k=2$.
  
Indeed, for $(r,z)\in E$, we have
           \begin{align*}
           Du_{A,2,2}(r,z) = \eta(r,z) \int_z^{x_0+\deltah} \psi'(Z_{r,y})DZ_{r,y} \, \d y. 
           \end{align*}
           It follows from \eqref{rd07_02e1} that for $(r,z)\in E$,
           \begin{align*}
           \|Du_{A,2,2}(r,z)\|_{\mathcal{H}} \lesssim R^{-1}\int_{x_0-\deltah}^{x_0+2\deltah}\|DZ_{r,y}\|_{\mathcal{H}}\, \d y.
           \end{align*}
           Hence, for $(r,z)\in E$,
           \begin{align*}
           \|Du_{A,2,2}(r,z)\|_{L^p(\Omega,\mathcal{H})}& \lesssim 
           R^{-1}\int_{x_0-\deltah}^{x_0+2\deltah}\|DZ_{r,y}\|_{L^p(\Omega, \mathcal{H})}\, \d y
           \lesssim R^{-1}\deltah^{5+p_0-\gamma_0},
           \end{align*}
           where the second inequality holds by (\ref{Z}). This is similar to \eqref{eq:UA1}, so the remainder of the proof follows along the same lines 
           as in Step 2b.
   \medskip
   
 \noindent
 {\it Step 2c:}
We claim that \eqref{rd07_02e5} holds for $k=3$.
  
 Indeed, the proof of (\ref{rd07_02e5}) for $k=3$ is the same as for $k=1$ (Step 2a), using the fact $\|\frac{\partial^2 \eta}{\partial z^2}\|_\infty\lesssim \deltah^{-2}$.
 
  \medskip
  
 \noindent {\it Step 2d:}
           We claim that \eqref{rd07_02e5} holds for $k=4$.

 Indeed, for $(r,z)\in E$,
           \begin{align*}
           Du_{A,2,4}(r,z) =- \frac{\partial \eta(r,z)}{\partial z} \int_r^{t_0+\deltah^2} \psi'(Z_{s,z})DZ_{s,z}\, \d s. 
           \end{align*}
Using (\ref{rd07_02e1}) and \eqref{eq:eta},  it follows that for $(r,z)\in E$, 
           \begin{align*}
           \|Du_{A,2,4}(r,z)\|_{\mathcal{H}} \lesssim \deltah^{-1} R^{-1} \int_{t_0-\deltah^2}^{t_0+2\deltah^2}\|DZ_{s,z}\|_{\mathcal{H}}\, \d s.
           \end{align*}
           By (\ref{Z}), for $(r,z)\in E$,
           \begin{align*}
           \|Du_{A,2,4}(r,z)\|_{L^p(\Omega, \mathcal{H})}& \lesssim \deltah^{-1} R^{-1} \int_{t_0-\deltah^2}^{t_0+2\deltah^2}\|DZ_{s,z}\|_{L^p(\Omega, \mathcal{H})}\, \d s
           \lesssim  R^{-1}\deltah^{5+p_0-\gamma_0}.
           \end{align*}
This is similar to \eqref{eq:UA1}, so the remainder of the proof follows along the same lines 
           as in Step 2a.
 
   \medskip
 \noindent
 {\it Step 2e:} We claim that for $(s,z)\in E$ and $p\geq1$,  
            \begin{equation} \label{partialZ}
             \E\left[\left\|D\frac{\partial Z_{s,z}}{\partial z}\right\|_{\mathcal{H}}^p \right] \lesssim \deltah^{p(p_0-\gamma_0+3)},
           \end{equation}
           where  the random variable $\frac{\partial Z_{s, z}}{\partial z}$ is explicited in \eqref{eq:partialZ}.
           
 Indeed, for $(s,z)\in E$, we write
           \begin{align*}
           D\frac{\partial Z_{s,z}}{\partial z}&=
           4p_0(J_1+J_2),
           \end{align*}
           where 
\begin{align*}
     J_1&=\int_{x_0}^z \frac {(v(t_0,x)-v(t_0,z))^{2p_0-1}  D(v(t_0, x)-v(t_0,z))}{|x-z|^{\gamma_0-2} }\d x, \\[7pt]
           J_2&=\int_{[t_0, s]^2}\d \tilde{s}\d t \int_{x_0}^z \d x\frac {(v(t,x)+v(\tilde{s},z)-v(t,z)-v(\tilde{s},x))^{2p_0-1}
             }{|t-\tilde{s}|^{1+2p_0\gamma_1}|x-z|^{1+2p_0\gamma_2}} \\
              &\qquad\qquad \times D(v(t,x)+v(\tilde{s},z)-v(t,z)-v(\tilde{s},x)).
\end{align*}
From \eqref{eq:increment2} and \eqref{eq:rect2}, we see that for $(s,z)\in E$,
           \begin{align*}
           &\left\|D\frac{\partial Z_{s,z}}{\partial z}\right\|_{\mathcal{H}}\lesssim 
           \int_{x_0-\deltah}^{x_0+2\deltah} \frac {|v(t_0,x)-v(t_0,z)|^{2p_0-1}}{|x-z|^{\gamma_0-\frac52} }\, \d x\\
           &\quad + \int_{[t_0-\deltah^2, t_0+2\deltah^2]^2}\d \tilde{s}\d t \int_{x_0-\deltah}^{x_0+2\deltah} \d x \, 
            \frac {|v(t,x)+v(\tilde{s},z)-v(t,z)-v(\tilde{s},x)|^{2p_0-1}
             }{|t-\tilde{s}|^{1+2p_0\gamma_1-\frac{\theta_1}{2}}|x-z|^{1+2p_0\gamma_2-\frac{\theta_2}{2}}}.
           \end{align*}
           Moreover, by Minkowski's inequality, 
          \begin{align*}
           &\left\|D\frac{\partial Z_{s,z}}{\partial z}\right\|_{L^p(\Omega, \mathcal{H})}\lesssim 
           \int_{x_0-\deltah}^{x_0+2\deltah} \frac {\left\||v(t_0,x)-v(t_0,z)|^{2p_0-1}\right\|_{p}}{|x-z|^{\gamma_0-\frac52} }\, \d x\\
           &\qquad+ \int_{[t_0-\deltah^2, t_0+2\deltah^2]^2}\d \tilde{s}\d t \int_{x_0-\deltah}^{x_0+2\deltah} \d x \, 
            \frac {\left\|\, |v(t,x)+v(\tilde{s},z)-v(t,z)-v(\tilde{s},x)|^{2p_0-1}\right\|_{p}
             }{|t-\tilde{s}|^{1+2p_0\gamma_1-\frac{\theta_1}{2}}|x-z|^{1+2p_0\gamma_2-\frac{\theta_2}{2}}}\\
             &\qquad   \lesssim \deltah \deltah^{\half (2p_0 -1) -\gamma_0+\frac 5 2} + \deltah^{5}\deltah^{2((2p_0-1)\frac{\theta_1}{2} -(1+2p_0\gamma_1)+\frac{\theta_1}{2})}\deltah^{(2p_0-1)\frac{\theta_2}{2}-(1+2p_0\gamma_2)+\frac{\theta_2}{2}}\\
             &\qquad=2\deltah^{p_0-\gamma_0+3},
           \end{align*}
           where the equality holds by \eqref{eq:theta12} and \eqref{rd07_01e1}.
\medskip
   
 \noindent {\it Step 2f:}
           We claim that \eqref{rd07_02e5} holds for $k=5$. 

Indeed, for $(r,z)\in E$, 
          by the expression of $u_{A,2,5}(r,z)$ in \eqref{for8} and \eqref{eq:H}, we have
          \begin{align*}
          Du_{A,2,5}(r,z) =-\frac12 \eta(r,z)\int_r^{t_0+\deltah^2}\left[\psi''(Z_{s, z})DZ_{s,z} \frac{\partial Z_{s,z}}{\partial z}
          + \psi'(Z_{s, z}) D\frac{\partial Z_{s,z}}{\partial z}\right]
          \d s.
          \end{align*}
Since $\|\psi''\|_{\infty} \lesssim R^{-2}$ and $\|\psi'\|_{\infty} \lesssim R^{-1}$ (see \eqref{rd07_02e1}), for $(r,z)\in E$, we have
 \begin{align*}
          \|Du_{A,2,5}(r,z)\|_{\mathcal{H}} &\lesssim R^{-2}\int_{t_0-\deltah^2}^{t_0+2\deltah^2} \|DZ_{s,z}\|_{\mathcal{H}}
          \left| \frac{\partial Z_{s,z}}{\partial z}\right| \d s 
          + R^{-1}\int_{t_0-\deltah^2}^{t_0+2\deltah^2}
                     \left\|D\frac{\partial Z_{s,z}}{\partial z}\right\|_{\mathcal{H}}\d s. 
\end{align*}
          By Minkowski's inequality and H\"older's inequality, 
          \begin{align*}
          \|Du_{A,2,5}(r,z)\|_{L^{p}(\Omega,\mathcal{H})} &\lesssim R^{-2}\int_{t_0-\deltah^2}^{t_0+2\deltah^2} \|DZ_{s,z}\|_{L^{2p}(\Omega,\mathcal{H})}
          \left\| \frac{\partial Z_{s,z}}{\partial z}\right\|_{2p} \d s\\
          &\qquad\qquad\qquad\qquad+ R^{-1}\int_{t_0-\deltah^2}^{t_0+2\deltah^2}
                     \left\|D\frac{\partial Z_{s,z}}{\partial z}\right\|_{L^p(\Omega, \mathcal{H})}\d s. 
          \end{align*}
          Recall from \eqref{z1z2}--\eqref{eq:z2} that
          \begin{align}\label{mom:partialZ}
                    \left\| \frac{\partial Z_{s,z}}{\partial z}\right\|_{L^{2p}(\Omega)}  \lesssim \deltah^{p_0-\gamma_0+3}.
          \end{align}
          We combine this with the inequalities  (\ref{Z}) and  (\ref{partialZ}) to deduce that 
          for $(r,z)\in E$ and $p\geq1$,
\begin{align*}
          \|Du_{A,2,5}(r,z)\|_{L^{p}(\Omega,\mathcal{H})} &\lesssim R^{-2} \deltah^2 \deltah^{4 + p_0 - \gamma_0} \deltah^{p_0 - \gamma_0 +3} + R^{-1} \deltah^2 \deltah^{p_0 - \gamma_0 +3}\\
           &= R^{-2}\deltah^{2(p_0-\gamma_0)+9} + R^{-1}\deltah^{p_0-\gamma_0+5}.
\end{align*}
Therefore, 
\begin{align}\nonumber
  \left\|\int_0^\infty\int_\R \|Du_{A,2,5}(r,z)\|_\mathcal{H}^2 \, \d z\d r\right\|_{\frac p2} &  \leq \int_{t_0-\deltah^2}^{t_0+2\deltah^2}\int_{x_0-\deltah}^{x_0+2\deltah} \|Du_{A,5}(r,z)\|_{L^p(\Omega, \mathcal{H})}^2 \, \d z\d r\\ \nonumber
     & \lesssim  \deltah^3 (R^{-4} \deltah^{4 (p_0 - \gamma_0) + 18} + R^{-2} \deltah^{2 (p_0 - \gamma_0 + 10)} ) \\
           &= R^{-4}\deltah^{4(p_0-\gamma_0)+21} + R^{-2}\deltah^{2(p_0-\gamma_0)+13}.
 \label{rd07_02e9}
 \end{align}
By \eqref{eq:R},  $R^{-4}\lesssim \deltah^{-4(p_0-\gamma_0)-16}$ and $R^{-2}\lesssim \deltah^{-2(p_0-\gamma_0)-8}$
           if $z_2\geq \deltah^{1/2}$. We deduce that for $z_2\geq \deltah^{1/2}$,
\begin{align} \label{rd08_06e2}
           \left\|\int_0^\infty\int_\R \|Du_{A,2,5}(r,z)\|_\mathcal{H}^2 \, \d z\d r\right\|_{\frac p2} \lesssim \deltah^{5},
\end{align}
           which completes the proof of \eqref{rd07_02e5} for $k=5$.
           
           Finally, 
  Steps 2a--2d and 2f establish \eqref{rd07_02e5} and this completes the proof of Lemma \ref{lem:DuA2}.
 \end{proof}
 
  \begin{lemma}\label{D2u_A2}
             For $z_2\geq \deltah^{1/2}$,  we have
             \begin{equation} \label{D2UA2}
              \|D^2u_{A,2}\|_{L^p(\Omega;\mathcal{H}^{\otimes3})} \lesssim \deltah^{5/2}.
 \end{equation}
 \end{lemma}
 
 \begin{proof}
           Recall the definition of $E$ in \eqref{eq:E}.
           First, we can perform the same calculations  as in the estimate of  \eqref{Z} to see that 
                       for $(s,y)\in E$ and $p\geq1$,
           \begin{equation} \label{D2Z}
             \E\left[\|D^2Z_{s,y}\|_{\mathcal{H}^{\otimes 2}}^p \right] \lesssim \deltah^{p(4+p_0-\gamma_0)}.
           \end{equation}
           In order to estimate the moment of $D^2u_{A,2,1}$, using \eqref{for8} and \eqref{eq:H}, 
           we write for $(r,z)\in E$,
           \begin{align*}
           &D^2u_{A,2,1}(r,z) = \frac{\partial \eta(r,z)}{\partial r}D^2H_A(r,z)\\
           &=\frac{\partial \eta(r,z)}{\partial r} \bigg(\int_{t_0}^{t_0+\deltah^2}\int_{x_0}^{x_0+\deltah}(\psi''(Z_{s,y}) DZ_{s,y}\otimes
           DZ_{s,y}
           + \psi'(Z_{s,y})D^2Z_{s,y}) \, 
           \d y\d s\\
           &\qquad\qquad\quad\quad
           - \int_{r}^{t_0+\deltah^2}\int_{z}^{x_0+\deltah}(\psi''(Z_{s,y})DZ_{s,y}\otimes
           DZ_{s,y}
           + \psi'(Z_{s,y})D^2Z_{s,y}) \, \d y\d s\bigg).
           \end{align*}
          Since $\|\frac{\partial \eta}{\partial r}\|_{\infty} \lesssim \deltah^{-2}$ by \eqref{eq:eta}, and
          $\|\psi'\|_{\infty} \lesssim R^{-1}$ and $\|\psi''\|_{\infty} \lesssim R^{-2}$ by \eqref{rd07_02e1}, 
          for $(r,z)\in E$, we have 
          \begin{align*}
          \|D^2u_{A,2,1}(r,z)\|_{\HH^{\otimes 2}} \lesssim \deltah^{-2} \int_{t_0-\deltah^2}^{t_0+2\deltah^2}\int_{x_0-\deltah}^{x_0+2\deltah}
          (R^{-2}\|DZ_{s,y}\|_{\HH}^2 + R^{-1}\|D^2Z_{s,y}\|_{\HH^{\otimes2}})\, \d y\d s.
          \end{align*}
          Using \eqref{Z}, \eqref{D2Z} and \eqref{eq:R},  
          we deduce that for $z_2\geq \deltah^{1/2}$ and $p\geq1$, 
\begin{equation*} 
    \|D^2u_{A,2,1}(r,z)\|_{\HH^{\otimes 2}} \lesssim \deltah^{-2} \deltah^3 \, (R^{-2} \deltah^{2(4 + p_0 - \gamma_0)} + R^{-1} \deltah^{4 + p_0 - \gamma_0}) \asymp \zeta .
\end{equation*}
As in the last part of the proof of Step 2a in the proof of Lemma \ref{lem:DuA2}, we deduce that
\begin{equation*} 
              \|D^2u_{A,2,1}\|_{L^p(\Omega;\mathcal{H}^{\otimes3})} \lesssim \deltah^{5/2}.
 \end{equation*}
          Similarly, for $k=2,3,4$, we can use the same argument as in the estimate of $\|Du_{A,2,k}\|_{L^p(\Omega, \HH^{\otimes2})}$
          to obtain that
          \begin{align*}
          \|D^2u_{A,2,k}\|_{L^p(\Omega, \HH^{\otimes3})} \lesssim \deltah^{5/2},
          \end{align*}
          for $z_2\geq \deltah^{1/2}$ and $p\geq1$.
          
          It remains to estimate  $\|D^2u_{A,2,5}\|_{L^p(\Omega, \HH^{\otimes3})}$. Similar to \eqref{partialZ}, we can derive that
          for $(s,z)\in E$ and $p\geq1$
                      \begin{equation} \label{partialZ2}
             \E\left[\left\|D^2\frac{\partial Z_{s,z}}{\partial z}\right\|_{\mathcal{H}^{\otimes 2}}^p \right] \lesssim \deltah^{p(p_0-\gamma_0+3)}.
           \end{equation}
           Then for $(r,z)\in E$,  using (\ref{eq:HAderi}), we can write
          \begin{align*}
          &D^2u_{A,2,5}(r,z) =\frac12 \eta(r,z)\bigg(\int_r^{t_0+\deltah^2}\Big[\psi'''(Z_{s, z})DZ_{s,y}\otimes
           DZ_{s,y} \frac{\partial Z_{s,z}}{\partial z} \\
         &\quad  +\psi''(Z_{s, z})D^2Z_{s,z} \frac{\partial Z_{s,z}}{\partial z} 
          +2 \psi''(Z_{s, z})DZ_{s,z} D\frac{\partial Z_{s,z}}{\partial z}
          + \psi'(Z_{s,z})D^2\frac{\partial Z_{s,z}}{\partial z} \Big]
          \d s\bigg).
          \end{align*}
          Thus, we can appeal to the estimates \eqref{rd07_02e1}, \eqref{Z}, \eqref{mom:partialZ}, \eqref{D2Z}, \eqref{partialZ} and \eqref{partialZ2}
          to obtain that for $z_2\geq \deltah^{1/2}$ and $p\geq1$,
\begin{align*}
    \|D^2u_{A,2,5}(r, z)\|_{\mathcal{H}^{\otimes 2}} &\lesssim \deltah^2 \left[R^{-3} \zeta^{2(4 + p_0 - \gamma_0)} \deltah^{p_0 - \gamma_0 +3} + R^{-2} \deltah^{4 + p_0 - \gamma_0} \deltah^{p_0 - \gamma_0 + 3} \right. \\
    &\qquad\qquad \left. + R^{-2}  \deltah^{4 + p_0 - \gamma_0} \deltah^{p_0 - \gamma_0 + 3}  + R^{-1} \deltah^{p_0 - \gamma_0 + 3}  \right].
\end{align*}
With \eqref{eq:R}, we obtain
\begin{align*}
    \|D^2u_{A,2,5}(r, z)\|_{\mathcal{H}^{\otimes 2}} &\lesssim \deltah,
\end{align*}
and therefore,
 \begin{equation*} 
              \|D^2u_{A,2,5}\|_{L^p(\Omega;\mathcal{H}^{\otimes3})} \lesssim \deltah^{5/2}.
 \end{equation*}
 
 The proof is complete. 
 \end{proof}

\begin{lemma}\label{lem:uA1uA2}
          For $z_2\geq \deltah^{1/2}$,
          \begin{align*}
          \|\delta(\langle u_{A,1}(*), D_{*}u_{A,2}\rangle_{\HH})\|_p  \lesssim \deltah^{5/2}.
          \end{align*}
\end{lemma}
\begin{proof}
          By Meyer's inequality (see \cite[Proposition 1.5.8]{nualart2006}),  for $p\geq1$,
          \begin{align*}
          &\|\delta(\langle u_{A,1}(*), D_{*}u_{A,2}\rangle_{\HH})\|_p\\
          &\quad \leq c_p\left(
          \|\E[\langle u_{A,1}(*), D_{*}u_{A,2}\rangle_{\HH}]\|_\HH + \| D \langle u_{A,1}(*), D_{*}u_{A,2}\rangle_{\HH} \|_{L^p(\Omega; \HH^{\otimes2})}
          \right).
          \end{align*}
          We write
          \begin{align*}
          &\|\E[\langle u_{A,1}(*), D_{*}u_{A,2}\rangle_{\HH}]\|^2_\HH\\
          &\quad= \int_{\R_+\times \R}
          \left(\E\left[\int_{\R_+\times\R}u_{A,1}(r_2,v_2)D_{r_2,v_2}u_{A,2}(r_1,v_1)\, \d r_2\d v_2\right]\right)^2 \d r_1\d v_1.
          \end{align*}
          Since $u_{A,1}$ is deterministic, by the Cauchy-Schwarz inequality, 
          \begin{align*}
                    &\|\E[\langle u_{A,1}(*), D_{*}u_{A,2}\rangle_{\HH}]\|^2_\HH \\
                    &\quad \leq 
           \|u_{A,1}\|_{\HH}^2   \int_{\R_+\times \R}
          \left(\E\left[\left(\int_{\R_+\times\R}[D_{r_2,v_2}u_{A,2}(r_1,v_1)]^2\, \d r_2\d v_2\right)^{1/2}\right]\right)^2 \d r_1\d v_1\\
          &\quad \lesssim   \E\left[ \int_{\R_+\times \R}
          \int_{\R_+\times\R}[D_{r_2,v_2}u_{A,2}(r_1,v_1)]^2\, \d r_2\d v_2 \d r_1\d v_1\right]\\
          &\quad = \|Du_{A,2}\|^2_{L^2(\Omega; \HH^{\otimes2})} \lesssim \deltah^{5}
          \end{align*}
          for $z_2\geq \deltah^{1/2}$, where the last inequality holds by \eqref{UA}.
          
          Moreover, we write
          \begin{align*}
          &\| D \langle u_{A,1}(*), D_{*}u_{A,2}\rangle_{\HH} \|^2_{L^p(\Omega; \HH^{\otimes2})}\\
          &\quad=\left\|\int_{(\R_+\times\R)^2} \left(
          \int_{\R_+\times\R}u_{A,1}(r,z)D_{r_2, v_2}D_{r,z}u_{A,2}(r_1,v_1)\, \d r\d z
          \right)^2 \d r_1\d v_1\d r_2\d v_2
          \right\|_{\frac p2}.
          \end{align*}
          Since $u_{A,1}$ is deterministic, by the Cauchy-Schwarz inequality, 
          \begin{align*}
          &\| D \langle u_{A,1}(*), D_{*}u_{A,2}\rangle_{\HH} \|^2_{L^p(\Omega; \HH^{\otimes2})}\\
          &\qquad \leq  \|u_{A,1}\|_{\HH}^2 \left\|\int_{(\R_+\times\R)^3} 
          [D_{r_2, v_2}D_{r,z}u_{A,2}(r_1,v_1)]^2\, \d r\d z
           \d r_1\d v_1\d r_2\d v_2
          \right\|_{\frac p2}\\
          &\qquad=  \|u_{A,1}\|_{\HH}^2\,  \|D^2u_{A,2}\|^2_{L^p(\Omega; \HH^{\otimes3})} \lesssim \deltah^{5}
          \end{align*}
          for $z_2\geq \deltah^{1/2}$, where the last inequality holds by \eqref{D2UA2}.        
          The proof is complete.
\end{proof}

 \begin{lemma}
  For $z_2\geq \deltah^{1/2}$, 
  \begin{equation} \label{for21}
    \|\delta(u_{A,2})\|_p   \lesssim \deltah^{5/2}.
    \end{equation}
     \end{lemma}
     
\begin{proof}
From Meyer's inequality (see \cite[Proposition 1.5.8]{nualart2006}), we deduce that for any $p>1$,
\[
  \|\delta(u_{A,2})\|_p \le c_p \left( \|\E[u_{A,2}]\|_\HH+ \| Du_{A,2} \|_{L^p(\Omega; \HH^{\otimes2}} \right)
  \lesssim \deltah^{5/2},
\]
  where, in the last inequality, we have used (\ref{equ3}) and (\ref{UA}).
  \end{proof}

  \begin{lemma}\label{lem:gamma}
 For $z_2\geq \deltah^{1/2}$, 
 \begin{equation} \label{Dgamma}
 \|D\gamma_A^{2,2}\|_{L^p(\Omega;\mathcal{H})} \lesssim \deltah^{3}.
 \end{equation}
 \end{lemma}
 
\begin{proof}
           From the definition of $\gamma_A^{2,2}$ in \eqref{rd08_04e1}, we have 
           \begin{align*}
           D\gamma_A^{2,2}=\int_{t_0}^{t_0+\deltah^2}\int_{x_0}^{x_0+\deltah}\psi'(Z_{s,y}) DZ_{s,y}\, \d y\d s
           \end{align*}
           and hence, using \eqref{rd07_02e1},
                      \begin{align*}
           \|D\gamma_A^{2,2}\|_{\mathcal{H}}\leq R^{-1}\int_{t_0}^{t_0+\deltah^2}\int_{x_0}^{x_0+\deltah}\|DZ_{s,y}\|_{\mathcal{H}}\, \d y\d s.
           \end{align*}
           Thus, by Minkowski's inequality, 
\begin{align*}
           \|D\gamma_A^{2,2}\|_{L^p(\Omega, \mathcal{H})}&\leq R^{-1}\int_{t_0}^{t_0+\deltah^2}\int_{x_0}^{x_0+\deltah}\|DZ_{s,y}\|_{L^p(\Omega, \mathcal{H})}\, \d y\d s
           \lesssim  R^{-1}\deltah^{p_0-\gamma_0+7},
 \end{align*}
           where the second inequality follows from (\ref{Z}).  
 By \eqref{eq:R}, we conclude  that for $z_2\geq \deltah^{1/2}$,
\begin{align*}
           \|D\gamma_A^{2,2}\|_{L^p(\Omega, \mathcal{H})}           &\lesssim  \deltah^{3}.
\end{align*}
           This proves Lemma \ref{lem:gamma}.
\end{proof}

\begin{lemma}\label{lem:uA2gamma}
           For $z_2\geq \deltah^{1/2}$, 
           \begin{align*}\label{eq:uA2gamma}
           \|\langle u_{A,1}, \langle Du_{A,2}(*), D_*\gamma_A^{2,2}\rangle_\HH\rangle_\HH \|_p\lesssim \deltah^{11/2}.
           \end{align*}
\end{lemma}

\begin{proof}
          By the Cauchy-Schwarz inequality, 
 \begin{align*}
          | \langle u_{A,1}, \langle Du_{A,2}(*), D_*\gamma_A^{2,2}\rangle_\HH\rangle_\HH | &\leq \|u_{A,1}\|_\HH \, \|Du_{A,2}\|_{\HH^{\otimes2}}\, 
          \|D\gamma_A^{2,2}\|_\HH \\
          &\lesssim \deltah^{5/2} \deltah ^3 = \deltah^{11/2},
\end{align*}
where the inequality follows from Lemmas \ref{lem:DuA2} and \ref{lem:gamma}. This proves Lemma \ref{lem:uA2gamma}.
\end{proof}

  \begin{lemma}\label{lem:Dgamma12}
   For $z_2\geq \deltah^{1/2}$, 
\begin{equation} \label{A12}
   \| \gamma_A^{1,2}\|_p \lesssim \deltah^{\frac 52},
\end{equation}
 \begin{equation} \label{DA12}
   \| D\gamma_A^{1,2}\|_{L^p(\Omega;\HH)} \lesssim \deltah^{\frac 52}
\end{equation}
   and
 \begin{equation} \label{DDA12}
       \| D^2\gamma_A^{1,2}\|_{L^p(\Omega;\HH\otimes \HH)} \lesssim \deltah^{\frac 52}.
\end{equation}
  \end{lemma}
  
  \begin{proof}
  By \eqref{rd08_04e1}, we can write
  \begin{align*}
  \| \gamma_A^{1,2}\|_p&= \left \| \int_0^{t_0} \int_\R D_{r,z} u(t_0,x_0) u_{A,2}(r,z)\, \d z \d r \right\|_p\\
  &\le  \| Du(t_0,x_0) \|_{L^{2p}(\Omega;\HH)}\, \|u_{A,2}\|_{L^{2p}(\Omega;\HH)}  
  \le C\, \|u_{A,2}\|_{L^{2p}(\Omega;\HH)}  \lesssim \deltah^{\frac 52},
  \end{align*}
  where, in the last inequality, we have used (\ref{equ10}).  The proofs of \eqref{DA12} and \eqref{DDA12} are similar using the fact that  $\|D^2u(t_0, x_0)\|_{\HH^{\otimes2}}$ and $\|D^3u(t_0, x_0)\|_{\HH^{\otimes3}}$ have uniformly bounded moments  over $(t_0, x_0)\in I \times J$, 
as mentioned in \eqref{rd08_06e3}.
  \end{proof}

 \begin{lemma} \label{vanish}
 We have
   \begin{equation} \label{for7}
  \langle u_{A,1},  D\gamma_A^{2,2}\rangle_\HH=0.
  \end{equation}
 \end{lemma}
 
 \begin{proof}
  In order to show (\ref{for7}),   we use \eqref{rd08_04e1} to write
 \[
 D\gamma_A^{2,2}=  \int_{t_0}^{t_0+\deltah^2}\int_{x_0}^{x_0+\deltah}\psi'(Z_{s,y})DZ_{s,y}\, \d y\d s .
 \]
 As a consequence,   (\ref{for7}) will be proved if we show that for any $(s,y)\in  [t_0,t_0+\deltah^2]\times [x_0,x_0+\deltah]$, 
 \begin{align*}
 \langle u_{A,1}, DZ_{s,y} \rangle_\HH=0.
 \end{align*}
 Taking into account  the expression of $DZ_{s,y}$ in \eqref{eq:sum123}, the fact that this inner product vanishes  follows from the following computations: for any  $(t,x), (s,y)\in [t_0,t_0+\deltah^2]\times [x_0,x_0+\deltah]$,
 \begin{align*}
 & \langle u_{A,1}, D[v(t,x)-v(s,y)] \rangle_\HH \\
   &\quad= \int_0^\infty \int_{\R}\d r \d z \left(G(t-r,x-z)\mathbf{1}_{\{r<t\}} -G(s-r,y-z) \mathbf{1}_{\{ r< s\}} \right)\\
  &\qquad\qquad \qquad \times \left( \frac {\partial}{\partial r}-\frac 12 \frac {\partial^2}{\partial z^2} \right)
    \left[f(r)g(z) \right] \\
    &\quad=f(t)g(x)-f(s)g(y)=1\times 1-1\times 1=0,
    \end{align*}
    where the second equality holds by Lemma \ref{lem:heat}.
    This completes the proof of Lemma \ref{vanish}.
 \end{proof}



\begin{thebibliography}{99}
 
 \bibitem{Adl90} Adler, R.J.  (1990). {\it An Introduction to Continuity, Extrema, and Related Topics for General Gaussian Processes.} \textit{Lect. Notes Monogr. Ser.} \textbf{12}, 1--155 

 
 \bibitem{bp1998} Bally, V. and  Pardoux, E. (1998).
Malliavin calculus for white noise driven parabolic SPDEs.
{\it Potential Anal.} {\bf 9} no. 1, 27–64.
 
 \bibitem{BLX} 
 Bierm\'e, H., Lacaux, C. and Xiao, Y. (2009).
Hitting probabilities and the Hausdorff dimension of the inverse images of anisotropic Gaussian random fields.
{\it Bull. Lond. Math. Soc.} {\bf 41}  no. 2, 253–273.
 
 \bibitem{CKNP23}
 Chen, L., Khoshnevisan, D.,  Nualart, D. and Pu, F. (2023)
Central limit theorems for spatial averages of the stochastic heat equation via Malliavin-Stein's method.
{\it Stoch. Partial Differ. Equ. Anal. Comput.} {\bf 11}  no. 1, 122–176.
 

 \bibitem{DKN07}
 Dalang, R.C., Khoshnevisan, D. and Nualart, \'E. (2007). Hitting probabilities for systems of non-linear stochastic heat equations with additive noise. {\it ALEA} {\bf 3}  231--271.


\bibitem{DKN09}
Dalang, R. C., Khoshnevisan, D. and Nualart, E. (2009).
Hitting probabilities for systems for non-linear stochastic heat equations with multiplicative noise.
{\it Probab. Theory Related Fields}   {\bf144}  no. 3-4, 371–427.


\bibitem{DP1}
Dalang, R. C. and  Pu, F. (2020).
On the density of the supremum of the solution to the linear stochastic heat equation.
{\it Stoch. Partial Differ. Equ. Anal. Comput.} {\bf 8}  no. 3, 461–508.



\bibitem{dss}
Dalang, R. C. and Sanz-Sol\'e, M. (2024). {\it Stochastic Partial Differential Equations, Space-time White Noise and Random Fields.} arXiv:2402.02119v5

\bibitem{fn}
Florit, C. and  Nualart, D. (1995).
A local criterion for smoothness of densities and application to the supremum of the Brownian sheet.
{\it Statist. Probab. Lett.}  {\bf 22} no. 1, 25–31.

\bibitem{FKM15} 
 Foondun, M., Khoshnevisan, D. and Mahboubi, P.:
Analysis of the gradient of the solution to a stochastic heat equation via fractional Brownian motion.
{\it Stoch. Partial Differ. Equ. Anal. Comput.}   {\bf 3} (2015), no. 2, 133--158.


\bibitem{HNV20}
Huang, J.,  Nualart, D. and Viitasaari, L. (2020)
A central limit theorem for the stochastic heat equation.
{\it Stochastic Process. Appl.}   {\bf 130} no. 12, 7170–7184.


 \bibitem{HP15}
 Hairer, M. and Pardoux, \'E. (2015).
A Wong-Zakai theorem for stochastic PDEs. {\it J. Math. Soc. Japan}   {\bf 67}  no. 4, 1551--1604.

\bibitem{khosh} Khoshnevisan, D.  (2002). {\it Multiparameter processes}.
An introduction to random fields
Springer Monogr. Math.
Springer-Verlag, New York, 2002. xx+584 pp.

\bibitem{Kho14} Khoshnevisan, D.  (2014). Analysis of Stochastic Partial Differential Equations. Published by the AMS on behalf of CBMS Regional Conference Series
in Mathematics \textbf{119}, 116. Providence RI

\bibitem{khosh_shi1999}Khoshnevisan, D. and Shi, Z. (1999). Brownian sheet and capacity.
{\it Ann. Probab.} \textbf{27}  no. 3, 1135–1159.

 \bibitem{KSXZ13}
Khoshnevisan, D.,  Swanson, J.,  Xiao, Y.  and Zhang, L. (2013). Weak Existence of a Solution to a Differential Equation Driven by a Very Rough fBm.  arXiv:1309.3613

\bibitem{KiP} Kim, J. and Pollard, D. (1990). Cube root asymptotics. \textit{Ann. Statist.} \textbf{18}, 191--219 

\bibitem{KRT} Kruk, I., Russo, F. and Tudor, C.A. Wiener integrals, Malliavin calculus and covariance measure structure.
{\it J. Funct. Anal.} \textbf{249} (2007), no. 1, 92--142

\bibitem{muellertribe} Mueller, C. and Tribe, R. (2002). Hitting properties of a random string.
{\it Electron. J. Probab.} {\bf7}  no. 10, 29 pp.

\bibitem{nualart2006}
Nualart, D. (2006).
{\it The Malliavin calculus and related topics.}
Second edition
Probab. Appl. (N.~Y.)
Springer-Verlag, Berlin, 2006. xiv+382 pp.


\bibitem{FP}
Pu, F. (2018). The stochastic heat equation: hitting probabilities and the probability density function of the supremum via Malliavin calculus. Ph.D. thesis, No. 8695,  \'{E}cole Polytechnique F\'{e}d\'{e}rale de Lausanne 

\bibitem{spdewalsh}
Walsh, J.~B. (1986).
\newblock An introduction to stochastic partial differential equations.
\newblock In {\em \'{E}cole d'\'et\'e de probabilit\'es de {S}aint-{F}lour,
  {XIV}---1984}, volume \textbf{1180}, {\em Lecture Notes in Math.}, pp. 265--439.
  Springer, Berlin.


 \end{thebibliography}
\end{document}